\documentclass[11pt]{amsart}
\newif\ifarxiv
\arxivtrue		% AMS style for arXiv
\usepackage{amsmath}
\usepackage{etoolbox}

\newtoggle{arxiv}
\newtoggle{wide}
\ifarxiv
\settoggle{arxiv}{true}
\settoggle{wide}{true}  %Make margins smaller...
\else
\settoggle{arxiv}{false}
\settoggle{wide}{false}
\fi

\usepackage{enumitem}
\iftoggle{arxiv}{}{
}
\iftoggle{arxiv}
{}
{}

\usepackage{graphicx}
\usepackage{epstopdf}
\usepackage{lmodern}
\usepackage{amsfonts,amssymb}
\usepackage[all]{xy}
\usepackage{hyperref}

\usepackage{xcolor}

\usepackage{verbatim}

\makeatletter
\g@addto@macro\@floatboxreset\centering
\makeatother

\setlist[itemize,1]{label={$\bullet$}}

\newcommand{\defn}[1]{{\textbf{#1}}}

\newcommand{\hspec}{h_\mathrm{spec}^\perp}

\newcommand{\NN}{\mathbb{N}}

\newcommand{\ZZ}{\mathbb{Z}}
\newcommand{\JJ}{\mathbb{J}}

\newcommand{\RR}{\mathbb{R}}
\newcommand{\AAA}{\mathcal{A}}
\newcommand{\BBB}{\mathcal{B}}
\newcommand{\CCC}{\mathcal{C}}
\newcommand{\LLL}{\mathcal{L}}

\newcommand{\EEE}{\mathcal{E}}
\newcommand{\RRR}{\mathcal{R}}
\newcommand{\DDD}{\mathcal{D}}
\newcommand{\FFF}{\mathcal{F}}
\newcommand{\HHH}{\mathcal{H}}

\newcommand{\SSS}{\mathcal{S}}
\newcommand{\MMM}{\mathcal{M}}
\newcommand{\PPP}{\mathcal{P}}

\newcommand{\YYY}{\mathcal{Y}}
\newcommand{\ZZZ}{\mathcal{Z}}
\newcommand{\XXX}{\mathcal{X}}
\newcommand{\tC}{\tilde{\mathcal{C}}}

\newcommand{\GGG}{\mathcal{G}}

\newcommand{\ulim}{\varlimsup}

\newcommand{\eps}{\varepsilon}

\newcommand{\ph}{\varphi}

\newcommand{\abs}[1]{\left\lvert#1\right\rvert}

\newcommand{\zz}{\mathbf{z}}

\newcommand{\yy}{\mathbf{y}}
\renewcommand{\aa}%{\mathbf{a}}
{a}

\newcommand{\Gs}{G^-_{\mathrm{s}}}
\newcommand{\GsM}{G^M_{\mathrm{s}}}
\newcommand{\HsM}{\HHH^M_{\mathrm{s}}}

\newcommand{\hC}{\hat{\CCC}}
\newcommand{\absd}[1]{\abs{#1}_\mathrm{d}}
\newcommand{\Cp}{\CCC^\mathrm{p}}
\newcommand{\Cs}{\CCC^\mathrm{s}}
\newcommand{\Cps}{\CCC^\mathrm{p,s}}
\newcommand{\Ep}{\EEE^\mathrm{p}}
\newcommand{\Es}{\EEE^\mathrm{s}}
\newcommand{\Eps}{\EEE^\mathrm{p,s}}
\newcommand{\Ch}{C_{\mathrm{h}}}
\newcommand{\Cb}{C^\beta}
\newcommand{\MT}{\MMM_T}
\newcommand{\Ms}{\MMM_\sigma}

\iftoggle{arxiv}{
\newtheorem{theorem}{Theorem}[section]
\newtheorem{lemma}[theorem]{Lemma}
\newtheorem{proposition}[theorem]{Proposition}
\newtheorem{corollary}[theorem]{Corollary}
\numberwithin{figure}{section}
\numberwithin{equation}{section}
\theoremstyle{remark}
\newtheorem{remark}[theorem]{Remark}
\newtheorem{example}[theorem]{Example}
\newtheorem{question}[theorem]{Question}
\newtheorem{definition}[theorem]{Definition}
}{
}

\iftoggle{arxiv}{

}{
}

\DeclareMathOperator{\Per}{Per}

\DeclareMathOperator{\Cor}{Cor}
\DeclareMathOperator{\End}{End}

\iftoggle{arxiv}{
\thanks{The author is partially supported by NSF grants DMS-1362838 and DMS-1554794.}
}{
}

\begin{document}

%\frontmatter

\title{Specification and towers in shift spaces}
\iftoggle{arxiv}{
\author{Vaughn Climenhaga}
\address{Department of Mathematics, University of Houston, Houston, Texas 77204}
\email{climenha@math.uh.edu}
\date{\today}
}{
\author{Vaughn Climenhaga}
\institute{Department of Mathematics, University of Houston, \email{climenha@math.uh.edu}}
\maketitle
%\thanks{The author was partially supported by NSF grants DMS}
}

\begin{abstract}
We show that a shift space on a finite alphabet with a non-uniform specification property can be modeled by a strongly positive recurrent countable-state Markov shift to which every equilibrium state lifts.  In addition to uniqueness of the equilibrium state, this gives strong statistical properties including the Bernoulli property, exponential decay of correlations, central limit theorem, and analyticity of pressure, which are new even for uniform specification.  We give applications to shifts of quasi-finite type, synchronised and coded shifts, and factors of $\beta$-shifts and $S$-gap shifts.
\end{abstract}

\iftoggle{arxiv}{\maketitle}{}

\setcounter{tocdepth}{1}
%\tableofcontents

%\mainmatter

\section{Introduction}
\label{sec:introduction}

\subsection{Main results}\label{sec:result}

Let $f\colon X\to X$ be a continuous map on a compact metric space, and let $\ph\colon X\to \RR$ be continuous.  The \defn{topological pressure}  is $P(\ph) := \sup \{h(\mu) + \int \ph\,d\mu\}$, where $h(\mu)$ is Kolmogorov--Sinai entropy and the supremum is taken over 
$f$-invariant Borel probability measures.  A measure achieving the supremum is called an \defn{equilibrium state} for $(X,f,\ph)$.
Existence, uniqueness, and statistical properties of equilibrium states 
have consequences for many areas of dynamics and geometry, such as
distribution of closed geodesics %in negative \cite{rB72} and non-positive curvature \cite{gK98};
\cite{rB72,gK98};
physical measures for smooth maps \cite{dR76}; 
entropy rigidity for geodesic flow \cite{aK82};
large deviations \cite{yK90};
multifractal analysis \cite{BSS02};
the Weil--Petersson metric \cite{cM08};
Teichm\"uller flow \cite{BG11};
phase transitions and quasicrystals \cite{BL13}; 
representation theory \cite{BCLS};
and diffusion along periodic surfaces \cite{AHS}.

Here we study symbolic dynamics.  Given a finite set $A$ (the \defn{alphabet}), the set $A^\ZZ$ of bi-infinite sequences is compact in the metric $d(x,y) = e^{-\min \{\abs{k} : x_k\neq y_k\}}$, and the \defn{shift map} $\sigma\colon A^\ZZ\to A^\ZZ$ defined by $(\sigma x)_k = x_{k+1}$ is continuous.  A \defn{two-sided shift space} on $A$ is a closed set $X\subset A^\ZZ$ with $\sigma(X) = X$.  
Replacing $A^\ZZ$ with $A^{\NN\cup \{0\}}$ gives a one-sided shift space; we will work with both one- and two-sided shifts.  A \defn{word} is a finite sequence of symbols $w\in A^* := \bigcup_{n\geq 0} A^n$; we write $\abs{w}$ for the length of $w$.  Given $x\in X$ and integers $i\leq j$, let $x_{[i,j]} := x_i x_{i+1} \cdots x_j$.  The set of words obtained this way is the \defn{language} $\LLL = \LLL(X) = \bigcup_{n\geq 0} \LLL_n$, where $\LLL_n = \{x_{[1,n]} : x\in X \} \subset A^n$ for $n>0$, and $\LLL_0$ is the set containing the empty word.

A shift space $X$ is \defn{topologically transitive} if for every $v,w\in \LLL$ there is $u\in \LLL$ such that $vuw\in \LLL$, and is a \defn{subshift of finite type (SFT)} if there is a finite set $F\subset A^*$ such that $X = \{x\in A^\ZZ : x_{[i,j]} \notin F$ for every $i<j\}$. When $X$ is a topologically transitive SFT, every H\"older continuous potential has a unique equilibrium state $\mu$, and $\mu$ has strong statistical properties (Bernoulli property and exponential decay of correlations up to a period, central limit theorem) \cite{rB75}; moreover, the topological pressure function is analytic \cite{PP90}.  

A weaker criterion for uniqueness is the \defn{specification} condition \cite{rB74},  which does not require the SFT condition but strengthens transitivity by requiring that there is $\tau\in \NN$ such that for every $v,w\in \LLL$ there is $u\in \LLL$ with $\abs{u} \leq \tau$ such that $vuw\in \LLL$.\footnote{There are many versions of specification in the literature; this definition is specialised for the symbolic setting and is slightly weaker than Bowen's original one.}
Uniqueness results using non-uniform specification conditions have been proved by the author and D.J.\ Thompson \cite{CT1,CT2}.  An important tool in these results, and here, is the notion of the \defn{pressure} of a collection of words $\DDD \subset \LLL$:
\begin{equation}\label{eqn:PD}
P(\DDD,\ph) := \ulim_{n\to\infty} \frac 1n \log \sum_{w\in \DDD \cap \LLL_n}  \sup_{\{x\in X : x_{[1,n]} = w\}}
e^{\ph(x) + \ph(\sigma x) + \cdots + \ph(\sigma^{n-1} x)}.
\end{equation}
%When $\DDD=\LLL$, 
The variational principle \cite[Theorem 9.10]{pW82} gives $P(\LLL,\ph) = P(\ph)$.  

Uniqueness is enough to show that the topological pressure function is G\^{a}teaux differentiable -- there are no first-order phase transitions -- but does not immediately imply analyticity, which would rule out phase transitions of all orders.  Similarly, the stronger statistical properties of $\mu$
%Stronger statistical properties such as Bernoullicity, decay of correlations, and the central limit theorem 
do not appear to have been studied using any version of specification.  They are known to hold for systems on which a certain `tower' can be built \cite{lY98,lY99}.  Our main result uses a non-uniform specification condition to establish uniqueness, statistical properties, and analyticity by building a tower.  To formulate the condition, we need the following notion: a word $w'\in \LLL_k$ is a \defn{prefix} of a word $w\in \LLL_n$ if $k\leq n$ and $w'=w_{[1,k]}$.  Similarly, if $w'=w_{[j,\abs{w}]}$ for some $1\leq j\leq\abs{w}$ then $w'$ is a \defn{suffix} of $w$.
See \S\ref{sec:statistical-properties} for full definitions of the various conclusions.

\begin{theorem}\label{thm:main}
Let $X$ be a one- or two-sided shift space on a finite alphabet with language $\LLL$ and let $\ph\colon X\to \RR$ be H\"older continuous. 
Suppose there is $\GGG \subset \LLL$ such that
\begin{enumerate}[leftmargin=*, widest=III, 
label = \textup{\textbf{[\Roman{*}{]}}}]
\item\label{spec}
there is $\tau\in \NN$  such that
for all $v,w\in \GGG$, there is $u\in \LLL$ with $\abs{u}\leq\tau$ such that $v'uw'\in \GGG$ whenever $v'\in \GGG$ is a suffix of $v$ and $w'\in \GGG$ is a prefix of $w$;
%for all $v,w\in \GGG$, there is $u\in \LLL$ with $\abs{u}\leq \tau$ satisfying $v_{[i,\abs{v}]} u w_{[1,j]} \in \GGG$ for every $1\leq i\leq \abs{v}$ and $1\leq j\leq\abs{w}$ with $v_{[i,\abs{v}]}, w_{[1,j]} \in \GGG$;
\item\label{gap}
there are $\Cp,\Cs\subset \LLL$ such that $P(\Cp \cup \Cs \cup (\LLL \setminus \Cp \GGG \Cs),\ph) < P(\ph)$;
\item\label{stay-good}
there is $L\in \NN$ such that if $u,v,w\in \LLL$ have $\abs{v}\geq L$, $uvw\in \LLL$, and $uv,vw\in \GGG$, then $v, uvw\in \GGG$.
\end{enumerate}
Then the following are true.
\begin{enumerate}[leftmargin=*, widest=iii,
label=\textup{(\roman{*})}]
\item\label{unique} $(X,\ph)$ has a unique equilibrium state $\mu$.
\item\label{gibbs} $\mu$ has the Gibbs property \eqref{eqn:gibbs} for $\ph$ on $\GGG$.
\item\label{periodic} $\mu$ is the limiting distribution of $\ph$-weighted periodic orbits.
\item\label{bernoulli} 
$(X,\sigma,\mu)$ has exponential decay of correlations for H\"older observables up to a period.
\item\label{clt} $(X,\sigma,\mu)$ satisfies the central limit theorem for H\"older observables $\psi$, with variance $0$ if and only if $\psi$ is cohomologous to a constant.
\item\label{analytic} Given any H\"older continuous $\psi\colon X\to \RR$, there is $\eps>0$ such that the topological pressure function $t\mapsto P(\ph + t\psi)$ is real analytic on $(-\eps,\eps)$.
\end{enumerate}
When $X$ is two-sided, $(X,\sigma,\mu)$ has the Bernoulli property up to a  period.
\end{theorem}

When $\GGG=\LLL$, \ref{spec} is the classical specification property, \ref{gap} holds with $\Cp=\Cs=\emptyset$, and \ref{stay-good} is immediate.  Conclusions \ref{unique}--\ref{periodic} are well-known in this case \cite{rB74}, but \ref{bernoulli}--\ref{analytic} are new.
When $\GGG\neq \LLL$,  \ref{spec}--\ref{stay-good} should be interpreted as 
\defn{non-uniform specification} conditions.\footnote{The term ``non-uniform specification'' is to be understood informally; it is not clear whether these conditions imply the ones in \cite{CT1,CT2}, or vice versa.}
The idea is that $\GGG$ is a collection of \defn{`good'} words for which specification holds, so $\Cp$, $\Cs$, and $\LLL\setminus\Cp\GGG\Cs$ contain all \defn{`obstructions to specification'}: every $w\in \Cp\GGG\Cs$ admits a decomposition $w=u^p v u^s$ with $v\in \GGG$ and $u^{p,s} \in \CCC^{\mathrm{p,s}}$, so $w$ can be made good by removing the prefix $u^p$ and the suffix $u^s$.  Condition \ref{gap} requires that the obstructions have small pressure and hence are `invisible' to equilibrium states.
%while $\Cp$ and $\Cs$ are collections of \defn{prefixes} and \defn{suffixes}; together with $\LLL \setminus \Cp \GGG \Cs$ (the  words with no \defn{decomposition}), $\Cp$ and $\Cs$ give a list of \defn{`obstructions to specification'}, which \ref{gap} requires to have small pressure.  

Theorem \ref{thm:main} is an immediate corollary of the following two results.  The first of these is a structure theorem relating $X$ to a countable-state Markov shift $\Sigma$; see \S\ref{sec:shift-spaces} for complete definitions.  The second uses the fact that \defn{strong positive recurrence} implies existence of a unique \defn{Ruelle--Perron--Frobenius} (RPF) measure $m$ on $\Sigma$ for the potential function $\ph\circ \pi\colon \Sigma\to \RR$, which has strong statistical properties \cite{CS09}.% and can be interpreted as the unique equilibrium state for the potential function $\ph\circ \pi\colon \Sigma\to \RR$.

%see Theorem \ref{thm:spr}.% for a precise statement of this fact.

\begin{theorem}\label{thm:structure}
Let $X$ be a one- or two-sided shift space on a finite alphabet and let $\ph\colon X\to \RR$ be H\"older.  If the language $\LLL$ of $X$ has a subset $\GGG$ satisfying the non-uniform specification conditions \ref{spec}--\ref{stay-good}, then there exists an equilibrium state $\mu$ for $(X,\ph)$ that satisfies the Gibbs property \eqref{eqn:Gibbs} for $\ph$ on $\GGG$.\footnote{This is in fact the unique equilibrium state, but the proof of uniqueness waits until Theorem \ref{thm:spr-a}.}  Moreover, there is a topologically transitive
countable-state Markov shift $\Sigma$ and a 1-block code $\pi\colon \Sigma\to X$ such that
\begin{enumerate}[leftmargin=*, widest=a, 
label=\textup{(\alph{*})}]
\item\label{spr} $\ph\circ \pi$ is strongly positive recurrent on $\Sigma$; and 
\item\label{lifts} 
there is $P' < P(\ph)$ such that for every ergodic Borel probability measure $\mu$ on $X$ with $h(\mu) + \int \ph\,d\mu > P'$, there is a shift-invariant Borel probability measure $\nu$ on $\Sigma$ such that $\mu = \pi_*\nu$ and $h(\mu) = h(\nu)$.
\end{enumerate}
In the case when $X$ is a two-sided shift space, the map $\pi$ is injective.\footnote{For $\ph=0$, this and \ref{lifts} imply that $X$ and $\Sigma$ are \defn{$h$-isomorphic} in the sense of \cite{jB97}.}
In both the one- and two-sided cases, the following weaker property holds:
\begin{equation}\label{eqn:inj}
\text{if $\zz,\zz'\in \Sigma$ have $\pi(\zz)=\pi(\zz')$, then 
$\zz_{[n,\infty)} = \zz'_{[n,\infty)}$ for some $n\in \NN$.
}
\end{equation}
\end{theorem}

Injectivity generally fails in the one-sided case; see
Example \ref{eg:1-sided-inj}.

\begin{theorem}\label{thm:spr-a}
Let $X$ be a one- or two-sided shift space on a finite alphabet and let $\ph\colon X\to \RR$ be H\"older.
If there is a topologically transitive
countable-state Markov shift $\Sigma$ and a 1-block code $\pi\colon \Sigma\to X$ satisfying \ref{spr}, \ref{lifts}, and \eqref{eqn:inj}, then $(X,\ph)$ has a unique equilibrium state $\mu$, which satisfies conclusions \ref{periodic}--\ref{analytic} of Theorem \ref{thm:main}, and the period $d$ in \ref{bernoulli} is given by the gcd of the lengths of periodic orbits in $\Sigma$.  
In the case when $X$ is two-sided, $(X,\sigma,\mu)$ is Bernoulli up to the period $d$.
\end{theorem}

%When $\ph=0$ and $X$ is two-sided, Theorem \ref{thm:structure}\ref{lifts} and the fact that $\pi$ is everywhere injective imply that $X$ and $\Sigma$ are \defn{$h$-isomorphic} in the sense of \cite{jB97}.  
%Injectivity fails for one-sided shifts; see \eqref{eqn:inj} for a related property that holds in this case.

See \S\ref{sec:subsidiary-results} for a road map of the proofs.
Theorem \ref{thm:spr-a} is a relatively routine consequence of well-known results about thermodynamic formalism for countable-state Markov shifts; see Theorem \ref{thm:spr}.  Most of the work in this paper is devoted to proving Theorem \ref{thm:structure}, whose starting point is an argument of Bertrand  for constructing a synchronising word using specification \cite{aB88}.  Although \ref{spec}--\ref{stay-good} do not imply that $X$ is synchronised, we can still `synchronise good words', produce a collection $\FFF \subset \LLL$ of words that can be \defn{`freely concatenated'} (Theorem \ref{thm:get0spec}), and use $\FFF$ to describe $\Sigma$ (Theorem \ref{thm:tower}).

\subsubsection*{Structure of the paper}
In \S\S\ref{sec:motivating}--\ref{sec:hyperbolic} we describe: motivating questions behind Theorem \ref{thm:main}; a method for verifying \ref{spec}--\ref{stay-good} that passes to factors; and applications to a number of classes of systems, including shifts of quasi-finite type, synchronised shifts, and coded shifts.
Further definitions and background are given in \S\ref{sec:defs}.  In \S\ref{sec:subsidiary-results} we formulate some intermediate results that imply Theorems \ref{thm:main}--\ref{thm:spr-a}.
In \S\ref{sec:prep} we give preparatory results for the proofs, including  mild strengthenings of the Birkhoff and Shannon--McMillan--Breiman theorems that hold quite generally, not just in the setting of this paper.
%Theorems \ref{thm:get0spec}, \ref{thm:tower}, and \ref{thm:spr} 
The results from \S\ref{sec:subsidiary-results} are proved in  \S\S\ref{sec:towers2}--\ref{sec:get0spec}. 
In \S\ref{sec:other-proofs} we prove the remaining results stated in the introduction, including the applications.
We will have occasion to use various conditions that are related to \ref{spec}, \ref{gap}, and \ref{stay-good}.  For the reader's convenience, we gather these in Appendix \ref{sec:conditions}.

%a separate appendix. %following \S\ref{sec:other-proofs} 
%so that they can be more easily referenced.

%\introsec{Description of approach}

%See \S\ref{sec:subsidiary-results} for a map of the proof of Theorem \ref{thm:structure}.  
%The starting point for the proof of Theorem \ref{thm:structure} (and hence of Theorem \ref{thm:main}) is an argument of Bertrand for constructing a synchronising word using specification \cite{aB88}.  Although \ref{spec}--\ref{stay-good} do not imply that $X$ is synchronised, we can still use \ref{spec} to `synchronise good words' and produce a collection $\FFF \subset \LLL$ of words that can be \defn{`freely concatenated'} (Theorem \ref{thm:get0spec}), and use this collection to describe $\Sigma$ (Theorem \ref{thm:tower}).  Then conclusions \ref{unique}--\ref{analytic} follow from existing results in the literature, gathered in Theorems \ref{thm:tower}\ref{gap-gibbs} and \ref{thm:spr}.

\subsection{Motivating questions}\label{sec:motivating}

The  non-uniform specification property introduced in  \cite{CT1} was motivated by $\beta$-shifts and their factors.  Given $\beta>1$ and $A = \{0,\dots, \lceil\beta\rceil - 1\}$, the greedy $\beta$-expansion of $1$ is the lexicographically maximal $\zz\in A^\NN$ satisfying $1 = \sum_{k=1}^\infty \zz_k \beta^{-k}$, and the $\beta$-shift $\Sigma_\beta$ is the subshift of $A^\NN$ defined by the condition that $x\in \Sigma_\beta$ if and only if $x_{[k,\infty)} \preceq \zz$ for all $k\in \NN$, where $\preceq$ is the lexicographic order.   The $\beta$-shift is the natural coding space for the map $x\mapsto \beta x \pmod 1$ on the unit interval.

For Lebesgue-a.e.\ $\beta>1$, the $\beta$-shift does not have specification \cite{jS97}; however, it can be described in terms of a countable-state Markov shift, which was used by  Hofbauer to prove uniqueness of the equilibrium state for $\ph=0$, called the measure of maximal entropy (MME) \cite{fH78}.  The corresponding result for subshift factors of $\beta$-shifts remained open for some time \cite[Problem 28.1]{mB08}, which led the author and D.J.\ Thompson to introduce a set of non-uniform specification conditions that are satisfied by the $\beta$-shifts, that pass to factors, and that guarantee uniqueness of the MME.
%conditions \ref{spec'-a}, \ref{gap2}, and \ref{extends} in \cite{CT1} as a mechanism for uniqueness that passes to factors.  
This raised the following questions.
\begin{enumerate}[leftmargin=*, widest=1]
\item Are there examples of systems with non-uniform specification that do \emph{not} come from a countable-state Markov shift?  
%Very loosely, does the machinery of non-uniform specification apply in a broader setting than the machinery of Young towers?
\item Hofbauer's graph structure for the $\beta$-shifts (see also \cite{pW78} for equilibrium states for $\ph\neq 0$) can be used to get the stronger conclusions \ref{bernoulli}--\ref{analytic} that do not follow from \cite{CT1}.  Can these conclusions be obtained using non-uniform specification? 
\end{enumerate}
These questions motivated this paper, which says that the answers are ``no'' and ``yes'', respectively: 
non-uniform specification in the sense of \ref{spec}--\ref{stay-good} 
leads to a countable-state Markov structure with strong positive recurrence,
so its statistical consequences are just as strong.\footnote{It should be emphasised that although \ref{spec}--\ref{stay-good} are similar in spirit to the conditions in \cite{CT1}, and apply to all the examples studied so far using those conditions, there does not appear to be a logical relationship between the two sets of conditions in the sense of one implying the other.}
This can be interpreted as a negative result in the sense that 
every system to which Theorem \ref{thm:main} applies could also have been studied by other techniques, by building a tower or using a countable-state Markov shift.
%we should not expect Theorem \ref{thm:main} to cover new classes of examples, since every shift satisfying \ref{spec}--\ref{stay-good} could also be described in terms of a countable graph.  This is discussed in \S\ref{sec:coded}.  
On the other hand, if the system is defined in a manner that does not make this Markov structure explicit, then it may be difficult to find the graph that does the job, or to determine its properties; see \S\ref{sec:qft-intro} for examples.%(consider QFTs or cocyclic subshifts).  

One can interpret the Markov shift $\Sigma$ in Theorem \ref{thm:structure} as a Young tower as in \cite{lY99} by passing to its one-sided version $\Sigma^+$ and inducing on a single state to get a full shift.  Condition \ref{spr} on strong positive recurrence implies that the tower has \defn{`exponential tails'}, and \ref{lifts} guarantees that every equilibrium state is \defn{liftable} to the tower; see \cite{PSZ} for further discussion of this approach, and \cite{fH79,gK89,rZ04} for the original investigations of liftability.  The effort we expend to prove \ref{spr} and \ref{lifts} illustrates a general theme: \emph{even when it is  clear how to build a tower, it is usually a non-trivial problem to verify that equilibrium states lift to the tower and that the tower's tails decay exponentially}; see \cite{jK04,jB05,kT06} for symbolic examples, and  %\cite{gK89,PZ07,PSZ08,PS08,BT09,IT10} 
\cite{gK89,PS08,BT09,IT10}
for smooth examples.
One goal of the present approach is to give a set of more readily verifiable conditions that can establish liftability and exponential tails.  This will be particularly valuable if it can be extended to the smooth setting.  The non-uniform specification properties from \cite{CT1,CT2} have been extended and applied to various smooth systems
%, where along with a non-uniform expansivity property they once again yield uniqueness 
\cite{CT4,CFT,CFT2,BCFT}, such as geodesic flows over rank 1 manifolds of nonpositive curvature.  It is expected that the results given here will admit a similar generalisation.

%%%% 1

\subsection{Obstructions to specification}

Since subshift factors of $\beta$-shifts were the original examples studied in \cite{CT1}, it is natural to ask whether Theorems \ref{thm:main}--\ref{thm:spr-a} apply to them as well.  Unlike the non-uniform specification conditions in \cite{CT1}, condition \ref{stay-good} does not pass directly to factors.  Here we give conditions that do pass to factors (see \S\ref{sec:factors}), and which imply \ref{spec}--\ref{stay-good}.  These will also be useful for some of our applications (see \S\ref{sec:qft-intro}).

Given $\CCC^+,\CCC^-\subset A^*$ and $M\in \NN$, consider the collection%\footnote{Note that this is not the collection $\GGG^M$ that appears in \cite{CT1,CT2}; see \S\ref{sec:connection-to-previous-results}.}
\begin{equation}\label{eqn:GMC}
\GGG(\CCC^\pm,M) := \{w\in \LLL \mid w_{[1,i]}\notin \CCC^-, w_{(\abs{w}-i,\abs{w}]}\notin \CCC^+ 
\text{ for all } M\leq i\leq \abs{w}\}
\end{equation}
of all words that do not start with a long element of $\CCC^-$ or end with a long element of $\CCC^+$.  
Say that $\CCC^\pm$ is a \defn{complete list of obstructions to specification} if
\begin{enumerate}[leftmargin=*, widest=I, 
label = \textup{\textbf{[\Roman{*}$^*${]}}}]
\item\label{spec*}
for every $M\in \NN$ there is $\tau=\tau(M)$ such that for all $v,w\in \GGG(\CCC^\pm,M)$ there is $u\in \LLL$ with $\abs{u}\leq \tau$ such that $vuw\in \LLL$.
\end{enumerate}
The word $vuw$ need not be in $\GGG(\CCC^\pm,M)$ (cf.\ \ref{spec}), so enlarging $\CCC^\pm$ cannot cause \ref{spec*} to fail; in particular, every $\CCC^\pm$ that contains a complete list of obstructions is itself a complete list. The following is used in \S\ref{sec:factors}  to study factors, and is proved in \S\ref{sec:CGC}.

\begin{theorem}\label{thm:CGC}
Let $X$ be a one- or two-sided shift space with language $\LLL$, and let $\ph\colon X\to \RR$ be H\"older.  Suppose $\CCC^\pm \subset \LLL$ is a complete list of obstructions to specification admitting the pressure bound $P(\CCC^- \cup \CCC^+,\ph) < P(\ph)$ and satisfying the following condition:
\begin{equation}\label{eqn:persistent2}
(vw\in \CCC^+ \Rightarrow v\in \CCC^+) \text{\quad and \quad}
(vw\in \CCC^- \Rightarrow w\in \CCC^-).
\end{equation}
Then there is $\GGG\subset \LLL$ satisfying \ref{spec}--\ref{stay-good},
so Theorems \ref{thm:main}--\ref{thm:spr-a} apply to $(X,\ph)$.
\end{theorem}

\begin{example}\label{eg:beta}
For the $\beta$-shift, let $\CCC^- = \emptyset$ and $\CCC^+ = \{ \zz_{[1,n]}\}_{n\in \NN}$, where $\zz$ is the $\beta$-expansion of $1$.    In the standard graph presentation, $\GGG(\CCC^\pm,M)$ is the set of words that label paths starting at the base vertex and ending in the first $M$ vertices, so it satisfies \ref{spec*} as shown in \cite{CT1}.  Thus $\CCC^\pm$ is a complete list of obstructions to specification.  
Clearly $\CCC^\pm$ satisfy \eqref{eqn:persistent2}, and it is shown in \cite{CT2} that $P(\CCC^+,\ph) < P(\ph)$ for every H\"older $\ph$.
\end{example}

\subsection{Shifts of quasi-finite type}\label{sec:qft-intro}

We apply Theorem \ref{thm:CGC} to the shifts of quasi-finite type introduced by Buzzi in \cite{jB05}.  
Say that $w\in \LLL$ is a \defn{left constraint} if there is $v\in \LLL$ such that $w_{[2,\abs{w}]}v\in \LLL$ but $wv\notin \LLL$; let
$\CCC^\ell$ be the collection of left constraints, and $\CCC^r$ the collection of analogously defined \defn{right constraints}.
The shift space $X$ is a \defn{shift of quasi-finite type} (QFT) if %at least one list of constraints has small entropy: 
$\min\{ h(\CCC^\ell), h(\CCC^r) \} < h(\LLL)$.
Topologically mixing QFTs may have multiple MMEs \cite[Lemma 4]{jB05}.  The following result, proved in \S\ref{sec:qft}, gives uniqueness if \emph{both} lists of constraints have small entropy, or if we have a stronger mixing condition.

%we will prove that the lists of constraints form a complete list of obstructions to specification, and thus Theorem \ref{thm:CGC} gives uniqueness if \emph{both} lists of constraints have small entropy; the same conclusion is true for any QFT satisfying a stronger mixing condition.

\begin{theorem}\label{thm:QFT}
Let $X$ be a one- or two-sided shift space on a finite alphabet with language $\LLL$ and let $\ph\colon X\to\RR$ be H\"older continuous.
\begin{enumerate}[leftmargin=*, widest=1, label=\textup{(\arabic{*})}]
\item\label{SQFT} If $X$ is topologically transitive, then  $\CCC^- = \CCC^r$ and $\CCC^+ = \CCC^\ell$ form a complete list of obstructions to specification and satisfy \eqref{eqn:persistent2}.  In particular, if $P(\CCC^\ell \cup \CCC^r, \ph) < P(\ph)$, then Theorem \ref{thm:CGC} applies.
\item\label{exact-QFT} Suppose $X^+ = \{x_1 x_2 \cdots \mid x\in X\}$ is \defn{topologically exact}: for every $w\in \LLL$ there is $N\in \NN$ such that $\sigma^N(\{x\in X^+ \mid x_{[1,\abs{w}]} = w\}) = X^+$.  Then $\CCC^- = \emptyset$ and $\CCC^+ = \CCC^\ell$ form a complete list of obstructions to specification and satisfy \eqref{eqn:persistent2}.  In particular, if $P(\CCC^\ell,\ph) < P(\ph)$ then Theorem \ref{thm:CGC} applies.
\end{enumerate}
\end{theorem}

Part \ref{exact-QFT} of Theorem \ref{thm:QFT} applies to many of the piecewise affine transformations studied by Buzzi in \cite{jB97}; these are maps on $[0,1]^d$ given by $f(x) = Ax + b \pmod {\ZZ^d}$, where $A\colon \RR^d\to \RR^d$ is an expanding linear map and $b\in \RR^d$.  Buzzi showed that the natural symbolic codings of such maps are QFTs with $h(\CCC^\ell) < h(\LLL)$, and that they are topologically exact if either (i) all eigenvalues of $A$ exceed $1+\sqrt{d}$ in absolute value, or (ii) $A,b$ have all integer entries.\footnote{Proposition 1 of \cite{jB97} only states that such $f$ are topologically mixing, but the proof in \cite[\S5.1, Lemma 5]{jB97} gives topological exactness.}  He proved that these examples have unique MMEs; Theorem \ref{thm:QFT}\ref{exact-QFT} gives another proof of this result, and lets it be extended to factors (see \S\ref{sec:factors}) and to a class of non-zero potential functions.  See also \cite{dT11} for further results on piecewise expanding maps.

\subsection{Synchronised shifts and uniform specification}

A shift $X$ is \defn{synchronised} if there is $s\in \LLL$ such that $vs,sw\in \LLL$ implies $vsw\in \LLL$; then $s$ is a \defn{synchronising word}.  Shifts with specification are synchronised \cite{aB88}, but not vice versa.  Note that $\beta$-shifts satisfy the non-uniform specification conditions \ref{spec}--\ref{stay-good} but are not all synchronised \cite{jS97}.

Synchronised shifts can have multiple equilibrium states even if they are topologically transitive.  
Every synchronised shift has a canonical presentation via a countable graph (the \defn{Fischer cover}), and Thomsen \cite{kT06} proved that the corresponding countable-state Markov shift is strongly positive recurrent (for the zero potential) if $h(\partial X) < h(X)$, where $\partial X$ is the \defn{derived shift} consisting of all $x\in X$ that can be approximated by periodic points of $X$ and do not contain any synchronising words.
The following gives a similar result for nonzero potentials (although the countable graph provided by Theorem \ref{thm:structure} need not be the Fischer cover); see \S\ref{sec:sync} for a proof.

\begin{theorem}\label{thm:sync}
Let $(X,\sigma)$ be a one- or two-sided topologically transitive shift with a synchronising word $s\in \LLL$.  Let $Y := \{x\in X \mid s \text{ does not appear in } x\}$.  If $\ph\colon X\to\RR$ is H\"older continuous and $P(Y,\ph) < P(X,\ph)$, then the conclusions of Theorems \ref{thm:main}--\ref{thm:spr-a} hold.% and \ref{thm:structure} hold. %conclusions \ref{unique}--\ref{analytic} hold 
%for $(X,\sigma,\ph)$.
\end{theorem}

Shifts with the classical specification property are well-known to be synchronised \cite{aB88} and
to satisfy conclusions \ref{unique}--\ref{periodic} for every H\"older $\ph$ \cite{rB74}.  In particular, they
have a unique MME, which is fully supported, leading to the conclusion that $h(\partial X) < h(X)$.  Thus Thomsen's results could be used to establish conclusions \ref{bernoulli}--\ref{analytic} for the unique MME, and it seems likely that a similar approach would work for nonzero $\ph$, but it does not appear that this procedure has been carried out in the literature.

The proof of conclusions \ref{bernoulli}--\ref{analytic} in Theorem \ref{thm:main} relies on \cite{CS09}, which uses a Perron--Frobenius theorem and a spectral gap for the appropriate transfer operator associated to $\Sigma$.  
Ruelle \cite{dR92} proved a Perron--Frobenius theorem using specification directly, but did not establish spectral gap or a rate of convergence.  One class of shifts with specification for which the transfer operator (for $\ph = 0$) has been studied explicitly are the cocyclic subshifts introduced by Kwapisz \cite{jK00,jK04}; a \defn{cocyclic subshift} over a finite alphabet $A$ is defined by fixing a finite dimensional vector space $V$ and linear transformations $\Phi_i \in \End(V)$ for $i\in A$, then 
putting $\LLL = \{w\in A^* : 
\Phi_{w_1} \cdots \Phi_{w_{\abs{w}}} \neq 0\}$.  Transitive cocyclic subshifts often fail to be SFTs (or even sofic), but have specification and hence satisfy conclusions \ref{unique}--\ref{periodic} of Theorem \ref{thm:main} by \cite{rB74}.  For the zero potential, spectral properties of the transfer operator were studied in \cite{jK04}, although conclusions \ref{bernoulli}--\ref{analytic} were not discussed there.  Theorem \ref{thm:main} establishes these conclusions for H\"older potentials.

\subsection{Coded shifts}

Given a finite alphabet $A$ and a set of words $G \subset \bigcup_{n\in \NN} A^n$, let
\[
G^\infty := \{x\in A^\ZZ \mid \text{there is } (n_k)_{k\in \ZZ} \subset \ZZ \text{ with } n_k < n_{k+1} 
\text{ and } x_{[n_k,n_{k+1})} \in G \text{ for all } k\}.
\]
A two-sided shift space $X\subset A^\ZZ$ is called \defn{coded} if $X=\overline{G^\infty}$ for some $G$.
%
%\footnote{Equivalently, $X$ is the closure of a uniformly continuous image of a topologically transitive countable-state Markov shift \cite{FF92}.}  
Such a $G$ is a \defn{generating set} for $X$.
There is a sense in which every question about equilibrium states for a shift satisfying \ref{spec}--\ref{stay-good} reduces to a question about coded shifts.  If $X$ satisfies \ref{spec}--\ref{stay-good} and $\pi\colon \Sigma\to X$ is the coding map from a countable-state Markov shift provided by Theorem \ref{thm:structure}, then $\overline{\pi(\Sigma)} \subset X$ is a coded subshift of $X$, and every equilibrium state for $\ph$ gives full weight to this coded shift.

A generating set $G$ is \defn{uniquely decipherable} if whenever $u^1 u^2 \cdots u^m = v^1 v^2 \cdots v^n$ with $u^i, v^j\in G$, we have $m=n$ and $u^j=v^j$ for all $j$
\cite[Definition 8.1.21]{LM95}.
The condition of unique decipherability for $G$ is related to \ref{stay-good} for $G^*$, but does not imply it.  Thus the following is not a direct consequence of Theorems \ref{thm:main}--\ref{thm:spr-a};
% and \ref{thm:structure}; 
%we prove it in \S\ref{sec:towers2}.
%see \S\ref{sec:subsidiary-results} for its proof.
it follows from Corollary \ref{cor:coded} and Theorem \ref{thm:spr} below.

\begin{theorem}\label{thm:coded}
Let $X$ be a coded shift on a finite alphabet and $\ph$ a H\"older potential on $X$.  If $X$ has a uniquely decipherable generating set $G$ such that $\DDD = \DDD(G) := \{w\in \LLL \mid w $ is a subword of some $g\in G\}$ satisfies $P(\DDD,\ph) < P(\ph)$, then $(X,\ph)$ satisfies conclusions \ref{spr}--\ref{lifts} of Theorem \ref{thm:structure} and conclusions \ref{unique}--\ref{analytic} of Theorem \ref{thm:main}, and $\mu$ is Bernoulli up to a period.  The coding map $\pi$ may not be injective, but it is still finite-to-one $\mu$-a.e.\ for the unique equilibrium state $\mu$, and is injective on $\{\zz\in \Sigma : \zz$ is periodic, $\zz_0=a\}$ for each state $a$ of $\Sigma$.
%and $\mu$ is Bernoulli up to a period.
\end{theorem}

It is natural to ask what happens if one removes the condition that $G$ be uniquely decipherable.  In this case one can still define a natural countable-state Markov shift $\Sigma$ and a coding map $\pi\colon \Sigma\to X$ such that $X = \overline{\pi(\Sigma)}$, but now we lose control of the multiplicity of $\pi$, and indeed, $\pi$ may decrease entropy; see Example \ref{eg:not1-1}.

It is shown in \cite[Proposition 2.1]{BH86} that every coded shift admits a uniquely decipherable generating set; similarly, \cite[Theorem 1.7]{FF92} shows that it is always possible to build a \defn{`bi-resolving'} $\Sigma$ for which $\pi$ is injective.  However, in both cases one must abandon the original generating set $G$ and pass to a new generating set $G'$, for which we may a priori have $P(\DDD(G'),\ph) = P(\ph)$, so Theorem \ref{thm:coded} may not apply.  On the other hand, the results in \cite{CT2} can be used to show that conclusions \ref{unique}--\ref{periodic} from Theorem \ref{thm:main} hold.  Thus we have the following open question.

\begin{question}\label{q:coded}
Let $X$ be a coded shift with language $\LLL$ and a (not uniquely decipherable) generating set $G\subset \LLL$ such that $P(\DDD(G),\ph) < P(\ph)$ for some H\"older $\ph$.  Let $\mu$ be the unique equilibrium state for $(X,\ph)$.   Do conclusions \ref{bernoulli}--\ref{analytic} of Theorem \ref{thm:main} still hold?  
%That is, is some iterate of $(X,\sigma,\mu)$ Bernoulli with exponential decay of correlations; does $(X,\sigma,\mu)$ satisfy the central limit theorem; and is the pressure function real analytic at $\ph$?
\end{question}

\subsection{Factors}\label{sec:factors}

Before stating our results for factors we describe 
another motivating example from \cite{CT1} to which Theorem \ref{thm:CGC} applies; the $S$-gap shifts.

\begin{example}\label{eg:S-gap}
Given  $S\subset \NN \cup \{0\}$, the $S$-gap shift $X_S$ is the coded system with generating set $\{10^n \mid n\in S\}$.  
%For finite $S$ this is sofic; for infinite $S$, 
Take $\CCC^+ = \CCC^- = \{0^k \mid k\in \NN\}$, so $\CCC^\pm$ satisfy \eqref{eqn:persistent2}, and  $\GGG(\CCC^\pm,M) = \{0^a 1 w 1 0^b\in \LLL \mid a,b< M\}$.  Taking $\tau(M) = \min\{s\in S \mid s\geq 2M\}$, we see that any two words in $\GGG(\CCC^\pm,M)$ can be joined by $0^c$ for some $0\leq c\leq \tau(M)$, and so \ref{spec*} holds.  Thus $\CCC^\pm$ is a complete list of obstructions to specification, and one can show that $P(\CCC^+ \cup \CCC^-,\ph) < P(\ph)$ for every H\"older  $\ph$ \cite[\S5.1.3]{CTY}.
\end{example}

%Restricting our attention to the case 
When $\ph=0$, we define the \defn{entropy of obstructions to specification} as\footnote{A similar but distinct quantity was defined in \cite{CT3}.}
\begin{equation}\label{eqn:hspec}
\hspec(X) = \inf\{ h(\CCC^+ \cup \CCC^-) \mid \CCC^\pm\subset \LLL(X) \text{ satisfy } \eqref{eqn:persistent2} \text{ and } \ref{spec*} \}.
\end{equation}
Note that $\beta$-shifts and $S$-gap shifts both have $\hspec(X) = 0$.  Although the most obvious way to get $\hspec(X)=0$ is to have $h(\CCC^+ \cup \CCC^-)=0$ for some $\CCC^\pm$, we expect that there are examples where $\hspec(X)=0$ but the infimum is not achieved; a natural class of candidates is given by shift spaces coding transitive piecewise monotonic transformations of the interval, whose structure has been described by Hofbauer \cite{fH79,fH81}.

It follows from Theorem \ref{thm:CGC} that $\hspec(X) < h(X)$ implies existence of a unique MME together with the other conclusions of Theorem \ref{thm:main}.  In \S\ref{sec:passing-to-factors}, we prove that $\hspec$ is non-increasing under passing to factors, and obtain the following results.\footnote{Compare Theorem \ref{thm:factors} to \cite[Corollary 2.3 and Theorem D]{CT1}; the conclusions here are stronger, and the hypotheses have some similarities but are independent.  Corollary \ref{cor:beta-S} strengthens \cite[Theorem A]{CT1}.}

\begin{theorem}\label{thm:factors}
Let $(X,\sigma)$ be a one- or two-sided shift space on a finite alphabet.
\begin{enumerate}[leftmargin=*, widest=1, label=\textup{(\arabic{*})}]
\item
Let $(\tilde X, \tilde\sigma)$ be a subshift factor of $(X,\sigma)$ such that $ h(\tilde X) > \hspec(X)$.  Then the language of $\tilde X$ contains a collection of words satisfying \ref{spec}--\ref{stay-good} for $\ph=0$, so $(\tilde X,\tilde \sigma,0)$ satisfies the conclusions of Theorems \ref{thm:main}--\ref{thm:spr-a}.
\item
Suppose $X$ satisfies $\gcd\{k\in\NN \mid \sigma^k(x) = x$ for some $x\in X\} = 1$ and $\hspec(X) = 0$, and that
there is $\GGG\subset \LLL$ satisfying \ref{spec} 
such that 
$\LLL w \LLL \cap \GGG\neq\emptyset$ for every $w\in \LLL$.\footnote{$\GGG$ need not satisfy \ref{gap} or \ref{stay-good}; in particular, it need not be the collection from Theorem \ref{thm:CGC}.}
Then every subshift factor of $(X,\sigma)$ satisfies the conclusion of the previous part.
\end{enumerate}
\end{theorem}

\begin{corollary}\label{cor:beta-S}
Let $\tilde X$ be a nontrivial subshift factor of a $\beta$-shift or of an $S$-gap shift.  Then $\tilde X$ % has a unique MME $\mu$; moreover, $\mu$ is the limiting distribution of periodic orbits, has the Bernoulli property and exponential decay of correlations up to a period, and satisfies the central limit theorem.  Finally, $t\mapsto P(t\psi)$ is real analytic on a neighbourhood of $0$ for every H\"older $\psi\colon \tilde X\to \RR$.
satisfies the conclusions of Theorems \ref{thm:main}--\ref{thm:spr-a} for $\ph=0$.
\end{corollary}

\subsection{Hyperbolic potentials}\label{sec:hyperbolic}

Theorem \ref{thm:factors} only deals with measures of maximal entropy, so one may ask what can be said about equilibrium states for non-zero potentials on the factors $(\tilde X,\tilde \sigma)$.  Following \cite{IRRL12}, say that  $\ph\colon X\to \RR$ is \defn{hyperbolic} if
\begin{equation}\label{eqn:hyperbolic}
\ulim_{n\to\infty} \sup_{x\in X} \frac 1n S_n \ph(x) < P(\ph).
\end{equation}
If $\ph$ is hyperbolic and $\hspec(X)=0$, then one can easily find $\CCC^\pm$ satisfying \eqref{eqn:persistent2}, \ref{spec*}, and $P(\CCC^- \cup \CCC^+,\ph) < P(\ph)$; in particular, Theorem \ref{thm:CGC} applies as long as $\ph$ is H\"older.

Every H\"older potential is hyperbolic when $X$ has specification \cite[Theorem 6.1]{CFT2}.  Buzzi proved that the same is true for the coding spaces of continuous topologically transitive piecewise monotonic interval maps \cite{jB04}, and conjectured that the result remains true without the assumption that the map is continuous.\footnote{The H\"older condition on the potential is with respect to the coding space, not the interval itself; H\"older on the interval is not enough, as the Manneville--Pomeau example shows.}
The result is known for a broad class of non-uniformly expanding interval maps \cite{LRL14}, and for $\beta$-shifts  \cite[Proposition 3.1]{CT2}  and $S$-gap shifts \cite[(5.1)]{CTY}.  The proofs of this result for $\beta$-shifts and for $S$-gap shifts are very specific to these examples and in particular do not pass to their factors.

In the other direction, if $X$ is the coded shift generated by $\{0^n1^n \mid n\in \NN\}$, then $\ph = t\mathbf{1}_{[0]}$ is not hyperbolic for $\abs{t}$ sufficiently large \cite{sC}.
%On the other hand, there are coded systems for which some H\"older potentials are not hyperbolic: for example, if $X$ is the shift generated by $\{0^n1^n \mid n\in \NN\}$ and $\ph = t\mathbf{1}_{[0]}$ for $\abs{t}$ sufficiently large, then \eqref{eqn:hyperbolic} fails and $\ph$ is not hyperbolic \cite{sC}.

\begin{question}
Is there an axiomatic condition on a shift space $X$, weaker than specification (perhaps some form of non-uniform specification), guaranteeing that every H\"older potential on $X$ is hyperbolic?  Is there such a condition that is preserved under passing to factors?  In particular, does every subshift factor of a $\beta$-shift or an $S$-gap shift have the property that every H\"older potential is hyperbolic?
\end{question}

\begin{remark}
Since this paper was completed, the author and V.\ Cyr have shown that if there is $\GGG\subset \LLL$ satisfying \ref{spec} such that every $w\in \LLL$ can be transformed into a word in $\GGG$ by changing at most $o(\log \abs{w})$ of its symbols, then every H\"older potential on $X$ is hyperbolic \cite{CC}.  This condition passes to factors and settles the above question in the affirmative for subshift factors of $\beta$-shifts.  However, $S$-gap shifts require 
more edits than this result allows, so the question for subshift factors of $S$-gap shifts is still open.
\end{remark}

\iftoggle{arxiv}{
\subsection*{Acknowledgments}
I am grateful to the anonymous referees for many comments that improved the exposition and for spotting errors in earlier versions of the result on factors and of Lemma \ref{lem:reset}; the latter, which was also pointed out to me by Qu Congcong, necessitated a change in the formulation of condition \ref{spec} from previous versions.
I am also grateful to Omri Sarig for clarifying aspects of strong positive recurrence as they appear in \S\ref{sec:shift-spaces}, and to Dominik Kwietniak for introducing me to cocyclic subshifts and \cite{jK00,jK04}.
}{}

\section{Definitions}\label{sec:defs}

\subsection{Shift spaces and thermodynamic formalism}\label{sec:shift-spaces}

As  in  \S\ref{sec:result}, we write $\LLL \subset A^*$ for the language of a shift space $X$ over a finite alphabet $A$.
%We recall basic definitions from symbolic dynamics; see \cite{LM95} for further details. 
When we work with an indexed collection of words, we write  indices as superscripts; thus $w^1, w^2$ represent two different words, while $w_1, w_2$ represent the first and second symbols in the word $w$.
We write $\abs{w}$ for the length of a word $w$; given $1\leq i\leq j\leq \abs{w}$, we write $w_{[i,j]} = w_i\cdots w_j$.  When convenient, we write $w_{(i,j]} = w_{[i+1,j]}$, and similarly for $w_{[i,j)}$ and $w_{(i,j)}$.   We use the same notation for subwords of an infinite sequence $x\in X$, allowing $i=-\infty$ or $j=\infty$.

Juxtaposition denotes concatenation and will be used liberally throughout the paper both for words and for collections of words; for example, given $w\in A^*$, we will have occasion to refer to the following sets, or similar ones:
\begin{align*}
w\LLL \cap \LLL &= \text{the set of all words in $\LLL$ that begin with $w$},\\
\LLL\setminus \LLL w \LLL &= \text{the set of all words in $\LLL$ that do not contain $w$ as a subword}.
\end{align*}
Given a collection $\DDD\subset \LLL$, we write
\[
\DDD_n := \{w\in \DDD \mid \abs{w}=n\},
\qquad
\DDD_{\leq n}:=\{w\in \DDD \mid \abs{w}\leq n\},
\qquad
\DDD_{\geq n}:=\{w\in \DDD \mid \abs{w}\geq n\}.
\]
%\subsubsection{Thermodynamic formalism for shift spaces}\label{sec:thermodynamic-formalism}
We recall the basics of thermodynamic formalism (adapted to the symbolic setting) and fix our notation.  See \cite{pW82,CT1,CT2} for further details.
Given a continuous function $\ph\colon X\to\RR$, called a \defn{potential},  let $S_n\ph(x) = \sum_{j=0}^{n-1} \ph(\sigma^jx)$ and define $\hat\ph\colon \LLL \to \RR$ by
\begin{equation}\label{eqn:Phi}
\hat\ph(w) := \sup_{x\in [w]} S_{\abs{w}}\ph(x),
\end{equation}
where $[w] := \{x\in X \mid x_{[0,\abs{w})}=w\}$ is the \defn{cylinder} defined by $w$.  Given $\DDD\subset \LLL$, let
\begin{equation}\label{eqn:pressure}
\Lambda_n(\DDD,\ph) := \sum_{w\in \DDD_n} e^{\hat\ph(w)}, 
\qquad
P(\DDD,\ph) 
:= \ulim_{n\to\infty} \frac 1n\log \Lambda_n(\DDD,\ph).
\end{equation}
The first quantity is the \defn{partition sum} of $\ph$ on $\DDD$, and the second is the \defn{pressure} of $\ph$ on $\DDD$; note that this agrees with  \eqref{eqn:PD}. It follows from \eqref{eqn:Phi} and \eqref{eqn:pressure} that
\begin{equation}\label{eqn:delta-P}
\abs{P(\DDD,\ph+\psi) - P(\DDD,\ph)} \leq \|\psi\| := \sup\{ \abs{\psi(x)} : x\in X\}
\end{equation}
for all $\DDD\subset \LLL$ and continuous $\ph,\psi$.
We will need the following consequence of \eqref{eqn:pressure}:
\begin{equation}\label{eqn:PCD}
P(\CCC \cup \DDD,\ph) = \max\{P(\CCC,\ph), P(\DDD,\ph)\}
\text{ for every } \CCC,\DDD\subset \LLL.
\end{equation}
When $\DDD = \LLL$  we write $\Lambda_n(\ph) := \Lambda_n(\LLL,\ph)$ and  $P(\ph) := P(\LLL,\ph)$.
In \S\ref{sec:result} we defined $P(\ph)$ as
$\sup\{h(\mu) + \int \ph\,d\mu : \mu \in \Ms(X)\}$, where $\Ms(X)$ is the space of all $\sigma$-invariant Borel probability measures on $X$, and $h(\mu)$ is the measure-theoretic (Kolmogorov--Sinai) entropy.  This agrees with the above definition by the variational principle \cite[Theorem 9.10]{pW82}.  A measure $\mu\in \Ms(X)$ achieving the supremum is an \defn{equilibrium state}.

When $\ph=0$ we write $h(\DDD) := P(\DDD,0)$ for the \defn{entropy} of $\DDD$, and $h(X) := h(\LLL)$ for the \defn{topological entropy} of the shift space $X$; an equilibrium state for $\ph=0$ is called a \defn{measure of maximal entropy}.

Because $\sigma$ is expansive, the map $\mu \mapsto h(\mu) + \int\ph\,d\mu$ is upper semicontinuous on the weak* compact set of all invariant Borel probability measures, so it achieves its supremum.  Thus
equilibrium states exist for all continuous $\ph\colon X\to\RR$.  To get uniqueness and statistical properties, we require a regularity condition on $\ph$.  Given $\beta>0$, let
\[
%\|\ph\| = \sup\{ \abs{\ph(x)} : x\in X \}
%\quad\text{and}\quad
\abs{\ph}_\beta := \sup\{ \abs{\ph(x)-\ph(y)}/d(x,y)^\beta : x,y\in X, x\neq y\}.
\]
Write $\Cb(X) := \{\ph\colon X\to\RR$ continuous $: \abs{\ph}_\beta < \infty\}$.  %By the definition of $d(x,y)$, we have
Note that for every $n\geq 0$, we have
\begin{equation}\label{eqn:Holder}
\abs{\ph(x) - \ph(y)} \leq \abs{\ph}_\beta e^{-\beta n} \text{ whenever } x,y\in X \text{ satisfy } x_k=y_k \text{ for all } \abs{k}\leq n.
\end{equation}
We write $\Ch(X) := \bigcup_{\beta>0} \Cb(X)$ for the set of all H\"older functions.  In addition to the H\"older semi-norm $\abs{\ph}_\beta$ we will use the %supremum norm $\|\ph\| = \sup\{ \abs{\ph(x)} : x\in X \}$, and the 
H\"older norm $\|\ph\|_\beta := \|\ph\| + \abs{\ph}_\beta$.  We note that
\begin{equation}\label{eqn:norm-shift}
\|\ph \circ \sigma\| = \|\ph\|
\quad\text{and}\quad
\abs{\ph\circ\sigma}_\beta \leq e^{\beta} \abs{\ph}_\beta.
\end{equation}

We also work with Markov shifts on an infinite alphabet:
given a directed graph with a countably infinite vertex set $V$, one defines a \defn{countable-state Markov shift} $\Sigma \subset V^\ZZ$ by the condition that $\zz\in V^\ZZ$ is in $\Sigma$ iff the graph contains an edge from $\zz_n$ to $\zz_{n+1}$ for every $n\in \ZZ$.  One-sided shifts are defined similarly.
We write $T\colon \Sigma\to\Sigma$ for the shift map in the countable-state Markov case, to distinguish from the shift $\sigma\colon X\to X$.
All of our countable-state Markov shifts are \defn{topologically transitive}: for any $a,b\in V$ there is a path in the graph that leads from $a$ to $b$.
Given a vertex $a\in V$, let
\begin{equation}\label{eqn:PP*}
\begin{aligned}
\Per_n(\Sigma,a) &:= \{ \zz\in \Sigma : T^n\zz = \zz \text{ and } \zz_0 = a \},
\\
\Per_n^*(\Sigma,a) &:= \{ \zz\in \Per_n(a) : \zz_i \neq a \text{ for all } 1\leq i\leq n-1 \}.
\end{aligned}
\end{equation}
We also write
\begin{equation}\label{eqn:Per}
\Per(\Sigma,a) := \bigcup_{n\in \NN} \Per_n(\Sigma,a) \quad\text{and}\quad
\Per_n(\Sigma) := \bigcup_{a\in V} \Per_n(\Sigma,a).% = \{\zz\in \Sigma : T^n\zz=\zz\}.
\end{equation}
Let $\Sigma$ be topologically transitive and put
\begin{equation}\label{eqn:d}
d = d(\Sigma) := \gcd %\{ k\in \NN : T^k(\zz) = \zz \text{ for some } \zz\in \Sigma\}.
\{ n\in \NN : \Per_n(\Sigma)\neq\emptyset \}.
\end{equation}
Then $\Sigma$ is  \defn{topologically mixing} if $d=1$, and when $d>1$ there is a \defn{spectral decomposition} of $\Sigma$ into disjoint closed sets $\Sigma_0,\dots,\Sigma_{d-1}$ such that $T(\Sigma_i) = \Sigma_{(i+1)\bmod d}$ and each $(\Sigma_i, T^d)$ is a topologically mixing countable-state Markov shift upon recoding by cylinders of length $d$ \cite[Remark 7.1.35]{bK98}.

Given a countable-state Markov shift $\Sigma$ and a shift $X$ on a finite alphabet $A$, we will be interested in \defn{1-block codes} $\pi\colon \Sigma\to X$; this means that there is a map $\tau\colon V\to A$ such that $\pi(\zz)_n = \tau(\zz_n)$ for every index $n$.

We use the thermodynamic formalism for countable-state Markov shifts developed by Sarig and co-authors in \cite{oS99,BS03,CS09}.  Those results are for one-sided shifts, but standard techniques extend the parts we need to two-sided shifts; see \S\ref{sec:spr-statistics}.

As above, $\abs{\Phi}_\beta = \sup \{ \abs{\Phi(\zz) - \Phi(\zz')} / d(\zz,\zz')^\beta : \zz,\zz' \in \Sigma, \zz\neq\zz'\}$  is the H\"older semi-norm of $\Phi$, and $\Cb(\Sigma) = \{\Phi\colon \Sigma\to \RR : \abs{\Phi}_\beta < \infty\}$, so that for every $n\geq 0$ we have%\footnote{In \cite{oS99}, Sarig obtains results for the broader class of \defn{locally H\"older continuous} functions, for which \eqref{eqn:Holder2} is only required to hold for all $n\geq 1$.  The stronger condition here suffices for our purposes.}
%$\Cb(\Sigma)$ for the set of all functions $\Phi\colon \Sigma\to \RR$ such that there is $\abs{\Phi}_\beta>0$ that makes the following hold for every $n\geq 0$.
\begin{equation}\label{eqn:Holder2}
\abs{\Phi(\zz) - \Phi(\zz')} \leq \abs{\Phi}_\beta e^{-\beta n} \text{ whenever } \zz,\zz' \in \Sigma \text{ satisfy } \zz_k = \zz_k' \text{ for all } \abs{k}\leq n.
\end{equation}
Let $\Ch(\Sigma)  := \bigcup_{\beta>0} \Cb(\Sigma)$ be the set of H\"older continuous functions.\footnote{Sarig's results hold for the broader class of 
%More generally, following \cite{oS99}, say that $\Phi$ is 
\defn{locally H\"older continuous} $\Phi$, for which there is $\abs{\Phi}_\beta$ such that \eqref{eqn:Holder2} holds for all $n\geq 1$, 
%but not necessarily for $n=0$, we say that $\Phi$ is  \cite{oS99}.
with no requirement on $n=0$; equivalently, with no requirement that $\Phi$ be bounded.  All the potentials we deal with will be bounded.}
Given $\Phi\in \Ch(\Sigma)$ and $a\in V$, 
write
\begin{align}
\label{eqn:Zn}
Z_n(\Phi,a) := \sum_{\zz\in \Per_n(\Sigma,a)}
%_{T^n\zz = \zz;\  \zz_0=a} %\stackrel{T^n\zz=\zz}{\zz_0=a}} 
e^{S_n\Phi(\zz)}, 
\qquad
Z_n^*(\Phi,a) := \sum_{\zz\in \Per_n^*(\Sigma,a)}
%_{\stackrel{T^n\zz=\zz;\ \zz_0=a}{ \zz_1,\dots,\zz_{n-1}\neq a}}
 e^{S_n\Phi(\zz)}.
\end{align}
These vanish if $n$ is not a multiple of $d$ from \eqref{eqn:d}.\footnote{The definitions here are used by Sarig in the topologically mixing case $d=1$; since the shifts we use may have $d>1$ before we apply the spectral decomposition, we make these definitions more generally.  See also \cite{BS03} for the extension from transitivity to mixing.}
The \defn{Gurevich pressure} of $\Phi$ is
\begin{equation}\label{eqn:gurevich}
P_G(\Phi) := \ulim_{n\to\infty} \frac 1n \log Z_n(\Phi,a);
\end{equation}
this is independent of $a$ by topological transitivity.  If $\Sigma$ is topologically mixing then the limit exists.  When $\Phi=0$ we obtain the \defn{Gurevich entropy} $h_G(\Sigma) = P_G(0)$.

The potential $\Phi$ is %\defn{positive recurrent} if there is $C>0$ with\footnote{In \cite{oS01} the equivalent condition $\sum nZ_n^*(\Phi,a) e^{-nP_G(\Phi)} < \infty$ is used; we follow \cite[Definition 2]{oS99}. }
%\begin{equation}\label{eqn:PR-1}
%C^{-1} e^{nP(\ph)} \leq Z_n(\Phi,a) \leq C e^{nP(\ph)} \text{ for all } n\in \NN\text{ such that } d|n,
%\end{equation}
%and 
\defn{strongly positive recurrent (SPR)} if
\begin{equation}\label{eqn:SPR}
\ulim_{n\to\infty} \frac 1n \log Z_n^*(\Phi,a) < P_G(\Phi).
\end{equation}
The definition of SPR in \cite{oS01} is given in terms of positivity of 
a certain \defn{discriminant}. % $\Delta_a[\Phi]$.  
Equivalence of the two definitions (in the topologically mixing case) follows from \cite{oS01} but is not explicitly stated there; for completeness we prove it in \S\ref{sec:spr-pf}.% that $\Delta_a[\Phi]>0$ iff \eqref{eqn:SPR} holds.

\subsection{Statistical properties}\label{sec:statistical-properties}

We recall several statistical properties a measure can have, which will all be satisfied for the unique equilibrium state produced by \ref{spec}--\ref{stay-good}.

\subsubsection{Gibbs property}\label{sec:gibbs}

Given $\GGG\subset \LLL$ and $\ph\colon X\to \RR$, we say that $\mu$ has the \defn{Gibbs property for $\ph$ on $\GGG$} if there is $Q_1>0$ such that\footnote{We will use $Q_1,Q_2,\dots$ to denote `global constants' that are referred to throughout the paper.  We will use $K$ or $C$ for `local constants' that appear only within the proof of a given lemma or proposition, and are not used for more than one or two paragraphs.}
\begin{equation}\label{eqn:gibbs}
\begin{aligned}
&\text{for every $w\in \LLL$, we have $\mu[w]\leq Q_1e^{-\abs{w}P(\ph) + \hat\ph(w)}$; and}\\
&\text{for every $w\in \GGG$, we have $\mu[w]\geq Q_1^{-1} e^{-\abs{w}P(\ph) + \hat\ph(w)}$.}
\end{aligned}
\end{equation}
Note that the lower bound is only required on $\GGG$, while the upper bound holds on all of $\LLL$.  This weakened version of the Gibbs property was introduced in \cite{CT1,CT2}.

\subsubsection{Periodic orbits}

For $n\in \NN$, let $\Per_n(X) = \{x\in X \mid \sigma^n x=x \}$ be the set of $n$-periodic points.  Note that $\Per_n(X)$ is finite.  Let
\begin{equation}\label{eqn:mun}
\mu_n := \frac 1{\sum_{k=1}^n \sum_{x\in \Per_k(X)} e^{S_k\ph(x)}}
\sum_{k=1}^n \sum_{x\in \Per_k(X)} e^{S_k\ph(x)} \delta_x
\end{equation}
be the \defn{$\ph$-weighted periodic orbit measure} corresponding to periodic orbits of length at most $n$.  Say that $\mu$ is the \defn{limiting distribution of $\ph$-weighted periodic orbits} if $\mu_n$ converges to $\mu$ in the weak* topology.

\subsubsection{Bernoulli property}

Given a state space $S$ and a probability vector $p = (p_a)_{a\in S}$, the \defn{Bernoulli scheme} with probability vector $p$ is $(S^\ZZ, \sigma, \mu_p)$, where $\sigma$ is the left shift map and $\mu_p[w] = \prod_{i=1}^{\abs{w}} p_{a_i}$ for every $w\in S^*$.  When $X$ is a two-sided shift space, we say that $\mu$ has the \defn{Bernoulli property up to a period}
if there are disjoint sets $Y_0,\dots, Y_{d-1} \subset X$ such that $\mu(\bigcup_{i=0}^{d-1} Y_i) =1$, $\sigma(Y_i) = Y_{(i+1)\bmod d}$, and $(Y_i,\sigma^d,\mu|_{Y_i}\cdot d)$ is measure-theoretically isomorphic to a Bernoulli scheme.
In particular, this means that $(X,\sigma,\mu)$ is measure-theoretically isomorphic to the direct product of a Bernoulli scheme and a finite rotation.

The analogous property for one-sided shifts is stronger than we can reasonably expect to obtain in general, since entropy is not a complete invariant of one-sided Bernoulli schemes.  Thus for one-sided shift spaces, all that our results give regarding the Bernoulli property is that the natural extension is Bernoulli up to a period.

\subsubsection{Decay of correlations}

Let $(X,\sigma,\mu)$ be a shift space with an invariant measure $\mu$.  Given $\psi_1,\psi_2\colon X\to \RR$, consider the \defn{correlation functions}
\[
\Cor^\mu_n(\psi_1,\psi_2) := \int (\psi_1\circ \sigma^n) \psi_2 \,d\mu - \int \psi_1\,d\mu \int \psi_2\,d\mu.
\]
Here we take $n\in \NN$ if $X$ is one-sided, and $n\in \ZZ$ is $X$ is two-sided.
We say that the system has \defn{exponential decay of correlations} for observables in $\Cb(X)$ if there is $\theta \in (0,1)$ such that for every $\psi_1,\psi_2 \in \Cb(X)$ there is $K(\psi_1,\psi_2)>0$ such that
\begin{equation}\label{eqn:edc}
\abs{\Cor_n^\mu(\psi_1,\psi_2)} \leq K(\psi_1,\psi_2) \theta^{\abs{n}}
\text{  for every $n$}.
\end{equation}
Since $\Cor_n^\mu$ is linear in both $\psi_1,\psi_2$, it suffices to consider the case $\int \psi_1\,d\mu =\int \psi_2\,d\mu=0$, so $\Cor_n^\mu(\psi_1,\psi_2) = \int (\psi_1 \circ \sigma^n) \psi_2\,d\mu$.
It is also worth noting that although \eqref{eqn:edc} does not specify how $K$ depends on $\psi_1,\psi_2$, we can use the following lemma, proved in \S\ref{sec:UBP}, to deduce existence of $C>0$ such that $K(\psi_1,\psi_2) = C\|\psi_1\|_\beta \|\psi_2\|_\beta$ satisfies \eqref{eqn:edc}. 

\begin{lemma}\label{lem:UBP}
Let $(\mathcal{B}_1,\|\cdot\|_1)$ and $(\mathcal{B}_2,\|\cdot\|_2)$ be Banach spaces, and $\{F_i \colon \BBB_1\times\BBB_2\to \RR\}_{i\in I}$ a family of continuous bilinear functions.  Suppose $\Delta\colon I\to (0,\infty)$ is such that for every $\psi_1\in \BBB_1$ and $\psi_2\in \BBB_2$, there is $K(\psi_1,\psi_2)>0$ such that $\abs{F_i(\psi_1,\psi_2)} \leq K(\psi_1,\psi_2) \Delta(i)$ for all $i\in I$.  
%let $I\subset \ZZ$, and let $\{F_n \colon \BBB_1\times \BBB_2 \to \RR\}_{n\in I}$ be a family of continuous bilinear functions for which there is $\theta>0$ such that $\abs{F_n(\psi_1,\psi_2)} \leq K(\psi_1,\psi_2) \theta^{\abs{n}}$ for all $n\in I$.  
Then there is $C>0$ such that
for all $\psi_1\in \BBB_1$, $\psi_2\in \BBB_2$, and $i\in I$, we have
 $\abs{F_i(\psi_1,\psi_2)} \leq C \|\psi_1\|_1 \|\psi_2\|_2 \Delta(i)$.
\end{lemma}

Lemma \ref{lem:UBP} applies to the estimate in \eqref{eqn:edc} by taking $I=\ZZ$ for two-sided shifts and $I=\NN$ for one-sided shifts, then putting $F_n(\psi_1,\psi_2) = \Cor_n^\mu(\psi_1,\psi_2)$ and $\Delta(i) = \theta^{\abs{i}}$.

As with the Bernoulli property, we say that $(X,\sigma,\mu)$ has exponential decay of correlations \defn{up to a period} if 
%there are $d\in \NN$ and a subshift $Y\subset X$ satisfying \eqref{eqn:YinX} such that $(Y,\sigma^d,\nu)$ has exponential decay of correlations, where once again $\nu = \frac 1{\mu(Y)}\mu|_Y  = \mu|_Y \cdot d$.
if there are disjoint sets 
$Y_0,\dots, Y_{d-1} \subset X$ such that $\mu(\bigcup_{i=0}^{d-1} Y_i) =1$, $\sigma(Y_i) = Y_{(i+1)\bmod d}$, and $(Y_i,\sigma^d,\mu|_{Y_i}\cdot d)$ has exponential decay of correlations.

\subsubsection{Central limit theorem}

Given $(X,\sigma,\mu)$ as above and $\psi\colon X\to \RR$, we say that the \defn{central limit theorem} holds for $\psi$ if $\frac 1{\sqrt{n}} S_n(\psi - \int\psi\,d\mu)$ converges in distribution to a normal distribution $\mathcal{N}(0,\sigma_\psi)$ for some $\sigma_\psi\geq 0$; that is, if
\[
\lim_{n\to\infty} \mu \left\{ x \mid \frac 1{\sqrt{n}} \left( S_n \psi(x) - n\int\psi\,d\mu\right) \leq \tau \right\}
= \frac 1{\sigma_\psi\sqrt{2\pi}} \int_{-\infty}^\tau e^{-t^2/(2\sigma_\psi^2)} \,dt
\]
for every $\tau\in \RR$.  (When $\sigma_\psi=0$ the convergence is to the Heaviside function.)
Say that $\psi$ is \defn{cohomologous to a constant} if there are a $\mu$-integrable function $u\colon X\to \RR$ and a constant $c\in \RR$ such that $\psi(x) = u(x) - u(\sigma x) + c$ for $\mu$-a.e.\ $x\in X$.  One generally expects that $\sigma_\psi=0$ if and only if $\psi$ is cohomologous to a constant.  This will hold for us as well.

%\footnote{One could likely go further and derive a \defn{Green--Kubo formula} expressing the variance $\sigma_\psi$ as a sum of correlations \cite[Theorem 1.1(2)]{cL96}, or deduce an \defn{almost sure invariance principle} \cite{MN05}, but since the main work in this paper is the construction of the tower itself, we will not discuss these here.}

\section{Intermediate results and structure of the proofs}\label{sec:subsidiary-results}

We will prove Theorem \ref{thm:structure} via Theorems \ref{thm:get0spec} and \ref{thm:tower} below (and their corollaries); for Theorem \ref{thm:spr-a}, see Theorem \ref{thm:spr}.  The strategy for Theorem \ref{thm:structure} is to use \ref{spec}--\ref{stay-good} to produce $\FFF\subset \LLL$ satisfying the following \defn{free concatenation} property, which 
is just \ref{spec} with $\tau=0$.\footnote{When $\tau=0$, the gluing word $u$ is forced to be the empty word, and thus is independent of $v,w$.}
\begin{enumerate}[leftmargin=*, widest=I, label = \textup{\textbf{[\Roman{*}$_0${]}}}]
\item\label{free}
Given any $v,w\in \FFF$ we have $vw\in \FFF$.
\end{enumerate}
Once the collection $\FFF$ has been produced there is a natural way to construct a countable-state Markov shift $\Sigma$ that represents (part of) $X$. %We describe the two-sided case; the one-sided is analogous. 
Writing
\begin{equation}\label{eqn:I}
I=I(\FFF) := \FFF \setminus \FFF \FFF = \{w\in \FFF \mid w\neq uv \text{ for any non-trivial } u,v\in \FFF\}
\end{equation}
for the set of \defn{irreducible} elements of $\FFF$,\footnote{We adopt the convention that $\FFF$ does not contain the empty word.} we use the alphabet
\begin{equation}\label{eqn:B}
A_I = \{ (w,k) \in I\times \NN \mid w\in I \text{ and } 1\leq k \leq \abs{w} \};
\end{equation}
think of $(w,k)$ as representing the state ``we are currently in the word $w$, and have seen the first $k$ symbols of $w$''.  Let $\Sigma = \Sigma(\FFF) \subset (A_I)^\ZZ$ be the Markov shift with transitions
\begin{equation}\label{eqn:transitions}
\begin{aligned}
(w,k) &\to (w,k+1) \text{ for any } w\in I \text{ and } %k\in \{1,\dots,\abs{w}-1\}, \\
1\leq k < \abs{w}, \\
(w,\abs{w}) &\to (v,1) \text{ for any } w,v\in I.
\end{aligned}
\end{equation}
Then $\Sigma$ is topologically transitive, and the value $d$ from \eqref{eqn:d} is given by $d = \gcd\{ \abs{w} : w\in \FFF \}$.  In particular, $\Sigma$ is topologically mixing if and only if $\gcd\{\abs{w}:w\in \FFF\}=1$.
Define a 1-block code $\pi\colon \Sigma\to X$ by $(w,k) \mapsto w_k$, the $k$th symbol of $w$. 
%We can also start the construction with $I$ instead of $\FFF$; if $I\subset \LLL$ is such that $I^*\subset \LLL$ and $w\neq uv$ for every $u,v,w\in I$, then writing $\FFF=I^*$ we get $(\Sigma,\pi)$ as above.
For $\Sigma$ and $\pi$ to satisfy the conclusions of Theorem \ref{thm:structure}, we need two more conditions on $I$ and $\FFF = I^*$.
\begin{enumerate}[leftmargin=*, widest=II,
label = \textup{\textbf{[\Roman{*}$'${]}}}]
\setcounter{enumi}{1}
\item\label{gap-2}
 $P(I,\ph) < P(\ph)$, and there are $\Ep,\Es \subset \LLL$ with $P(\Ep \cup \Es \cup (\LLL \setminus \Ep\FFF\Es),\ph) < P(\ph)$.
\end{enumerate}
Note that if $\FFF$ satisfies \ref{gap-2}, then $\GGG = \FFF$ satisfies \ref{gap} with $\Cps = \Eps$.
We also need the following variant of \ref{stay-good}.
\begin{enumerate}[leftmargin=*, widest=III, label = \textup{\textbf{[\Roman{*}$^*${]}}}]
\setcounter{enumi}{2}
\item\label{overlaps}
If $x\in X$ and $i\leq j\leq k\leq \ell$ are integers such that $x_{[i,k)},x_{[j,\ell)}\in \FFF$, and there are $a<j$ and $b>k$ such that $x_{[a,j)},x_{[k,b)}\in \FFF$, then $x_{[j,k)}\in \FFF$.
\end{enumerate}
To compare \ref{overlaps} and \ref{stay-good} it is helpful to reformulate the latter as follows.\footnote{Note that \ref{overlaps} has no requirement on $k-j$, while \ref{stay-good} does not ask for $a,b$, so neither of these conditions implies the other one.}
\begin{enumerate}[leftmargin=*, widest=III, label = \textup{\textbf{[\Roman{*}{]}}}]
\setcounter{enumi}{2}
\item There is $L\in \NN$ such that if $x\in X$ and $i\leq j \leq k\leq \ell$ are integers such that $k-j \geq L$ and $x_{[i,k)},x_{[j,\ell)}\in \GGG$, then $x_{[j,k)}, x_{[i,\ell)} \in \GGG$.
\end{enumerate}
%The naturality of \ref{stay-good} and \ref{overlaps} is discussed in \S\ref{sec:nonsymbolic}, and the question of their necessity is discussed in \S\ref{sec:coded}.  
Figure \ref{fig:overlaps} illustrates both conditions.  Although $i,j,k,\ell$ must appear in the order shown, $a$ can be to either side of $i$ (or equal to it), so long as $a <j$; similarly, there is no constraint on $b$ and $\ell$.  There is no requirement in \ref{overlaps} that $k-j$ be large; in practice we will produce $\FFF$ such that every overlap has length at least $L$, so the case $k-j<L$ will never arise.

\begin{figure}[htbp]
\includegraphics[width=300pt]{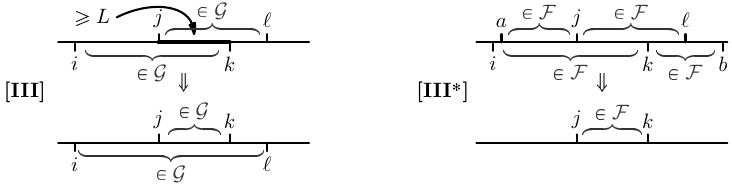}
\caption{Conditions \ref{stay-good} and \ref{overlaps}.}
\label{fig:overlaps}
\end{figure}

Condition \ref{stay-good} implies the following two conditions (but not vice versa). These can be used in its place in Theorems \ref{thm:main} and \ref{thm:structure},\footnote{We will use \ref{inter} repeatedly, while \ref{union} only appears once (see the footnote at the end of \S\ref{sec:get-gap}).} which is useful for the applications in \cite{CP}.
%in some applications it is useful to be able to work with the following weaker version of \ref{stay-good}; see \cite{CP}.
\begin{enumerate}[leftmargin=*, widest=a, label = \textup{\textbf{[III$_\mathrm{\alph{*}}${]}}}]
\item\label{inter}
There is $L$ such that if $uv,vw\in \GGG$, $\abs{v}\geq L$, and $uvw\in \LLL$, then $v\in \GGG$.
\item\label{union}
There is $L$ such that if $uv,vw\in \GGG$, $\abs{v}\geq L$, and $xuvw\in \GGG$ for some $x\in \LLL$, then $uvw\in \GGG$.
\end{enumerate}
%Note that \ref{inter} and \ref{union} both follow from \ref{stay-good}, but do not necessarily imply it.
The following result
%, which is proved in \S\ref{sec:get0spec}, 
gives one half of the proof of Theorem \ref{thm:structure}; the other half is given by %Theorems \ref{thm:tower} and \ref{thm:spr} below.
Theorem \ref{thm:tower} and its Corollaries \ref{cor:2-sided} and \ref{cor:1-sided}. % See \S\ref{sec:get0spec} for the proof.

\begin{theorem}\label{thm:get0spec}\label{THM:GET0SPEC}
Let $X$ be a one- or two-sided shift space 
on a finite alphabet, and let $\ph \in \Ch(X)$. Suppose there is $\GGG\subset \LLL = \LLL(X)$ satisfying \ref{spec},  \ref{gap}, \ref{inter}, and \ref{union}.  Then there is $\FFF \subset \LLL$ satisfying \ref{free}, \ref{gap-2}, and \ref{overlaps} such that a measure $\mu$ has the Gibbs property for $\ph$ on $\GGG$ if and only if it has the Gibbs property for $\ph$ on $\FFF$.  
\begin{samepage}
If in addition $\GGG$ satisfies
\begin{enumerate}[leftmargin=*, widest=I, label = \textup{\textbf{[\Roman{*}$'${]}}}]
\item\label{s-spec}
there is $\tau\in \NN$  such that for all $v,w\in \GGG$, there is $u\in \LLL$ with $\abs{u}= \tau$ such that $v'uw'\in\GGG$ whenever $v'\in\GGG$ is a suffix of $v$ and $w'\in\GGG$ is a prefix of $w$,
%and $v_{[i,\abs{v}]} u w_{[1,j]} \in \GGG$ for every $1\leq i\leq \abs{v}$ and $1\leq j\leq\abs{w}$ with $v_{[i,\abs{v}]}, w_{[1,j]} \in \GGG$;
\end{enumerate}
then $\gcd\{ \abs{w} \mid w\in \FFF\} = \gcd \{ \abs{v} + \tau \mid v\in \GGG\}$.
\end{samepage}
\end{theorem}

The proof of Theorem \ref{thm:get0spec} is given in \S\ref{sec:get0spec}.

%Theorem \ref{thm:get0spec} is proved in \S\ref{sec:get0spec}.  The starting point is the construction of $\FFF$, which is inspired by Bertrand's construction of a synchronising word for a shift with specification \cite{aB88}.
% and uses ideas from the author's previous work with D.J.\ Thompson on non-uniform specification \cite{CT1,CT2}, as well as some new ideas.  In particular, the construction of $\FFF$ is inspired by Bertrand's construction of a synchronising word for a shift with specification \cite{aB88}.

\begin{remark}\label{rmk:hE}
Although Theorem \ref{thm:get0spec} guarantees that $\EEE := I \cup \Ep \cup \Es \cup (\LLL \setminus \Ep \FFF \Es)$ has $P(\EEE,\ph) < P(\ph)$ (by \ref{gap-2}), we may have $P(\EEE,\ph) > P(\Cp \cup \Cs \cup (\LLL\setminus \Cp\GGG\Cs,\ph)$.
%where $\CCC := \Cp \cup \Cs \cup (\LLL\setminus \Cp\GGG\Cs)$.  
This happens already for SFTs;
let $X$ be the SFT on the alphabet $\{1,\dots,k\}$ where $x$ is allowed if and only if $x_{n+1} - x_n = 1 \text{ or } 2 \pmod k$ for every $n$.  Then $h(X) = \log 2$, and we show in \S\ref{sec:hE} that if $\FFF\subset \LLL$ is any collection satisfying \ref{free}, then for every choice of $\Ep,\Es \subset \LLL$ we have $h(\EEE) \geq h(X) - \frac {4\log 2}k$.  This is despite the fact that $\LLL$ itself has specification and so we can take $\GGG=\LLL$ and $\Cp = \Cs=\emptyset$ in Theorem \ref{thm:get0spec}.
\end{remark}

Once we have obtained $\FFF$ satisfying \ref{free}, \ref{gap-2}, and \ref{overlaps}, the relationship between $\Sigma$ and $X$ is given by the following result; see Figure \ref{fig:structure} for the overall structure.

%. The relationship between the various conditions and conclusions is illustrated by \eqref{eqn:structure} following the statement. 

\begin{theorem}\label{thm:tower}\label{THM:TOWER}
Let $X$ be a two-sided shift space on a finite alphabet with language $\LLL$, and  let $\ph\in \Ch(X)$.
Suppose that $\FFF \subset \LLL$
satisfies \ref{free} and let $I= \FFF\setminus\FFF\FFF$.
Let $\Sigma,\pi$ be as in \eqref{eqn:I}--\eqref{eqn:transitions}.
\begin{enumerate}[leftmargin=*, widest=G,
 label=\textup{$\langle$\Alph{*}$\rangle$}]
\item\label{overlaps-injective}
If $\FFF$ satisfies \ref{overlaps} then $\pi$ is 1-1 on $\Sigma$.
%\item If $\pi$ is 1-1 then $I$ is uniquely decipherable.
\item\label{decipherable-periodics}\label{injective-decipherable}
$I$ is uniquely decipherable if and only if $\pi$ is 1-1 on 
$\Per(\Sigma,a)$ for every $a\in A_I$.
%\item\label{periodic-PG}
%If $\pi$ is 1-1 on $\Per_n(\Sigma,a)$ for every $a\in A_I$ and $n\in \NN$, then $\Phi = \ph\circ \pi \in \Ch(\Sigma)$ satisfies $P_G(\Phi) \leq P(\ph)$.
\item\label{decipherable-spr}
If $I$ is uniquely decipherable and $\FFF$ satisfies \ref{gap-2}, then 
$\Phi = \ph\circ \pi \in \Ch(\Sigma)$ has $P_G(\Phi) = P(\ph)$ and is strongly positive recurrent.
\item\label{charges}
If $\FFF$ satisfies \ref{gap-2}, then there is $P' < P(\ph)$ such that every ergodic $\mu\in \Ms(X)$ with $h(\mu) + \int\ph\,d\mu > P'$ satisfies $\mu(\pi(\Sigma))=1$.%
\footnote{If $\pi$ is 1-1, this immediately implies the liftability condition $\mu = \pi_* \nu$ by taking $\nu := (\pi^{-1})_* \mu$.  Without injectivity the situation is more delicate.}
\item\label{decipherable-finite}
If $I$ is uniquely decipherable and $\FFF$ satisfies \ref{gap-2}, then there is $P' < P(\ph)$ such that every ergodic $\mu\in \Ms(X)$ with $h(\mu) + \int\ph\,d\mu > P'$ has $\#\pi^{-1}(x)<\infty$  $\mu$-a.e.
\item\label{finite-lifts}
If $\mu$ is a $\sigma$-invariant probability measure on $\pi(\Sigma)$ 
and $\mu$-a.e.\ $x$ has $\#\pi^{-1}(x) < \infty$, then there is a $T$-invariant measure $\nu$ on $\Sigma$ such that $\mu = \pi_*\nu$ and $h(\mu) = h(\nu)$.\footnote{This differs from the rest of the statements in the theorem in that it is simply a special case of a standard result on finite-to-one measurable factor maps; see Lemma \ref{lem:finite-lifts}.}
%In particular, conclusion \ref{lifts} of Theorem \ref{thm:structure} holds.
\item\label{D-gap} If $I$ is a generating set for $X$ and $\DDD = \DDD(I) := \{w\in \LLL \mid uwv \in I$ for some $u,v\in \LLL\}$ has the property that $P(\DDD(I),\ph) < P(\ph)$, then $\FFF=I^*$ satisfies \ref{gap-2}.
\item\label{gap-gibbs}
If $\FFF$ satisfies \ref{gap-2}, then there is at least one equilibrium state $\mu$ that has the Gibbs property for $\ph$ on $\FFF$.
\end{enumerate}
\end{theorem}

%Theorem \ref{thm:tower} and Corollary \ref{cor:tower} are proved in \S\ref{sec:towers2}; the corollary relies on pre-existing machinery developed by various people, including especially Sarig \cite{oS99,oS01,CS09,oS11} and Young \cite{lY98,lY99}.  

Theorem \ref{thm:tower} is proved in \S\ref{sec:towers2} and is
illustrated by Figure \ref{fig:structure}. The properties involving $\mu$ in the last row of that figure are to hold for every ergodic $\mu$ with $h(\mu) + \int\ph\,d\mu>P'$, where $P'$ is some threshold $<P(\ph)$; in particular, they hold for every equilibrium state (since every ergodic component of an equilibrium state is itself an equilibrium state).\footnote{In fact we will eventually prove \emph{a posteriori} that the equilibrium state is unique, but this must wait until Theorem \ref{thm:spr}.}
The labels on the arrows indicate which part of Theorem \ref{thm:tower} establishes the corresponding implication.  Two arrows leading to a single property indicate an ``and'' relationship: both of the corresponding hypotheses are required to deduce the conclusion.

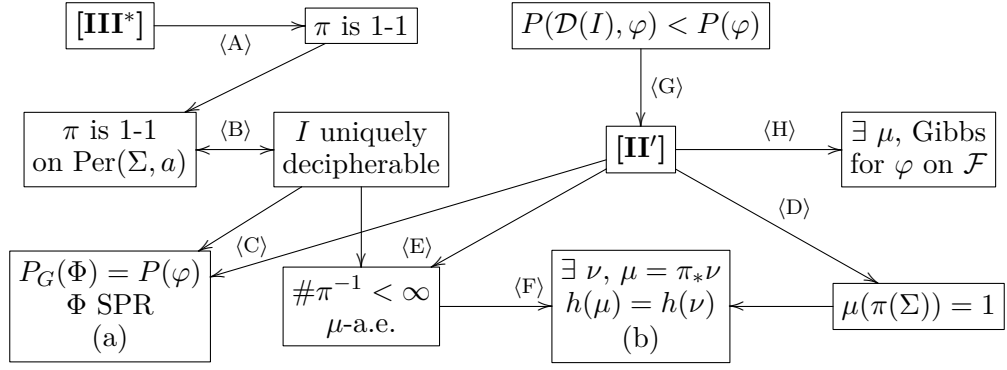
\begin{figure}[htbp]
%\begin{equation}\label{eqn:structure}
\begin{gather*}
\xymatrix{
*+[F]{\ref{overlaps}}\ar_{\ref{overlaps-injective}}[r]
& *+[F]{\txt{$\pi$ is 1-1}}%\ar^{\ref{injective-decipherable}}[d]
\ar[dl]
& *+[F]{
\iftoggle{arxiv}{P(\DDD(I),\ph) < P(\ph)}{\txt{$P(\DDD(I),\ph)$ \\ $ < P(\ph)$}}
}\ar^{\ref{D-gap}}[d]
& \\
*+[F]{\txt{$\pi$ is 1-1\\ on $\Per(\Sigma,a)$}}
& *+[F]{\txt{$I$ uniquely\\ decipherable}}\ar_{\ref{decipherable-periodics}}@{<->}[l]
\ar^{\ref{decipherable-spr}}[dl]
\ar[d]
& *+[F]{\ref{gap-2}}\ar^{\ref{gap-gibbs}}[r]
\ar[dll]
\ar_(.75){\ref{decipherable-finite}}[dl]
\ar^{\ref{charges}}[dr]
& *+[F]{\txt{$\exists$ $\mu$, Gibbs \\ for $\ph$ on $\FFF$}} \\
*+[F]{\txt{$P_G(\Phi)=P(\ph)$ \\ $\Phi$ SPR \\ \ref{spr}}}
& *+[F]{\txt{$\#\pi^{-1}<\infty$ \\ $\mu$-a.e.}}\ar^(.6){\ref{finite-lifts}}[r]
& *+[F]{\txt{$\exists$ $\nu$, $\mu = \pi_* \nu$ \\ $h(\mu) = h(\nu)$ \\ \ref{lifts}}}
& *+[F]{\mu(\pi(\Sigma))=1} \ar[l]
}
\end{gather*}
\caption{The logical structure of Theorem \ref{thm:tower}, with conclusions \ref{spr} and \ref{lifts} of Theorem \ref{thm:structure} marked.}
\label{fig:structure}
%\end{equation}
\end{figure}

%We give three corollaries describing the consequences of Theorem \ref{thm:tower} for 

\begin{corollary}\label{cor:2-sided}
Let $X$ be a two-sided shift space on a finite alphabet, and let $\ph\in\Ch(X)$.  Suppose $\FFF\subset \LLL(X)$ satisfies \ref{free}, \ref{gap-2}, and \ref{overlaps}, and let $\Sigma$ be the two-sided countable-state Markov shift given by \eqref{eqn:I}--\eqref{eqn:transitions}.  Then the 1-block code $\pi\colon \Sigma\to X$ defined by $(w,k)\mapsto w_k$  is 1-1, and conclusions \ref{spr}--\ref{lifts} of Theorem \ref{thm:structure} hold.  Moreover, $(X,\sigma,\ph)$ has at least one equilibrium state satisfying the Gibbs property for $\ph$ on $\FFF$.%\footnote{In Corollaries \ref{cor:2-sided}--\ref{cor:1-sided}, the equilibrium state is in fact unique, but we need Theorem \ref{thm:spr} to prove this.}
\end{corollary}
\begin{proof}
Follows from parts \ref{overlaps-injective}, \ref{injective-decipherable}, \ref{decipherable-spr}, \ref{charges}, and \ref{gap-gibbs} of Theorem \ref{thm:tower}.
\iftoggle{arxiv}{}{\qed}
\end{proof}

We remark that as in Theorem \ref{thm:tower}, the equilibrium state in Corollaries \ref{cor:2-sided}, \ref{cor:coded}, and \ref{cor:1-sided} is in fact unique, but we need Theorem \ref{thm:spr} to prove this.

\begin{corollary}\label{cor:coded}
Let $X$ be a coded shift, and suppose $I$ is a uniquely decipherable generating set with $P(\DDD(I),\ph) < P(\ph)$ for some $\ph\in \Ch(X)$.  Let $\Sigma$ be the two-sided countable-state Markov shift defined in \eqref{eqn:B}--\eqref{eqn:transitions}, and $\pi\colon \Sigma\to X$ the 1-block code defined there.
Then there is $P' < P(\ph)$ such that if $\mu\in \MMM_\sigma(X)$ is ergodic and has $h(\mu) + \int\ph\,d\mu > P'$, then $\mu$-a.e.\ $x\in X$ has $\#\pi^{-1}(x) < \infty$.  Moreover,
$\pi$ is 1-1 on $\Per(\Sigma,a)$ for every state $a$ of $\Sigma$;
conclusions \ref{spr}--\ref{lifts} of Theorem \ref{thm:structure} hold; and $(X,\sigma,\ph)$ has at least one equilibrium state satisfying the Gibbs property for $\ph$ on $\FFF=I^*$.
\end{corollary}
\begin{proof}
Follows from parts \ref{D-gap}, \ref{decipherable-periodics}, \ref{decipherable-spr}, \ref{charges}, \ref{decipherable-finite}, \ref{finite-lifts}, and \ref{gap-gibbs} of Theorem \ref{thm:tower}.
\iftoggle{arxiv}{}{\qed}
\end{proof}

The following example shows that without condition \ref{overlaps} or unique decipherability, the coding map $\pi$ can fail to relate equilibrium states on $\Sigma$ to equilibrium states on $X$.

\begin{example}\label{eg:not1-1}
Let $X\subset \{0,1\}^\ZZ$ be the SFT defined by forbidding the word $111$, and let $\FFF = \{0\} \cup \{w\in \LLL_{\geq 2} : w$ neither starts nor ends with the word $11\}$.  Then $\FFF$ satisfies \ref{free}, we get $I = I(\FFF) = \{0,01,10,101\}$ as the irreducible elements of $\FFF$, and 
%\ref{gap-2} holds for every $\ph\in \Ch(X)$ by taking $\Ep = \Es = \{11\}$.
$\DDD(I)$ is finite so $P(\DDD(I),\ph) < P(\ph)$ for all $\ph\in \Ch(X)$.
  However, $010 = (01)(0) = (0)(10) \in \FFF$ has two different `factorisations'; in particular,    $\pi^{-1}(x)$ is uncountable whenever $x$ contains infinitely many copies of the word $010$.
%$\pi^{-1}((010)^\infty)$ is uncountable.  
It is not hard to show that \ref{overlaps} fails and $h(\Sigma) > h(X)$.
\end{example}

For one-sided shift spaces, it turns out that injectivity (even on a set of full measure) is too much to ask for.

\begin{example}\label{eg:1-sided-inj}
Let $X\subset \{0,1\}^\NN$ be the one-sided SFT defined by forbidding the word $11$, and let $\mathcal{F} = \{w\in \LLL(X) : w_{|w|} \neq 1\}$.  One can easily check that $I = \{0, 10\}$ and conditions \ref{free}, \ref{gap-2}, and \ref{overlaps} are satisfied, but every $x\in X$ starting with the symbol $0$ has two preimages under $\pi$.
\end{example}

Consequently, in the following version of Corollary \ref{cor:2-sided} for one-sided shifts, we can only deduce the weaker version of injectivity given in \eqref{eqn:inj}.

\begin{corollary}\label{cor:1-sided}
Let $X$ be a one-sided shift space on a finite alphabet, and let $\ph\in\Ch(X)$.  Suppose $\FFF\subset \LLL(X)$ satisfies \ref{free}, \ref{gap-2}, and \ref{overlaps}.  Let $\Sigma \subset A_I^{\NN\cup \{0\}}$ be the one-sided countable-state Markov shift given by \eqref{eqn:I}--\eqref{eqn:transitions}, and $\pi$ the 1-block code there.  Then \eqref{eqn:inj} holds (and thus
the map $\pi$ is 1-1 on $\Per(\Sigma) := \bigcup_{n\in \NN} \Per_n(\Sigma)$);
conclusions \ref{spr}--\ref{lifts} of Theorem \ref{thm:structure} hold; and $(X,\sigma,\ph)$ has at least one equilibrium state satisfying the Gibbs property for $\ph$ on $\FFF$.
\end{corollary}

Corollary \ref{cor:1-sided} is proved in \S\ref{sec:1-sided-corollary}.  Now Theorem \ref{thm:structure} follows from Theorem \ref{thm:get0spec} and Corollaries \ref{cor:2-sided} and \ref{cor:1-sided}.  Theorem \ref{thm:spr-a} is an immediate corollary of the following result, which differs from it only by clarifying that an even weaker injectivity property\footnote{Injectivity on $\Sigma$ implies \eqref{eqn:inj}, which in turn implies injectivity on $\Per_n(\Sigma,a)$ for every $a,n$.} than \eqref{eqn:inj} suffices to deduce everything except the precise value of the period; this allows it to be used with Corollary \ref{cor:coded} to prove Theorem \ref{thm:coded}.

%especially Sarig \cite{oS99,oS01,CS09,oS11} and Young

\begin{theorem}\label{thm:spr}
Let $(X,\sigma)$ be a one- or two-sided shift space on a finite alphabet, and let $\ph\in \Ch(X)$.  Suppose that there is a topologically transitive countable-state Markov shift $(\Sigma,T)$ and a 1-block code $\pi\colon \Sigma\to X$ 
such that there is a state $a$ for $\Sigma$ such that $\pi$ is 1-1 on $\Per_n(\Sigma,a)$ for all $n\in \NN$, and in addition
properties \ref{spr}--\ref{lifts} from Theorem \ref{thm:structure} hold.
%with the property that 
%\begin{itemize}
%\item $P_G(\Phi) = P(\ph)$ for $\Phi = \ph\circ \pi$; 
%\item $\Phi$ is strongly positive recurrent; and 
%\item there is $P' < P(\ph)$ such that every ergodic $\mu\in \Ms(X)$ with $h(\mu) + \int\ph\,d\mu > P'$  has the form $\mu = \pi_* \nu$ for some $T$-invariant measure $\nu$ on $\Sigma$ with $h(\nu) = h(\mu)$.
%\end{itemize}
Then $P_G(\ph\circ \pi) = P(\ph)$, and $(X,\ph)$ has a unique equilibrium state $\mu$, which satisfies conclusions \ref{periodic}--\ref{analytic} of Theorem \ref{thm:main}.  The period $d$ in \ref{bernoulli} is a factor of $p = \gcd \{k : \Per_k(\Sigma)\neq\emptyset\}$, and $d=p$ if \eqref{eqn:inj} holds.\footnote{In particular, $d=p$ if $\pi$ is injective on $\Sigma$.}
%Then $(X,\ph)$ has a unique equilibrium state $\mu$, which is the limiting distribution of $\ph$-weighted periodic orbits.  With $d=\gcd \{k : Per_k(\Sigma)\neq\emptyset\}$, there are disjoint sets $Y_0,\dots, Y_{d-1} \subset X$ such that $\mu(\bigcup_{i=1}^{d-1} Y_i) = 1$, $\sigma(Y_i) = Y_{(i+1) \bmod d}$, and $(Y_i,\sigma^d,\mu|_{Y_i} \cdot d)$ has the Bernoulli property and exponential decay of correlations for H\"older observables.
%$(X,\sigma,\mu)$ satisfies the central limit theorem for H\"older observables $\psi$, with variance $0$ if and only if $\psi$ is cohomologous to a constant.
%Given any H\"older continuous $\psi\colon X\to \RR$, there is $\eps>0$ such that the topological pressure function $t\mapsto P(\ph + t\psi)$ is real analytic on $(-\eps,\eps)$.
In the case when $X$ is two-sided, $(X,\ph,\mu)$ is Bernoulli up to the period $d$.
\end{theorem}

Theorem \ref{thm:spr} is proved in \S\ref{sec:spr-statistics} %and uses the coding map to deduce conclusions \ref{unique}--\ref{analytic} of Theorem \ref{thm:main} from 
using the existing literature on countable-state Markov shifts.  We use the results formulated by Cyr and Sarig \cite{CS09}, but the work of Young \cite{lY98,lY99} should certainly be mentioned here as well.

%The only remaining result from the introduction is Theorem \ref{thm:coded}, which follows from Corollary \ref{cor:coded} and Theorem \ref{thm:spr}.

%Now Theorems \ref{thm:main} and \ref{thm:structure} follow from Theorem \ref{thm:get0spec}, Corollaries \ref{cor:2-sided} and \ref{cor:1-sided}, and Theorem \ref{thm:spr}.  Theorem \ref{thm:coded} follows from Corollary \ref{cor:coded} and Theorem \ref{thm:spr}.

%Now Theorems \ref{thm:main} and \ref{thm:structure} follow from Theorems \ref{thm:get0spec}, \ref{thm:tower}, and \ref{thm:spr}. Note that Theorem \ref{thm:spr} gives conclusions \ref{unique} and \ref{periodic}--\ref{analytic} of Theorem \ref{thm:main}, but not the Gibbs property \ref{gibbs}.  For this, Theorem \ref{thm:tower}\ref{gap-gibbs} shows that the unique equilibrium state $\mu$ has get the Gibbs property relative to $\ph$ and $\FFF$; by Theorem \ref{thm:get0spec}, this gives the Gibbs property relative to $\ph$ and $\GGG$.

%Theorem \ref{thm:coded} follows from Theorems \ref{thm:tower} and \ref{thm:spr} in a similar way.  Note that in this case we do not assume \ref{overlaps} and thus cannot apply Theorem \ref{thm:tower}\ref{overlaps-injective} to deduce that $\pi$ is 1-1; instead we must use the unique decipherability of the generating set to deduce finite multiplicity $\mu$-a.e.\ via Theorem \ref{thm:tower}\ref{decipherable-finite}.

\section{Preparation for the proofs}\label{sec:prep}

%Before proving Theorems \ref{thm:get0spec}, \ref{thm:tower}, and \ref{thm:spr} in \S\S\ref{sec:towers2}--\ref{sec:get0spec}, we establish some preliminaries.

%preparatory results.% that will be needed later on.  
%In \S\ref{sec:strong-ergodic} we prove stronger versions of the Birkhoff and Shannon--McMillan--Breiman ergodic theorems (Theorems \ref{thm:strong-birkhoff} and \ref{thm:strong-smb}).  In \S\ref{sec:bdd-dist} we show that H\"older potentials have bounded distortion within cylinders, and use this together with \ref{spec} and \ref{gap} to establish a number of uniform bounds on partition sums, a process that also played a key role in \cite{CT1,CT2}.  

\subsection{Ergodic theorems}\label{sec:strong-ergodic}

In the proof of Theorem \ref{thm:tower}, we will need mild strengthenings of the Birkhoff and Shannon--McMillan--Breiman ergodic theorems.  These are general results that hold beyond the setting of this paper.  Throughout this section, $(X,T,\mu)$ is any invertible ergodic measure-preserving transformation.

The Birkhoff ergodic theorem \cite[Theorem 2.2.3]{kP89} 
%can be stated as follows: if $(X,T,\mu)$ is an ergodic measure-preserving transformation and 
says that if
$f\colon X\to \RR$ is an $L^1$ function, then for $\mu$-a.e.\ $x\in X$ and every $\eps>0$ there is $N=N(x,\eps)$ such that for all $n\geq N$ we have
$
|\frac 1n S_nf(x) - \int f\,d\mu| < \eps.
$
We start with an elementary lemma.

\begin{lemma}\label{lem:f-grows-slowly}
For all $(X,T,\mu)$ as above, all $f\in L^1(X,\mu)$, and $\mu$-a.e.\ $x$, 
the function $\tilde f_n(x) = \max\{\lvert f(T^kx)\rvert \mid -n\leq k \leq n \}$ satisfies $\lim_{n\to\infty} \frac 1n \tilde f_n(x) = 0$.
\end{lemma} 
\begin{proof}
By Birkhoff's ergodic theorem, $\mu$-a.e.\ $x$ has the property that $\frac 1n \sum_{k=0}^{n-1} f(T^kx)$ and $\frac 1n \sum_{k=0}^{n-1} f(T^{-k}x)$ both converge; each such $x$ has $\frac 1n f(T^nx) \to 0$ and $\frac 1n f(T^{-n}x)\to 0$.
Let $k_n\in [-n,n]$ be such that $\tilde f_n(x) = f(T^{k_n}x)$.  If there is $K\in \NN$ such that $|k_n| \leq K$ for all $n$, then $\frac 1n |\tilde f_n(x)| \leq \frac 1n |\tilde f_K(x)| \to 0$.  If there is no such $K$, then $|k_n|\to \infty$ since it is nondecreasing, and since $|k_n|\leq n$ we get $\frac 1n |\tilde f_n(x)| \leq |f(T^{k_n}x) / k_n| \to 0$.
\iftoggle{arxiv}{}{\qed}
\end{proof}

\begin{theorem}\label{thm:strong-birkhoff}
If $(X,T,\mu)$ is an invertible ergodic measure-preserving transformation and $f\colon X\to \RR$ is an $L^1$ function, then for $\mu$-a.e.\ $x\in X$ and every $\eps>0$ there is $N=N(x,\eps)$ such that for all $n\geq N$ and  $\ell\in [0,n]$ we have
%\begin{equation}\label{eqn:strong-birkhoff}
$|\frac 1n S_nf(T^{-\ell}x) - \int f\,d\mu| < \eps$.
%\end{equation}
\end{theorem}
\begin{proof}
Without loss of generality assume that $\int f\,d\mu=0$ (else replace $f$ by $f-\int f\,d\mu$). Applying Birkhoff's theorem to $T$ and $T^{-1}$, for every $0<\eps<1$ and $\mu$-a.e.\ $x\in X$ there is $N_0\in \NN$ such that for all $n\geq N_0$ we have $\abs{S_nf(x)} < n\eps/2$ and $\abs{S_n(f^{-n}x)} < n\eps/2$. Let
%\[M = \max(\{ |f^n(x)| : -N_0 < n < N_0\} \cup \{1\})\quad\text{and}\quad N=2MN_0/\eps.\]
$M = \max(1, \tilde f_{N_0}(x))$ and $N=2MN_0/\eps$.
Given $n\geq N$ and $\ell\in [0,n]$, we are in one of the following three cases: (i) $\ell,n-\ell \geq N_0$; (ii) $\ell<N_0$ and $n-\ell\geq N-N_0 \geq MN_0/\eps \geq N_0$; (iii) $n-\ell < N_0$ and $\ell \geq MN_0/\eps\geq N_0$.  In the first case our choice of $N_0$ gives
\[
|S_n f(T^{-\ell} x)| \leq |S_\ell f(T^{-\ell}x)| + |S_{n-\ell} f(x)|
 < \ell\eps/2 + (n-\ell)\eps/2 = n\eps/2 < n\eps.
\]
In the second case our definition of $M$ and $N$ gives $N_0 M = N\eps/2 \leq n\eps/2$, so
\[
|S_n f(T^{-\ell} x)| \leq |S_\ell f(T^{-\ell} x)| + |S_{n-\ell} f(x)|
< N_0 M + (n-\ell) \eps/2 \leq n\eps,
\]
and the third case is similar, which proves the theorem.
\iftoggle{arxiv}{}{\qed}
\end{proof}

Given %$(X,T,\mu)$ as above and 
a countable (or finite) measurable partition $\alpha$ of $X$, write $\alpha(x)$ for the partition element containing $x$, and for integers $i<j$, write $\alpha_i^j = \bigvee_{k=i}^{j-1} T^{-k} \alpha$.  Recall that
\[
H_\mu(\alpha) := \sum_{A\in \alpha} -\mu(A) \log \mu(A), 
\qquad
h_\mu(\alpha,T) := \lim_{n\to\infty} \frac 1n H_\mu(\alpha_0^n).
\]
The Shannon--McMillan--Breiman theorem \cite[Theorem 6.2.3]{kP89} says that if $H_\mu(\alpha) < \infty$, then for $\mu$-a.e.\ $x\in X$ and every $\eps>0$ there is $N=N(x,\eps)$ such that for all $n\geq N$ we have
$
|-\frac 1n \log \mu(\alpha_0^n(x)) - h_\mu(\alpha,T)| < \eps.
$
%We will prove the following stronger version.

\begin{theorem}\label{thm:strong-smb}
If $(X,T,\mu)$ is an invertible ergodic measure-preserving transformation and $\alpha$ is a countable measurable partition with $H_\mu(\alpha)<\infty$, then for $\mu$-a.e.\ $x\in X$ and every $\eps>0$ there is $N=N(x,\eps)$ such that for all $n\geq N$ and $\ell\in [0,n]$ we have
\begin{equation}\label{eqn:strong-smb}
\abs{-\frac 1n \log \mu(\alpha_{-\ell}^{n-\ell}(x)) - h_\mu(\alpha,T)} < \eps.
\end{equation}
\end{theorem}
\begin{proof}
We adapt the standard argument to get the Shannon--McMillan--Breiman theorem from Birkhoff's theorem \cite[Theorem 6.2.3]{kP89}, using Theorem \ref{thm:strong-birkhoff} where necessary.

Let $I_n(x) = -\log \mu(\alpha_0^n(x))$ and $f_n(x) = I_n(x) - I_{n-1}(Tx)= -\log \big( \frac{\mu(\alpha_0^n(x)) }{ \mu(\alpha_1^n(x))} \big) $.  Note that $I_n, f_n \geq 0$.  Then  $f^* := \sup_{n\geq 1} f_n \in L^1$ \cite[Corollary 6.2.2]{kP89}, and as shown in the proof of \cite[Theorem 6.2.3]{kP89}, we have $f_n \to f\geq 0$ both pointwise a.e.\ and in $L^1$, where $f$ is an $L^1$ function such that $\int f\,d\mu = h_\mu(\alpha, T)$.
Moreover, iterating the relation $I_n = f_n + I_{n-1} \circ T$ gives $I_n =f_n + f_{n-1} \circ T + I_{n-2} \circ T^2 = \cdots = \sum_{k=0}^{n-1} f_{n-k} \circ T^k$, and so
%\begin{align*}
%I_n &= f_n + I_{n-1}\circ T = f_n + f_{n-1} \circ T + I_{n-2} \circ T^2 = \cdots = \sum_{k=0}^{n-1} f_{n-k} \circ T^k.
%\end{align*}
%Thus
\begin{equation}\label{eqn:Infn}
-\frac 1n \log \mu(\alpha_{-\ell}^{n-\ell}(x)) = %\frac 1n I_n(x) = \frac 1n S_n f(x) + \frac 1n \sum_{k=0}^{n-1} (f_{n-k} - f)(T^k x).
%\end{equation}
%We can use \eqref{eqn:Infn} to get
%\begin{equation}\label{eqn:fn-sum}
\frac 1n I_n(T^{-\ell}x) = \frac 1n S_n f(T^{-\ell}x) + \frac 1n \sum_{k=0}^{n-1} (f_{n-k} - f)(T^{k-\ell} x).
\end{equation}
By Theorem \ref{thm:strong-birkhoff}, for $\mu$-a.e.\ $x$ and every $\eps>0$ there is $N$ such that for all $n\geq N$, we have
\[
\abs{\frac 1n S_n f(T^{-\ell}x) - h_\mu(\alpha,T)} =  \abs{\frac 1n S_n f(T^{-\ell}x) - \int f\,d\mu} \leq \frac \eps2.
\]
Thus to prove \eqref{eqn:strong-smb} it suffices to find $N'$ such that for all $n\geq N'$ and $0\leq \ell\leq n$, we have
\begin{equation}\label{eqn:sumfnk}
\frac 1n \sum_{k=0}^{n-1} \abs{(f_{n-k} - f)(T^{k-\ell} x)} < \frac\eps 2.
\end{equation}
Given $m\in \NN$, let $F_m = \sup_{j\geq m} \abs{f_j - f}$.  Then $0\leq F_m \leq f^* + f\in L^1$, and $F_m\to 0$ pointwise a.e., so by the dominated convergence theorem there is $m$ such that $0\leq \int F_m \,d\mu < \eps/4$.  We control the part of the sum in \eqref{eqn:sumfnk} with $k\leq n-m$ by applying Theorem \ref{thm:strong-birkhoff} to $F_m$, obtaining $N_1\in \NN$ such that for all $n\geq N_1$ and $\ell\in [0,n]$ we have
\begin{equation}\label{eqn:small-k}
\frac 1n \sum_{k=0}^{n-m} \abs{(f_{n-k} - f)(T^{k-\ell} x)} \leq
\frac 1n \sum_{k=0}^{n-m} F_m(T^{k-\ell}x) < \frac\eps 4.
\end{equation}
To control the part of the sum with $k>n-m$, note that $|f_{n-k} - f| \leq g := f^* + f \in L^1$ for all $k$, so by Lemma \ref{lem:f-grows-slowly} we have
\[
\frac 1n \sum_{k=n-m+1}^{n-1}\abs{(f_{n-k} - f)(T^{k-\ell} x)} \leq \frac 1n \sum_{k=n-m+1}^{n-1} (f^*+f)(T^{k-\ell}x) 
\leq \frac mn \tilde g_n(x) \to 0
\]
as $n\to \infty$ since $m$ is fixed.  Choosing $N_2\in \NN$ such that $\frac mn\tilde g_n(x) < \eps/4$ for all $n\geq N_2$, we can add this to \eqref{eqn:small-k} and obtain \eqref{eqn:sumfnk} for all $n\geq N' := \max(N_1,N_2)$ and $0\leq\ell\leq n$.
\iftoggle{arxiv}{}{\qed}
\end{proof}

\subsection{Bounded distortion and counting estimates}\label{sec:bdd-dist}\label{sec:counting}

Now we return to the setting of symbolic dynamics.  Let $X$ be a one- or two-sided shift space, and $\LLL$ its language.
Given $\beta>0$ and $\ph\in \Cb(X)$, we see that for every $w\in \LLL_n$ and every $x,y\in [w]$, \eqref{eqn:Holder} yields
$\abs{\ph(\sigma^kx)-\ph(\sigma^ky)} \leq \abs{\ph}_\beta e^{-\beta \min(k,n-k)}$ for all $0\leq k < n$,
so that in particular
\begin{equation}\label{eqn:bdd-distortion}
\abs{S_n\ph(x)-S_n\ph(y)} \leq 2\abs{\ph}_\beta \sum_{j=0}^\infty e^{-\beta j} =: \absd{\ph} < \infty.
\end{equation}
This can be thought of as a bounded distortion condition, and we think of $\absd{\ph}$ as the bound on distortion of $S_n\ph$ within an $n$-cylinder.

It follows from \eqref{eqn:bdd-distortion} that for every $v,w \in \LLL$ such that $vw\in \LLL$, we have
\begin{equation}\label{eqn:phvw}
\hat\ph(v) + \hat\ph(w) - \absd{\ph} \leq \hat\ph(vw) \leq \hat\ph(v) + \hat\ph(w).
\end{equation}
(The upper bound is immediate from \eqref{eqn:Phi}.)

In the proofs of both Theorems \ref{thm:get0spec} and \ref{thm:tower}, we will need various estimates on partition sums over $\LLL$ and over $\GGG$.  We start with the general observation that given $\CCC,\DDD\subset \LLL$ and $m,n\in \NN$, the bound $\hat\ph(uv) \leq \hat\ph(u)+\hat\ph(v)$ gives
\begin{equation}\label{eqn:CD}
\Lambda_{m+n}(\CCC_m \DDD_n \cap \LLL,\ph)
\leq \sum_{u\in \CCC_m} \sum_{v\in \DDD_n} e^{\hat\ph(u)} e^{\hat\ph(v)} = \Lambda_m(\CCC,\ph) \Lambda_n(\DDD,\ph);
\end{equation}
this will be used in several places, as will the following consequence.

\begin{lemma}\label{lem:CD}
Given $\CCC,\DDD\subset \LLL$, we have $P(\CCC\DDD \cap \LLL,\ph) \leq \max\{P(\CCC,\ph), P(\DDD,\ph)\}$.  Thus if $\GGG$ satisfies \ref{gap} then $P(\GGG,\ph) = P(\ph)$, and if $\FFF$ satisfies \ref{gap-2} then $P(\FFF,\ph)=P(\ph)$.
\end{lemma}
\begin{proof}
Let $P' > \max\{P(\CCC,\ph),P(\DDD,\ph)\}$ be arbitrary; then there is $C>0$ such that $\Lambda_k(\CCC,\ph) \leq C e^{kP'}$ and $\Lambda_\ell(\DDD,\ph) \leq C e^{\ell P'}$ for all $k,\ell\in \NN$, and \eqref{eqn:CD} gives
\[
\Lambda_n(\CCC\DDD \cap \LLL,\ph) \leq \sum_{k=0}^n \Lambda_k(\CCC,\ph) \Lambda_{n-k}(\DDD,\ph) \leq \sum_{k=0}^n (C e^{kP'}) (C e^{(n-k)P'}) = (n+1)C^2 e^{nP'},
\]
so $P(\CCC\DDD\cap\LLL,\ph) \leq P'$.  Sending $P' \searrow \max\{P(\CCC,\ph),P(\DDD,\ph)\}$ completes the proof.
\iftoggle{arxiv}{}{\qed}
\end{proof}

The following are similar to estimates appearing in \cite[Lemmas 5.1--5.4]{CT1} and \cite[Section 5]{CT2}.  The chief difference here is that we may have $\LLL \setminus \Cp \GGG \Cs \neq \emptyset$, but because the pressure of this collection is controlled, we get the same results.

\begin{lemma}\label{lem:counting}
Let $X$ be a shift space on a finite alphabet and $\ph\in \Ch(X)$.  Let $\GGG\subset \LLL(X)$ be such that \ref{spec} and \ref{gap} hold.  Then there is $Q_2>0$ such that for every $n$ we have
\begin{equation}
\label{eqn:countingL}
e^{nP(\ph)} \leq \Lambda_n(\ph) \leq Q_2 e^{nP(\ph)}.
\end{equation}
Furthermore, there are $Q_3>0$ and $N,n_0\in \NN$ such that for every $n\geq n_0$ there is $j\in (n-N,n]$ with
\begin{equation}\label{eqn:countingG2}
\Lambda_j(\GGG,\ph) \geq Q_3 e^{jP(\ph)}.
\end{equation}
\end{lemma} 
\begin{proof}
For the first inequality in \eqref{eqn:countingL}, we start by observing that $\LLL_{kn} \subset \LLL_n \LLL_n \cdots \LLL_n$ ($k$ times), and so by iterating \eqref{eqn:CD} we get
\begin{equation}\label{eqn:submult}
\Lambda_{kn}(\ph) \leq \Lambda_n(\ph)^k,
\end{equation}
which yields
$
\frac 1{kn}\log\Lambda_{kn}(\ph) \leq \frac 1n \log\Lambda_n(\ph).
$
Sending $k\to\infty$ gives the first half of \eqref{eqn:countingL}.  Next we use 
\ref{spec} and \eqref{eqn:bdd-distortion} to prove that there is $C>0$, independent of $n$, such that
\begin{equation}
\label{eqn:countingG}
\Lambda_n(\GGG,\ph) \leq C e^{nP(\ph)}
\end{equation}
for every $n$; then we use \eqref{eqn:countingG} and  \ref{gap} to prove the second half of \eqref{eqn:countingL}.

By \ref{spec} there is a map $\pi\colon \GGG_m \times \GGG_n \to \GGG$ given by $\pi(v,w) = vuw$, where $u\in \LLL$ depends on $v,w$ but always satisfies $\abs{u}\leq \tau$.  Iterating and abusing notation slightly gives a map $\pi\colon (\GGG_n)^k \to \GGG$ of the form $\pi(v^1,\dots, v^k) = v^1 u^1 v^2 \cdots u^{k-1} v^k$.  This map may not be injective but we can control its multiplicity by observing that each $u^i$ has length in $\{0,\dots, \tau\}$, and so $\#\pi^{-1}(w) \leq (\tau+1)^{k-1}$.  Truncating $\pi(w)$ to the first $kn$ symbols gives a map $\hat\pi\colon (\GGG_n)^k \to \LLL_{nk}$; because we delete at most $k\tau$ symbols, the multiplicity of this truncation map is $\leq (\#A+1)^{k\tau}$.  Thus for each $w\in \LLL_{nk}$ we have $\#\hat\pi^{-1}(w)\leq (\tau+1)^{k-1} (\#A+1)^{k\tau}$.

Furthermore, \eqref{eqn:phvw} yields
\[
\hat\ph(v^1 u^1 \cdots u^{k-1} v^k) \geq \hat\ph(v^1) + \cdots + \hat\ph(v^k) - k(\tau\|\ph\| + \absd{\ph}),
\]
and since truncation deletes at most $k\tau$ symbols, we have
\[
\hat\ph(\hat\pi(v^1,\dots,v^k)) \geq \hat\ph(v^1) + \cdots + \hat\ph(v^k)
%\left(\sum_{i=1}^k \hat\ph(v^i)\right) 
- C' k
\]
for some $C'$ independent of $n,k, v^1,\dots, v^k$.  It follows that
\begin{align*}
\Lambda_{kn}(\ph) &\geq (\tau+1)^{-k}(\#A+1)^{-k\tau} e^{-C' k} \Lambda_n(\GGG,\ph)^k, \\
\frac 1{kn}\log\Lambda_{kn}(\ph) &\geq \frac 1n\log \Lambda_n(\GGG,\ph) - \frac 1n\big(C' + \log(\tau+1) + \tau\log (\#A+1)\big).
\end{align*}
Sending $k\to\infty$ gives \eqref{eqn:countingG}.

By condition \ref{gap} there is $\eps>0$ and $K > 0$ such that
\begin{align}
\label{eqn:LambdaCC}
\Lambda_n(\Cp \cup \Cs,\ph) &\leq K e^{n(P(\ph) - \eps)}, \\
\label{eqn:LambdaLC}
\Lambda_n(\LLL \setminus \Cp \GGG\Cs,\ph) &\leq K e^{n(P(\ph) - \eps)}
\end{align}
for all $n$.  From \eqref{eqn:LambdaLC} we get $\Lambda_n(\ph) \leq \Lambda_n(\Cp\GGG\Cs,\ph) + K e^{n(P(\ph) - \eps)}$, so it suffices to prove the upper bound in \eqref{eqn:countingL} for $\Lambda_n(\Cp\GGG\Cs,\ph)$.

Write $a_j = \Lambda_j(\GGG,\ph) e^{-jP(\ph)}$, and observe that $a_j\leq C$ by \eqref{eqn:countingG}.  Since every word $x\in (\Cp\GGG\Cs)_n$ can be decomposed as $x=uvw$ where $u\in \Cp$, $v\in \GGG$, and $w\in \Cs$, we have
\begin{equation}\label{eqn:ijk}
\begin{aligned}
\Lambda_n(&\Cp\GGG\Cs,\ph) \leq \sum_{i+j+k=n} \Lambda_i(\Cp,\ph) \Lambda_j(\GGG,\ph) \Lambda_k(\Cs,\ph) \\
&\leq K^2 \sum_{i+j+k=n} e^{i(P(\ph) - \eps)} a_j e^{jP(\ph)} e^{k(P(\ph) - \eps)} \\
&= K^2 e^{nP(\ph)} \sum_{i+j+k=n} a_j e^{-(i+k)\eps}
= K^2 e^{nP(\ph)} \sum_{m=0}^n a_{n-m} (m+1)e^{-m\eps}.
\end{aligned}
\end{equation}
Because $a_{n-m} \leq C$ and $\sum_{m\geq 0} (m+1)e^{-m\eps} < \infty$, this  proves the second half of \eqref{eqn:countingL}.

Finally, we use \eqref{eqn:LambdaLC} and \eqref{eqn:ijk} to show \eqref{eqn:countingG2}.  Note that it suffices to produce $j\in (n-N,n]$ with $a_j \geq Q_3$.  Using \eqref{eqn:LambdaLC}, \eqref{eqn:ijk} and the first half of \eqref{eqn:countingL}, we have
\[
e^{nP(\ph)} - K e^{n(P(\ph) - \eps)} \leq \Lambda_n(\Cp\GGG\Cs,\ph)
\leq
K^2 e^{nP(\ph)}\sum_{m=0}^n a_{n-m} (m+1)e^{-m\eps},
\]
which yields
$
\frac 12 \leq 1 - K e^{-n\eps} \leq K^2 \sum_{m=0}^n a_{n-m}(m+1) e^{-m\eps}
$
whenever $n\geq n_0$, where $n_0$ is chosen such that $K e^{-n_0\eps} < \frac 12$.  Thus we can use the inequality $a_{n-m} \leq Q_2$ to get
\[
\frac 12 K^{-2} \leq \sum_{m=0}^{N-1} a_{n-m}(m+1)e^{-m\eps} + \sum_{m\geq N} Q_2 (m+1)e^{-m\eps}.
\]
Let $N$ be large enough that $K' := \frac 12 K^{-2} - Q_2 \sum_{m \geq N} (m+1)e^{-m\eps} > 0$, and let $Q_3>0$ be small enough that $(m+1)e^{-m\eps} \leq K'/(NQ_3)$ for all $m\geq 0$.  Then
\[
\frac {K'}{NQ_3}\sum_{m=0}^{N-1} a_{n-m} \geq \sum_{m=0}^{N-1} a_{n-m} (m+1)e^{-m\eps} \geq 
K'
\quad\Rightarrow\quad
\frac 1N\sum_{m=0}^{N-1} a_{n-m} \geq Q_3,
\]
so there is $0\leq m<N$ with $a_{n-m} \geq Q_3$, which completes the proof of Lemma \ref{lem:counting}.
\iftoggle{arxiv}{}{\qed}
\end{proof}

Lemma \ref{lem:counting} leads to the following bound.  Given $v\in \LLL_k$ and $1\leq i\leq n - k$, consider 
\begin{equation}\label{eqn:Hnvi}
\HHH_n(v,i) = \{w\in \LLL_n \mid w_{[i,i+k)} = v \},
\end{equation}
the set of words where $v$ appears starting in index $i$, but the entries of $w$ before $i$ and after $i+k$ are free to vary (this is the finite-length analogue of a cylinder set).  We will mostly be interested in the case when $v\in \GGG$.  
We will also need to consider
\begin{equation}\label{eqn:GH+}
\GGG^+ = \{w_{[1,i]} : w\in \GGG, 1\leq i\leq \abs{w}\}
\quad\text{and}\quad
\HHH^+(v) = \GGG^+ \cap v\LLL.
\end{equation}
We have the following non-stationary version of the Gibbs property (for the measure-theoretic equivalent, see \cite[\S 5.2]{CT2} and \S\ref{sec:gibbs}).

\begin{proposition}\label{prop:Gibbs}
If $\GGG$ satisfies \ref{spec} and \ref{gap}, then there is $Q_4>0$ such that for every $1\leq i \leq i+k \leq n$, we have
\begin{equation}\label{eqn:Gibbs}
\begin{alignedat}{3}
\Lambda_n(\HHH_n(v,i),\ph) &\leq Q_4 e^{(n-k)P(\ph) + \hat\ph(v)} &\text{ for every $v\in \LLL_k$},\\
\Lambda_n(\HHH_n(v,i),\ph) &\geq Q_4^{-1} e^{(n-k)P(\ph) + \hat\ph(v)} &\text{ for every $v\in \GGG_k$},
\end{alignedat}
\end{equation}
and similarly, for every $1\leq k\leq n$, we have
\begin{equation}\label{eqn:Gibbs-2}
\Lambda_n(\HHH^+(v),\ph) \geq Q_4^{-1} e^{(n-k)P(\ph) + \hat\ph(v)} \text{ for every $v\in \GGG^+_k$}.
\end{equation}
\end{proposition}
\begin{proof}
For the upper bound in \eqref{eqn:Gibbs}, observe that \eqref{eqn:CD} gives
%arguing as in \eqref{eqn:submult} gives
\[
\Lambda_n(\HHH_n(v,i),\ph) \leq \Lambda_i(\LLL,\ph) e^{\hat\ph(v)} \Lambda_{n-(i+k)}(\LLL,\ph),
\]
and using \eqref{eqn:countingL} gives $\Lambda_n(\HHH_n(v,i),\ph) \leq (Q_2)^2 e^{(n-k)P(\ph)} e^{\hat\ph(v)}$.

\begin{figure}[htbp]
\includegraphics[width=200pt]{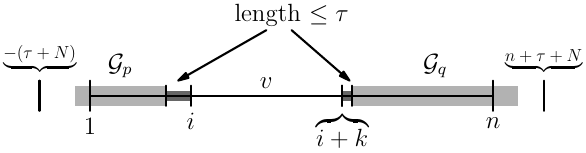}
\caption{Estimating $\Lambda_n(\HHH_n(v,i))$.}
\label{fig:gibbs}
\end{figure}

For the lower bound in \eqref{eqn:Gibbs}, we use the usual specification argument.  Let $N,n_0,Q_3$ be as in Lemma \ref{lem:counting}, and let $c = \max(n_0,\tau)$.  We start by considering the case when $i\geq c$ and $i+k \leq n - c$, illustrated in Figure \ref{fig:gibbs}.
%, illustrated in Figure \ref{fig:gibbs} for $\HHH_n(v,i)$ when $i\geq \tau$ and $i+k \leq n-\tau$; we consider this case first. 
By Lemma \ref{lem:counting}, for every such $i,k,n$, there are $p,q\in \NN$ such that
\begin{equation}\label{eqn:Lpq}
\begin{aligned}
\Lambda_p(\GGG,\ph)&\geq Q_3 e^{p P(\ph)} \text{ and } p \in [ i, i +N], \\
\Lambda_q(\GGG,\ph)&\geq Q_3 e^{q P(\ph)} \text{ and } i + k + q \in [n,n+N].
\end{aligned}
\end{equation}
Let $w^1\in \GGG_p$ and $w^2\in \GGG_q$ be arbitrary.  Then by \ref{spec} there are $u^1,u^2\in \LLL$ with $\abs{u^j}\leq \tau$ such that $w^1u^1vu^2w^2\in \GGG$.  Note that $\abs{w^1u^1} \in [i,i+N+\tau]$, and so by truncating at most $\tau+N$ symbols from the beginning and end of $w^1u^1vu^2w^2$, we obtain a word 
$T(w^1,w^2)\in \HHH_n(v,i)$ with the property that the first $i-\abs{u^1}$ symbols of $T(w^1,w^2)$ match the last $i-\abs{u^1}$ symbols of $w^1$, and similarly for the end of $T(w^1,w^2)$ and the beginning of $w^2$.

This defines a map $T\colon \GGG_p\times \GGG_q \to \HHH_n(v,i)$.  Note that
\begin{equation}\label{eqn:T-dist}
\begin{aligned}
\hat\ph(T(w^1,w^2)) &\geq \hat\ph(w^1 u^1 v u^2 w^2) - (2N+2\tau)\|\ph\| \\
&\geq \hat\ph(w^1) + \hat\ph(v) + \hat\ph(w^2) - (2N+4\tau)\|\ph\| - 2\absd{\ph},
\end{aligned}
\end{equation}
Moreover, since $\abs{u^1},\abs{u^2} \in \{0,\dots, \tau\}$ and  truncation removes at most $2(N+\tau)$ symbols from $w^1u^1vu^2w^2$,  each word in $\HHH_n(v,i)$ has at most $(\tau+1)^2(\#A+1)^{2(N+\tau)}$ preimages under the map $T$.  Thus there is $C>0$, independent of $n,v,i$, such that
\begin{align*}
\Lambda_n(\HHH_n(v,i)) &\geq (\tau+1)^{-2}(\#A+1)^{-2(N+\tau)} \sum_{w^1\in \GGG_p} \sum_{w^2 \in \GGG_q} e^{\hat\ph(T(w^1,w^2))} \\
&\geq C \sum_{w^1\in \GGG_p} \sum_{w^2 \in \GGG_q} e^{\hat\ph(w^1)}e^{\hat\ph(v)}e^{\hat\ph(w^2)} 
\geq C (Q_3)^2 e^{\hat\ph(v)} e^{-(n-k)P(\ph)},
\end{align*}
where the first inequality uses the multiplicity bound, the second uses \eqref{eqn:T-dist}, and the third uses \eqref{eqn:Lpq}.  This proves the lower bound in \eqref{eqn:Gibbs} as long as $i\geq c$ and $i+k\leq n-c$.  If $i<c$ then we dispense with $\GGG_p$ and define $T\colon \GGG_q \to \HHH_n(v,i)$ by $T(w) = (u^1 v u^2 w)_{[1,n]}$, where $u^2$ comes from \ref{spec} and $u^1\in \LLL_{i-1}$ is any word such that $u^1vu^2w\in \LLL$.  The rest of the proof proceeds as above, and the case $i+k > n-c$ is similar.  Finally, the lower bound in \eqref{eqn:Gibbs-2} for $\Lambda_n(\HHH^+(v),\ph)$ proceeds as in the case $i=1$, taking $u^1$ to be the empty word and noting that $T(w) = (vu^2w)_{[1,n]}\in \GGG^+$ since $vu^2w\in \GGG$ by \ref{spec}, and hence $T(w) \in \HHH^+(v)$.
\iftoggle{arxiv}{}{\qed}
\end{proof}

We need one more counting estimate that we will use in the proof of Theorem \ref{thm:tower}, which strengthens \eqref{eqn:countingG2} when \ref{spec} is replaced with \ref{free}.

\begin{lemma}\label{lem:gcd-counting}
Suppose that $\FFF$ satisfies \ref{free} and \ref{gap}, and let $d = \gcd\{\abs{w} \mid w\in \FFF\}$.
Then there is $Q_5>0$ such that $\Lambda_{nd}(\FFF,\ph) \geq Q_5 e^{ndP(\ph)}$ for all sufficiently large $n$.
\end{lemma}
\begin{proof}
Replacing $\sigma$ with $\sigma^d$, we assume without loss of generality that $\gcd\{\abs{w} \mid w\in \FFF\}=1$, so there is $m\in \NN$ such that for every $n\geq m$ we have $n = \sum_{i=1}^k a_i \abs{w^i}$ for some $a_i\in \NN$ and $w^i\in \FFF$.  Write $(w^i)^{a_i}$ for the word $w^i$ repeated $a^i$ times and note that by \ref{free} we have $w := (w^1)^{a_1} \cdots (w^k)^{a_k} \in \FFF$, and $\abs{w} = n$.  Thus $\FFF_n$ is non-empty for every $n\geq m$.

Now by \eqref{eqn:countingG2} there are $n_0,N\in \NN$ and $Q_3>0$ such that for every $n\geq n_0 + m$ there is $j\in (n-m-N,n-m]$ with 
\begin{equation}\label{eqn:jmN}
\Lambda_j(\FFF,\ph) \geq Q_3 e^{jP(\ph)} \geq Q_3 e^{-(m+N)P(\ph)} e^{nP(\ph)}.
\end{equation}
Since $n-j \in [m, m+N)$, by the definition of $m$ there is $w\in \FFF_{n-j}$; note that $\abs{\hat\ph(w)} \leq (n-j)\|\ph\| \leq (m+N)\|\ph\|$.
Now we can use \ref{free} to get
\[
\Lambda_n(\FFF,\ph) \geq \sum_{v\in \FFF_j} e^{\hat\ph(vw)} 
\geq e^{-\hat\ph(w) - \absd{\ph}} \sum_{v\in \FFF_j} e^{\hat\ph(v)} 
\geq e^{-\|\ph\|(m+N)- \absd{\ph}}  \Lambda_j(\FFF,\ph) ,
\]
where the second inequality uses the first half of \eqref{eqn:phvw}.  Together with \eqref{eqn:jmN}, this completes the proof of Lemma \ref{lem:gcd-counting}.
\iftoggle{arxiv}{}{\qed}
\end{proof}

\section{Proof of Theorem \ref{thm:tower} and Corollary \ref{cor:1-sided}}\label{sec:towers2}

%In this section we prove Theorems \ref{thm:tower} and \ref{thm:spr}.  Without loss of generality we assume that $X$ is a two-sided shift space; if it is one-sided then we pass to the natural extension and define $\ph$ to depend only on non-negative coordinates.  We will prove the conclusions for both the one-sided and two-sided shifts (see \eqref{eqn:DeltaX}).

%Now we prove Theorem \ref{thm:tower}.
Throughout this section,  $X$ will be a two-sided shift space, $\LLL$ its language, and $\FFF\subset \LLL$ will satisfy \ref{free}, so that
$I = \FFF \setminus \FFF\FFF$ has $I^* = \FFF \subset \LLL$.  Let $\Sigma$ be the countable-state Markov shift constructed in \S\ref{sec:subsidiary-results}, $T\colon \Sigma\to \Sigma$ the shift map, and $\pi\colon \Sigma\to X$ the one-block code given there.
We will denote a typical element of $\Sigma$ by $\zz = \{\zz_j\}_{j\in \ZZ}$, where each $\zz_j$ is of the form $(w,k)$ for some $w\in I$, $1\leq k\leq \abs{w}$. 

The list of items to prove in Theorem \ref{thm:tower} was illustrated in \eqref{eqn:structure}.  In \S\ref{sec:F-marking} we prove \ref{overlaps-injective} and \ref{injective-decipherable}; in \S\ref{sec:correspondence} we prove \ref{decipherable-spr}; in \S\ref{sec:liftability} we prove \ref{charges}; in \S\ref{sec:decipherable-lifts} we prove \ref{decipherable-finite} and \ref{finite-lifts}; in \S\ref{sec:getting-gibbs} we prove \ref{D-gap} and \ref{gap-gibbs}.  Corollary \ref{cor:1-sided} is proved in \S\ref{sec:1-sided-corollary}.

\subsection{$\FFF$-marking sets, injectivity of $\pi$, and unique decipherability}\label{sec:F-marking}
\label{sec:injectivity}

\begin{proposition}\label{prop:11}
If $\FFF$ satisfies \ref{overlaps}, then $\pi\colon \Sigma\to X$ is injective.  %If $\pi$ is injective, then $I$ is uniquely decipherable (whether or not $\FFF$ satisfies \ref{overlaps}). 
\end{proposition}

\begin{proposition}\label{prop:unique-inj}
The set $I$ is uniquely decipherable if and only if $\pi$ is 1-1 on $\Per(\Sigma,a)$ for every $a\in A_I$.
\end{proposition}

This section proves Propositions \ref{prop:11} and \ref{prop:unique-inj}, which are \ref{overlaps-injective} and \ref{injective-decipherable}, respectively.
First we set up some terminology.  Say that a (finite or infinite) set $J\subset \ZZ$ is \defn{$\FFF$-marking} for $x\in X$ if $x_{[i,j)}\in \FFF$ for all $i,j\in J$ with $i<j$.  Call $i,j\in J$ \defn{consecutive} if $k\notin J$ for all $k$ between $i$ and $j$.  We record some immediate consequences of \ref{free} as a lemma.

\begin{lemma}\label{lem:G-marking}
$J\subset \ZZ$ is $\FFF$-marking for $x$ if and only if $x_{[i,j)}\in \FFF$ for all consecutive indices $i<j$ in $J$.  In particular, the following are equivalent.
\begin{enumerate}[leftmargin=*, widest=3, label=\textup{(\arabic{*})}]
\item $J$ is $\FFF$-marking.
\item There are $a_k\to-\infty$ and $b_k\to\infty$ such that $J\cap [a_k,b_k]$ is $\FFF$-marking for every $k$.
\item $J\cap [a,b]$ is $\FFF$-marking for every $a,b\in \ZZ$ with $a<b$.
\end{enumerate}
\end{lemma}

Say that $J\subset \ZZ$ is \defn{bi-infinite} if $J\cap [0,\infty)$ and $J\cap (-\infty,0]$ are both infinite.  Say that $J\subset \ZZ$ is \defn{maximally $\FFF$-marking} for $x$ if there is no $\FFF$-marking set $J'\subset \ZZ$ with $J'\supsetneqq J$.  Recall that $I$ is the collection of irreducible elements of $\FFF$.  The following lemma collects properties of bi-infinite $\FFF$-marking sets, and relates these to the multiplicity of the map $\pi$.
The proofs are immediate from the definitions.

\begin{lemma}\label{lem:marking-sets}\hfill
\begin{enumerate}[leftmargin=*, widest=d,
label=\textup{(\alph{*})}]%,ref=\textup{(\arabic{enumi}.\alph{*})}]
\item \label{lem:I-marking}
A bi-infinite set $J\subset \ZZ$ is maximally $\FFF$-marking for $x$ if and only if $x_{[i,j)}\in I$ for all consecutive indices $i<j$ in $J$.
\item \label{lem:Jz}
Given $\zz\in \Sigma$, the set $J(\zz) := \{j \mid \zz_j = (w,1) \text{ for some }w\in I\}$ is bi-infinite and maximally $\FFF$-marking for $\pi(\zz)\in X$.
\item \label{lem:J-z}
If $J\subset \ZZ$ is bi-infinite and maximally $\FFF$-marking for $x\in X$, then there is exactly one $\zz\in \Sigma$ such that $\pi(\zz)=x$ and $J(\zz)=J$.
\item\label{lem:piSigma}
Given $x\in X$, we have $x\in \pi(\Sigma)$ 
if and only if there is a bi-infinite $\FFF$-marking set $J\subset \ZZ$ for $x$.  There is a 1-1 correspondence between elements of $\pi^{-1}(x) \subset \Sigma$ and bi-infinite maximal $\FFF$-marking sets for $x$.
\end{enumerate}
\end{lemma}

As \ref{lem:piSigma} clarifies, Lemma \ref{lem:marking-sets}\ref{lem:J-z} does not yet prove injectivity of $\pi$, since it is a priori possible that some $x\in X$ has multiple maximally $\FFF$-marking sets. 
To show injectivity of $\pi$ it suffices to show that every $x\in X$ has at most one bi-infinite maximally $\FFF$-marking set.  We accomplish this by showing that arbitrary unions of bi-infinite $\FFF$-marking sets are still bi-infinite and $\FFF$-marking.  This is where we need \ref{overlaps}.

\begin{lemma}\label{lem:union-good}
Suppose $\FFF$ satisfies \ref{overlaps}, and let $\{J_\lambda\}_{\lambda\in \Lambda}$ be any collection of sets $J_\lambda \subset \ZZ$ such that each $J_\lambda$ is $\FFF$-marking for $x$.   Let $r<s$ be integers such that $r > \min J_\lambda$ and $s< \max J_\lambda$ for all $\lambda$.
Then $\left(\bigcup_\lambda J_\lambda\right)\cap [r,s]$ is $\FFF$-marking for $x$.
\end{lemma}
\begin{proof}
Pick $j<k$ in $\bigcup_\lambda J_\lambda$ with $r\leq j<k\leq s$.  Let $\lambda,\lambda'$ be such that $j\in J_{\lambda}$ and $k\in J_{\lambda'}$.  Because $\min J_{\lambda'} \leq r$ there is $i\in J_{\lambda'}$ with $i\leq j$; similarly, $\max J_\lambda\geq s$ implies that there is $\ell\in J_\lambda$ with $\ell \geq k$.  Thus $i\leq j < k \leq \ell$ are such that $x_{[i,k)}, x_{[j,\ell)}\in \FFF$; see Figure \ref{fig:union-good}.
Moreover, choosing $a\in J_\lambda$ with $a<r$ and $b\in J_{\lambda'}$ with $b>s$, we have $x_{[a,j)},x_{[k,b)}\in \FFF$, and  it follows from Condition \ref{overlaps} that $x_{[j,k)}\in \FFF$.
This holds for all $j,k \in (\bigcup_\lambda J_\lambda)\cap [r,s]$, so we are done.
\iftoggle{arxiv}{}{\qed}
\end{proof}

\begin{figure}[htbp]
\includegraphics[width=150pt]{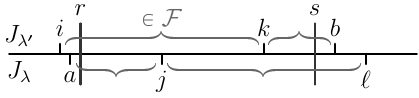}
\caption{The union of $\FFF$-marking sets is $\FFF$-marking.}
\label{fig:union-good}
\end{figure}

\begin{lemma}\label{lem:bi-inf-union-good}
If $\FFF$ satisfies \ref{overlaps} and $\{J_\lambda \subset \ZZ\}_{\lambda\in \Lambda}$ is a collection of bi-infinite $\FFF$-marking sets for $x$, then $\bigcup_\lambda J_\lambda$ is a bi-infinite $\FFF$-marking set for $x$.
\end{lemma}
\begin{proof}
By Lemma \ref{lem:union-good}, $J\cap [a,b]$ is $\FFF$-marking for all integers $a<b$, so by Lemma \ref{lem:G-marking}, $J$ is $\FFF$-marking for $x$.
\iftoggle{arxiv}{}{\qed}
\end{proof}
%When the sets $J_\lambda$ in Lemma \ref{lem:union-good} are all bi-infinite, we can %take $r\to-\infty$ and $s\to\infty$ in Lemma \ref{lem:G-marking}.

%We conclude that if $\FFF$ has \ref{overlaps} and  $\{J_\lambda\}_{\lambda\in \Lambda}$ is any collection of bi-infinite $\FFF$-marking sets $J_\lambda \subset \ZZ$ for $x\in X$, then $\bigcup_\lambda J_\lambda$ is bi-infinite and $\FFF$-marking for $x$.
Now we can prove that $\pi$ is injective, as follows: given $x\in \pi(\Sigma)$, let $\{J_\lambda\}_{\lambda\in \Lambda}$ be the collection of \emph{all} bi-infinite $\FFF$-marking sets for $x$.  This collection is non-empty by Lemma \ref{lem:marking-sets}\ref{lem:Jz}.  By Lemma \ref{lem:bi-inf-union-good}, $J = \bigcup_\lambda J_\lambda$ is bi-infinite and $\FFF$-marking for $x$.  Moreover, if $J'$ is any bi-infinite $\FFF$-marking set, we have $J'\subset J$ by construction, so $J$ is maximal, and it is the only bi-infinite maximal $\FFF$-marking set.  Thus $\pi$ is 1-1, which proves Proposition \ref{prop:11}.

Proposition \ref{prop:unique-inj} is a consequence of the following.

\begin{lemma}\label{lem:Pi}
For each $a=(v,i)\in A_I$, the following are equivalent.
\begin{enumerate}[leftmargin=*, widest=3, label=\textup{(\arabic{*})}]
\item\label{I} $I$ is uniquely decipherable.
\item\label{Pi} The map $\Pi_n\colon \Per_n(\Sigma,a) \to \FFF_{n-\abs{v}}$ given by $\Pi_n(\zz) = \pi(\zz)_{(\abs{v}-i, n-i]}$ is 1-1 for all $n>\abs{v}$.
\item\label{pi} $\pi$ is injective on $\Per(\Sigma,a)$.
\end{enumerate}
\end{lemma}
\begin{proof}
We start by proving that \ref{I} and \ref{Pi} are equivalent.  Define $\iota \colon I\to A_I^*$ by
$\iota(w) = (w,1)\cdots (w,\abs{w})$, and let $G_\Sigma = \{ \iota(w) : w\in I \}$.  Observe that $G_\Sigma^*$ is contained in the language of $\Sigma$, and the map $\Pi_n$ is the composition of the following two maps.
\begin{alignat*}{3}
\Per_n(\Sigma,a) &\to (G_\Sigma^*)_{n-\abs{v}} \qquad\qquad   & (G_\Sigma^*) &\to \FFF = I^* \\
\zz &\mapsto \zz_{(\abs{v} - i, n-i]} & \iota(w^1) \cdots \iota(w^k) &\mapsto w^1\cdots w^k
\end{alignat*}
The first of these is always injective since every $\zz\in \Per_n(\Sigma,a)$ is completely determined by $\zz_{(-i,n-i]}$ and satisfies $\zz_{(-i,\abs{v}-i]} = \iota(v)$.  The second map is injective if and only if $I$ is uniquely decipherable; indeed, this is the definition of unique decipherability.  Thus \ref{I} and \ref{Pi} are equivalent.

For the equivalent of the last two conditions, first observe that if \ref{Pi} fails then \ref{pi} fails.  Conversely, if $\yy,\zz\in \Per(\Sigma,a)$ have $\yy\neq\zz$ and $\pi(\yy)=\pi(\zz)$, then taking $n>\abs{v}$ to be a common multiple of their least periods, we get $\yy,\zz\in \Per_n(\Sigma,a)$ and $\Pi_n(\yy)=\Pi_n(\zz)$.
\end{proof}

\subsection{Strong positive recurrence}\label{sec:correspondence}

\begin{proposition}\label{prop:spr}
If  $\FFF$ satisfies \ref{gap-2} and
$I = \FFF\setminus \FFF\FFF$ is uniquely decipherable, then $\Phi = \ph\circ\pi\in \Ch(\Sigma)$ satisfies $P_G(\Phi) = P(\ph)$  and is strongly positive recurrent.
\end{proposition}

In this section we prove Proposition \ref{prop:spr}, which is \ref{decipherable-spr}.
From \eqref{eqn:Holder} we have $\abs{\ph(x) - \ph(y)} \leq \abs{\ph}_\beta e^{-\beta n}$ whenever $x,y\in X$ have $x_k = y_k$ for all $\abs{k}\leq n$.  
%Define $\Phi\colon \Sigma\to \RR$ by $\Phi = \ph\circ \pi$; 
Then for every $\zz,\zz'\in \Sigma$ with $\zz_k = \zz'_k$ for all $\abs{k}\leq n$, we have $\pi(\zz)_k = \pi(\zz')_k$ for all $\abs{k}\leq n$ (since $\pi$ is a one-block code) and hence $\abs{\Phi(\zz) - \Phi(\zz')} \leq \abs{\ph}_\beta e^{-\beta n}$.  In particular, $\Phi\in \Ch(\Sigma)$.  

%Fix $v\in I$ and let $\aa = (v,1) \in A_I$; recall the definition of $Z_n(\Phi,\aa)$ and $Z_n^*(\Phi,\aa)$ in \eqref{eqn:Zn}.  To prove Proposition \ref{prop:spr}, we use unique decipherability to prove injectivity of $\pi$ on $\Per(\Sigma,\aa)$, which lets us relate $Z_n$, $Z_n^*$ to partition sums on $X$; then we assume  $P(I,\ph)<P(\ph)$ and use this to prove \eqref{eqn:SPR}.

\begin{lemma}\label{lem:Z-Lambda}
If $I$ is uniquely decipherable, then there is $Q_6>0$ such that for every $\aa = (v,i) \in A_I$ and $n> \abs{v}$ we have
\begin{align}
\label{eqn:ZnF}
e^{-Q_6} \Lambda_{n-\abs{v}}(\FFF,\ph) &\leq Z_n(\Phi,\aa) \leq e^{Q_6} \Lambda_{n-\abs{v}}(\FFF,\ph), \\
\label{eqn:Zn*F}
e^{-Q_6} \Lambda_{n-\abs{v}}((I\setminus\{v\})^*,\ph) &\leq Z_n^*(\Phi,\aa) \leq e^{Q_6} \Lambda_{n-\abs{v}}((I\setminus\{v\})^*,\ph).
\end{align}
\end{lemma}
\begin{proof} Given $n>\abs{v}$, let $\Pi_n$ be as in Lemma \ref{lem:Pi}, and note that $\Pi_n$ is a bijection between $\Per_n(\Sigma,a)$ and $\FFF_{n-\abs{v}}$ for all $n>\abs{v}$.  Let $\tau = \Pi_n^{-1}$; then
\begin{equation}\label{eqn:Zn-Lambdan}
Z_n(\Phi,\aa) = \sum_{w\in \FFF_{n-\abs{v}}} e^{S_n\Phi(\tau(w))},\qquad
Z_n^*(\Phi,\aa) = \sum_{w\in ((I\setminus \{v\})^*)_{n-\abs{v}}} e^{S_n\Phi(\tau(w))}.
\end{equation}
For all $x,y\in [w]$ and $w\in \LLL_n$, \eqref{eqn:bdd-distortion} gives
$\abs{S_n\ph(x) - S_n\ph(y)} \leq \absd{\ph}$,
so each $w\in \FFF_n$ has
\[
\abs{S_n\Phi(\tau(w)) - \hat\ph(w)} \leq \absd{\ph} + \abs{v}\|\ph\|.
\]
Along with \eqref{eqn:Zn-Lambdan}, this completes the proof of Lemma \ref{lem:Z-Lambda}.
\iftoggle{arxiv}{}{\qed}
\end{proof}

It follows from \eqref{eqn:ZnF} and Lemma \ref{lem:CD} that $P_G(\Phi) = P(\FFF,\ph) = P(\ph)$.
By \eqref{eqn:SPR} and Lemma \ref{lem:Z-Lambda}, in order to prove that $\Phi$ is strongly positive recurrent, which will complete the proof of Proposition \ref{prop:spr}, it suffices to show that $P((I')^*,\ph) < P(\ph)$ whenever $I' \subsetneqq I$.  This uses the following result; note that $P(I,\ph) < P(\ph) = P(I^*,\ph)$ by \ref{gap-2} and Lemma \ref{lem:CD}.

\begin{lemma}\label{lem:hat-I}
If $I\subset \LLL$ is uniquely decipherable such that $I^*\subset \LLL$ and $P(I,\ph)<P(I^*,\ph)$, then for every $I' \subsetneqq I$ we have $P((I')^*,\ph) < P(I^*,\ph)$.
\end{lemma}
\begin{proof}
Given $w\in I^*$, let $0 = j_0 < j_1 < \cdots < j_\ell < j_{\ell+1} = \abs{w}$ be such that $w_{(j_i,j_{i+1}]} \in I$ for all $0\leq i\leq \ell = \ell(w)$.
 %\footnote{For purposes of this lemma we do not need unique decipherability, so the sequence $j_i$ may not be uniquely defined; we can select any sequence that does the job.  Note that the application of the lemma to obtain strong positive recurrence does require unique decipherability in order to use Lemma \ref{lem:Z-Lambda}.}  
Given $\ell,n\in \NN$ with $\ell<n$,
let $\JJ_\ell = \{J \subset [1,n) \mid \#J = \ell\}$; for each $J\in \JJ_\ell$, let
\[
\XXX_n(J) = \{w\in (I^*)_n \mid \ell(w) = \ell \text{ and } \{j_i(w)\}_{i=1}^\ell = J\}.
\]
Given $\delta>0$, let $\RRR_\delta = \{w\in I^* \mid \#\ell(w) \geq \delta \abs{w}\}$; we will prove Lemma \ref{lem:hat-I} by showing that for sufficiently small values of $\delta$, we have
\begin{align}
\label{eqn:I*R}
P(I^*\setminus \RRR_\delta,\ph) &< P(I^*,\ph), \\
\label{eqn:hat-I-R}
P((I')^* \cap \RRR_\delta,\ph) &< P(I^*,\ph),
\end{align}
and then applying \eqref{eqn:PCD} to $(I')^* \subset ((I')^* \cap \RRR_\delta) \cup (I^*\setminus \RRR_\delta)$.  To prove \eqref{eqn:I*R}, we start by writing $(I^* \setminus \RRR_\delta)_n = \bigcup_{\ell = 0}^{\lfloor \delta n\rfloor} \bigcup_{J\in \JJ_\ell} \XXX_n(J)$, so that
\begin{equation}\label{eqn:Lambda-n-I-R}
\Lambda_n(I^*\setminus \RRR_\delta) \leq \sum_{\ell=0}^{\lfloor \delta n\rfloor} \sum_{J\in \JJ_\ell} \Lambda_n(\XXX_n(J),\ph).
\end{equation}
To get \eqref{eqn:I*R}, we will get upper bounds on $\#\JJ_\ell$ and on $\Lambda_n(\XXX_n(J),\ph)$.  For the first of these, we observe that $\#\JJ_\ell\leq \binom{n}{\ell}$ and use the following useful result.

\begin{lemma}\label{lem:entropy-bound}
Given $\delta\in (0,1)$, write $h(\delta) = -\delta\log \delta - (1-\delta)\log(1-\delta)$ for the standard entropy function.  Then for every $n\in \NN$ and $0\leq \ell\leq n$, we have
$\binom{n}{\ell} \leq n e^{h(\frac\ell n) n}$.
\end{lemma}
\begin{proof}
First note that $\int_1^k \log t\,dt \leq \sum_{j=1}^k \log j = \log(k!) \leq \int_1^{k} \log t\,dt + \log k$.  Evaluating the integrals gives $k\log k - k + 1 \leq \log(k!) \leq k\log k - k + 1 + \log k$, and so
%We use the bound
%$k\log k - k + 1 \leq \log(k!) \leq k\log k - k + 1 + \log(k+1),$
%which can be obtained by integrating $\log t$ over $[1,k]$ (for the lower bound) and over $[1,k+1]$ (for the upper). This gives
\begin{equation}\label{eqn:log-n-choose-l}
\begin{aligned}
\log \textstyle\binom{n}{\ell} &= \log(n!) - \log(\ell!) - \log(n-\ell)! \\
&\leq n\log n +  \log n - \ell\log \ell - (n-\ell)\log(n-\ell) 
= h(\tfrac\ell{n}) n + \log n,
\end{aligned}
\end{equation}
which proves the lemma.
\iftoggle{arxiv}{}{\qed}
\end{proof}

Given $\delta\in (0,\frac 12)$ and $0\leq \ell \leq \delta n$, we conclude from Lemma \ref{lem:entropy-bound} that $\#\JJ_\ell \leq ne^{h(\delta) n}$.
To bound $\Lambda_n(\XXX_n(J),\ph)$, fix $\eps>0$ such that $P(I,\ph) < P(I^*,\ph) - 2\eps$; then there is $K$ such that $\Lambda_j(I,\ph) \leq K e^{j(P(I^*,\ph) - \eps)}$ for all $j$, and so
\[
\Lambda_n(\XXX_n(J),\ph) \leq \prod_{i=0}^{\ell} \Lambda_{j_{i+1}-j_i}(I,\ph) \leq K^{\ell+1} e^{n(P(I^*,\ph) - \eps)}.
\]
Together with \eqref{eqn:Lambda-n-I-R} and the bound on $\#\JJ_\ell$, this gives
\[
\Lambda_n(I^*\setminus \RRR_\delta) \leq (\delta n + 1 )n e^{h(\delta) n } K^{\delta n + 1} e^{n(P(I^*,\ph) - \eps)},
\]
and \eqref{eqn:I*R} follows by taking $\delta$ small enough that $h(\delta) + \delta \log(K) < \eps$.

To prove \eqref{eqn:hat-I-R}, we will consider for each $k\geq 0$ the collection $\AAA^k = ((I')^k I^*) \cap \RRR_\delta$.  By unique decipherability, each $w\in \AAA^k$ uniquely determines $u^1,\dots, u^k\in I'$, $\ell\geq k+1$, and $u^{k+1},\dots, u^\ell \in I$ such that $w = u^1\cdots u^\ell$.
It follows from the definition that $\AAA^{k+1} \subset \AAA^k$ and $((I')^* \cap \RRR_\delta)_n \subset \AAA_n^k$ for all $k\leq {\lfloor \delta n\rfloor}$, so we can estimate $\Lambda_n((I')^* \cap \RRR_\delta,\ph)$ from above by estimating $\Lambda_n(\AAA_n^k \setminus \AAA_n^{k+1},\ph)$.  Let $d = \gcd\{\abs{u} \mid u \in I\}$; note that $((I')^*)_n = \emptyset$ whenever $n$ is not a multiple of $d$, so we can restrict our attention to the case when $n$ is a multiple of $d$.

Fix $v\in I\setminus I'$.  Given $u\in (I')^k$ with $|uv|\leq n$, we have $(uv(I^*)_{n-\abs{uv}}) \cap \RRR_\delta \subset \AAA_n^k \setminus \AAA_n^{k+1}$.\footnote{This requires unique decipherability to guarantee that $uvw\notin \AAA_n^{k+1}$ for every $w\in (I^*)_{n-\abs{uv}}$.}  Taking the union over all such $u$ gives
\begin{equation}\label{eqn:Ank}
\Lambda_n(I^* \setminus \RRR_\delta,\ph) + \Lambda_n(\AAA^k\setminus \AAA^{k+1},\ph) 
\geq \sum_{u\in ((I')^k)_{\leq n-\abs{v}}} \Lambda_n(uv(I^*)_{n-\abs{uv}},\ph),
\end{equation}
and we observe that
\[
\Lambda_n(uv(I^*)_{n-\abs{uv}},\ph)
\geq\sum_{w\in (I^*)_{n-\abs{uv}}} e^{\hat\ph(u) + \hat\ph(v) + \hat\ph(w) - 2\absd{\ph}}
=
e^{\hat\ph(v) - 2\absd{\ph}}  e^{\hat\ph(u)} \Lambda_{n-\abs{uv}}(I^*,\ph).
\]
Since we assumed that $n$ is a multiple of $d$, and the same is clearly true of $\abs{uv}$, Lemmas \ref{lem:counting} and \ref{lem:gcd-counting} give, for $n$ sufficiently large,
\[
\Lambda_{n-\abs{uv}}(I^*,\ph) \geq Q_5 e^{(n-\abs{uv})P(\ph)}
\geq Q_5 e^{-\abs{v}P(\ph)} Q_2^{-1} \Lambda_{n-\abs{u}}(I^*,\ph),
\]
so writing $\gamma = Q_5 e^{-\abs{v}P(\ph)} Q_2^{-1} e^{\hat\ph(v)- 2\absd{\ph}}$ gives
\begin{equation}\label{eqn:uvI*}
\Lambda_n(uv(I^*)_{n-\abs{uv}},\ph) \geq \gamma e^{\hat\ph(u)} \Lambda_{n-\abs{u}}(I^*,\ph).
\end{equation}
Moreover, for every $k \leq \lfloor \delta n \rfloor - \abs{v}$, we have
\begin{equation}\label{eqn:AnkII}
\AAA_n^k = \bigcup_{\ell\geq \delta n} \big( (I')^k I^{\ell-k}\big)_n 
\subset \bigcup_{\ell\geq \delta n} %\bigcup_{j\leq n-(\ell-k)} ((I')^k)_j (I^{\ell-k})_{n-j}
((I')^k)_{\leq n - (\ell - k)} (I^*)
\subset ((I')^k)_{\leq n-\abs{v}} (I^*),
\end{equation}
where the first inclusion uses the fact that each element of $I$ has length at least 1, and the last inclusion uses the fact that $\ell - k \geq \lfloor\delta n\rfloor - (\lfloor \delta n \rfloor - \abs{v}) = \abs{v}$.  Combining \eqref{eqn:Ank}, \eqref{eqn:uvI*}, and \eqref{eqn:AnkII},
for every $k\leq \lfloor \delta n \rfloor - \abs{v}$ we get
\[
\Lambda_n(I^* \setminus \RRR_\delta,\ph) +
\Lambda_n(\AAA^k \setminus \AAA^{k+1},\ph) \geq 
\sum_{u\in ((I')^k)_{\leq n - \abs{v}}} 
\sum_{w\in (I^*)_{n-\abs{u}}} \gamma e^{\hat\ph(u) }e^{\hat\ph(w)}
\geq \gamma \Lambda_n(\AAA^k,\ph).
\]
%where the last inequality uses the fact that $\AAA_n^k \subset (I')^k_{\leq n-\abs{v}} I^*$ whenever $k + \abs{v} \leq \lfloor \delta n \rfloor$, since $\AAA_n^k \subset (\RRR_\delta)_n = \bigcup_{\ell \geq  \delta n } (I^\ell)_n$.  
We conclude that
\[
\Lambda_n(\AAA^{k+1},\ph) = \Lambda_n(\AAA^k,\ph) - \Lambda_n(\AAA^k \setminus \AAA^{k+1},\ph)
\leq (1-\gamma) \Lambda_n(\AAA^k,\ph) + \Lambda_n(I^*\setminus \RRR_\delta,\ph)
\]
for all $k\leq \lfloor \delta n \rfloor - \abs{v}$.
Using the fact that $\AAA_n^0 = (I^*)_n$ and $((I')^* \cap \RRR_\delta)_n \subset \AAA_n^{\lfloor \delta n\rfloor - \abs{v}}$, we get
\[
\Lambda_n((I')^* \cap \RRR_\delta, \ph) \leq \Lambda_n(\AAA^{\lfloor \delta n\rfloor - \abs{v}},\ph) \leq (1-\gamma)^{\lfloor \delta n \rfloor - \abs{v}}
\Lambda_n(I^*,\ph) + \delta n \Lambda_n(I^*\setminus \RRR_\delta,\ph).
\]
For any sequences $b_n,c_n>0$, we have $\ulim \frac 1n \log(b_n + c_n) \leq \max (\ulim \frac 1n \log b_n, \ulim \frac 1n \log c_n)$, and using \eqref{eqn:I*R} we conclude that
%Taking logs, dividing by $n$, and sending $n\to\infty$ gives
\[
P((I')^*\cap \RRR_\delta,\ph) \leq 
\max\big( P(I^*,\ph) +  \delta\log(1-\gamma), P(I^* \setminus \RRR_\delta, \ph) \big) < P(I^*,\ph),
\]
which proves \eqref{eqn:hat-I-R} and completes the proof of Lemma \ref{lem:hat-I}.
\iftoggle{arxiv}{}{\qed}
\end{proof}

\subsection{Approximate equilibrium states charge the tower}\label{sec:liftability}

\begin{proposition}\label{prop:charges}
If $\FFF$ satisfies \ref{gap-2} with $I = \FFF\setminus \FFF\FFF$, then the collection $\EEE' = I \cup \Ep \cup \Es \cup (\LLL\setminus \Ep\FFF\Es)$ has the following property: for any ergodic measure $\mu$ on $X$ with $h(\mu) + \int \ph\,d\mu > P(\EEE',\ph)$, we have $\mu(\pi(\Sigma))=1$.  
\end{proposition}

Note that Proposition \ref{prop:charges} establishes \ref{charges} by taking $P' = P(\EEE',\ph)$, since then \ref{gap-2} gives $P' < P(\ph)$.  To prove the proposition, we will need the following notion: say that $R\subset \ZZ$ is  \defn{$\EEE'$-restricting for $x\in X$} if for every $i<j\in \ZZ$ with $x_{[i,j)}\in \EEE'$, the interval $[i,j]$ contains at most one element of $R$.

First we prove in Lemma \ref{lem:restricting-marking} that $x\in \pi(\Sigma)$ whenever $x$ has a bi-infinite $\EEE'$-restricting set; then in Lemma \ref{lem:infinite-restricting} we show that $\mu$-a.e.\ $x\in X$ has such a set whenever $\mu$ is ergodic and  $h(\mu) + \int\ph\,d\mu > P(\EEE',\ph)$.  Together these will complete the proof of Proposition \ref{prop:charges}.

\begin{lemma}\label{lem:restricting-marking}
If $x\in X$ has a bi-infinite $\EEE'$-restricting set $R\subset \ZZ$, then it has a bi-infinite $\FFF$-marking set $J$, and hence $x\in \pi(\Sigma)$ by Lemma \ref{lem:marking-sets}\ref{lem:piSigma}.
\end{lemma}
\begin{proof}
Enumerate $R$ as $R = \{r_n\}_{n\in \ZZ}$ where $r_n$ is increasing (see Figure \ref{fig:infinite-restricting}).  Given $n\in \NN$, note that $x_{[r_{-n},r_n)} \in \Ep \FFF \Es$ since otherwise we would have a word in $\LLL \setminus \Ep \FFF \Es \subset \EEE'$ that crosses more than one index in $R$.  Thus there are
%we can  apply the decomposition $\Ep \FFF\Es$ to $x_{[r_{-n},r_n)}$ to get 
$j' \leq j''$ in $[r_{-n},r_n)$ such that
\[
x_{[r_{-n},j')}\in \Ep,\quad
x_{[j',j'')}\in \FFF,\quad
x_{[j'',r_n)}\in \Es.
\]
Now use the decomposition $\FFF = I^* = \bigcup_{k\in \NN} I^k$ to get an increasing sequence $\{j^n_i\}_{i=1}^{k+1}$ such that $j^n_1 = j'$, $j^n_{k+1} = j''$, and $x_{[j^n_i,j^n_{i+1})}\in I$ for every $1\leq i \leq k$.
Put $j^n_0 = r_{-n}$ and $j^n_{k+2} = r_n$.  Then  $x_{[j^n_i, j^n_{i+1})}\in \EEE'$ for each $0\leq i\leq k+1$.  
Since $R$ is $\EEE'$-restricting, we conclude that
\[
\# \{ \ell \in [-n,n) : r_\ell \in [j^n_i, j^n_{i+1}] \} \leq 1
\quad\text{for every } 0\leq i\leq k+1;
\]
in particular, if $i,\ell$ are such that $j_i^n \leq r_\ell \leq j_{i+1}^n$, then $j_{i+1}^n \in [r_\ell, r_{\ell+1})$.  Since $j_0^n =r_{-n}$ and $j_{k+2}^n = r_n$, it follows that for every 
$\ell\in [-n,n)$ the interval $[r_\ell, r_{\ell+1})$ contains at least one element of $J_n := \{j^n_i\}_{i=1}^{k+1}$.  Note that $J_n$ is $\FFF$-marking for $x$.  

\begin{figure}[htbp]
\includegraphics[width=190pt]{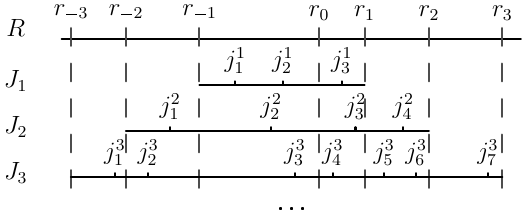}
\caption{Constructing a bi-infinite $\FFF$-marking set.}
\label{fig:infinite-restricting}
\end{figure}

By the previous paragraph, there is $s_0\in [r_0,r_1)$ such that the set $\mathbf{N}_0 := \{n\in \NN \mid s_0 \in J_n\}$ is infinite.  Similarly, there are $s_1\in [r_1,r_2)$ and $s_{-1}\in [r_{-1},r_0)$ such that $\mathbf{N}_1 := \{n\in \mathbf{N}_0 \mid s_1,s_{-1}\in J_n\}$ is infinite.  Continuing in this manner, we choose for each $\ell\in \NN$ two indices $s_\ell\in [r_\ell,r_{\ell+1})$ and $s_{-\ell}\in [r_{-\ell},r_{-\ell+1})$ such that $\mathbf{N}_\ell = \{n\in \mathbf{N}_{\ell-1} \mid s_\ell, s_{-\ell} \in J_n\}$ is infinite.  It follows from the definition of $J_n$ that $x_{[s_\ell,s_{\ell+1})} \in \FFF$ for every $\ell\in \ZZ$, so $J := \{s_\ell\}_{\ell\in \ZZ}$ is a bi-infinite $\FFF$-marking set for $x$.
\iftoggle{arxiv}{}{\qed}
\end{proof}

\begin{lemma}\label{lem:Cn-decay}
Let $\mu$ be an ergodic measure for $X$, and let $\DDD\subset \LLL$ be such that $P(\DDD,\ph) < h(\mu) + \int\ph\,d\mu$.  Then for $\mu$-a.e.\ $x\in X$, there is $n=n(x)\in\NN$ such that for all $k\geq n$ and all $\ell\in [0,k]$ we have $x_{[-\ell,k-\ell)}\notin \DDD$.  
Equivalently, the measure of the following sets decays to 0 as $n\to\infty$:%\footnote{If $\mu$ is known to satisfy the upper Gibbs bound, then a simple computation shows that in fact $\mu(B_n(\DDD))$ decays exponentially, but we will not need this.}
\begin{equation}\label{eqn:Bn}
B_n(\DDD) := \{x\in X \mid x_{[-\ell,k-\ell)} \in \DDD \text{ for some $k\geq n$ and $\ell\in [0,k]$}\}.
\end{equation}
\end{lemma}
\begin{proof}
%Without loss of generality we may assume that $\mu$ is ergodic.
Fix $\eps>0$ such that $h(\mu) + \int\ph\,d\mu - 5\eps > P(\DDD,\ph)$.   By Theorems \ref{thm:strong-birkhoff} and \ref{thm:strong-smb}, for $\mu$-a.e.\ $x\in X$ there is $N_x\in \NN$ such that for all $n\geq N_x$ and $\ell\in [0,n]$ we have
\begin{equation}\label{eqn:mu-Sn}
\mu[x_{[-\ell,-\ell+n)}] \leq e^{-n h(\mu) + n\eps},
\qquad
S_n\ph (\sigma^{-\ell}x) \geq n\left( \int\ph\,d\mu - \eps \right), 
\qquad
n\eps \geq \absd{\ph}.
\end{equation}
Let $A_n = \{x\in X \mid N_x \leq n\}$; then $\mu(A_n)\to 1$ as $n\to\infty$.  If $k\geq n$ and $w\in \LLL_k$ is such that $\sigma^\ell[w] \cap A_n \neq \emptyset$ for some $0\leq \ell\leq k$, then we can choose $x$ in the intersection, so that $x_{[-\ell,-\ell+k)} = w$ and $k\geq N_x$; then \eqref{eqn:mu-Sn} gives
\[
\mu[w]  \leq e^{-k h(\mu) + k\eps} \quad\text{and}\quad
\hat\ph(w) \geq (S_k\ph(\sigma^{-\ell}x)) - \absd{\ph} \geq k\left( \int\ph\,d\mu - 2\eps \right),
\]
so we get
\begin{equation}\label{eqn:muw}
\mu[w] e^{-\hat\ph(w)} \leq e^{-k (h(\mu) + \int\ph\,d\mu - 3\eps)} 
%= e^{-k (P(\ph) - 3\eps)}
\leq e^{-k P(\DDD,\ph)} e^{-2k\eps}.
\end{equation}
Summing over all long cylinders that intersect $A_n \cap B_n(\DDD)$ gives
\begin{multline*}
\mu(A_n \cap B_n(\DDD)) \leq \sum_{k\geq n} \sum_{w\in \DDD_k} \sum_{\ell=0}^k \mu(\sigma^\ell [w]\cap A_n) \\
\leq \sum_{k\geq n} (k+1) \sum_{w\in \DDD_k} e^{-kP(\DDD,\ph)} e^{-2k\eps} e^{\hat\ph(w)} 
= \sum_{k\geq n} (k+1)e^{-kP(\DDD,\ph)} e^{-2k\eps} \Lambda_k(\DDD,\ph).
\end{multline*}
Choose $C$ such that $\Lambda_k(\DDD,\ph) \leq C e^{k(P(\DDD,\ph)+\eps)}$ for all $k$.  Then
\[
\mu(A_n\cap B_n(\DDD)) \leq C \sum_{k\geq n} (k+1) e^{-k\eps} \to 0 \text{ as } n\to\infty,
\]
so
$
\mu(B_n(\DDD)) \leq \mu(X\setminus A_n) + \mu(A_n\cap B_n(\DDD)) \to 0 \text{ as } n\to\infty.
$
Thus $\mu(\bigcap_{n\in \NN} B_n(\DDD)) = 0$, and so for $\mu$-a.e.\ $x\in X$ there is $n>0$ with $x\notin B_n(\DDD)$, which proves Lemma \ref{lem:Cn-decay}.
\iftoggle{arxiv}{}{\qed}
\end{proof}

\begin{lemma}\label{lem:infinite-restricting}
Let $\mu$ be any ergodic measure on $X$ with $h(\mu) + \int\ph\,d\mu > P(\EEE',\ph)$.  Then $\mu$-a.e.\ $x\in X$ has a bi-infinite $\EEE'$-restricting set.  
\end{lemma}
\begin{proof}
Let $E$ be the set of points satisfying the conclusion of Lemma \ref{lem:Cn-decay} with $\DDD=\EEE'$; given $x\in E$ there is $n(x)$ such that for all $k\geq n(x)$ and all $\ell\in [0,k]$ we have $x_{[-\ell,k-\ell)}\notin \EEE'$.  By Lemma \ref{lem:Cn-decay} we have $\mu(X\setminus E)=0$, hence $\mu(\sigma^{-m}(X\setminus E))=0$ for every $m\in \ZZ$, and we conclude that $E' := \bigcap_{m\in \ZZ} \sigma^{-m} E$ has full $\mu$-measure.  For every $x\in E'$ and every $m\in \ZZ$ there is $n(m)\in \NN$ such that for all $a\leq m \leq b$ with $b-a \geq n(m)$, we have $x_{[a,b)}\notin \EEE'$.

Given $x\in E'$, define  $r_j \in \ZZ$ by
%\begin{enumerate}
%\item 
$r_0=0$;
%\item 
$r_{j+1} = r_j + n(r_j)$ for $j\geq 0$;
%\item 
and
$r_{j-1} = r_j - n(r_j)$ for $j\leq 0$.
%\end{enumerate}
 Let $R = \{ r_j \}_{j\in \ZZ}$, and note that $R$ is bi-infinite.  We claim that $R$ is $\EEE'$-restricting for $x$.  Note that by the construction of $R$, we have $r_{j+1} - r_j \geq \min(n(r_j),n(r_{j+1}))$ for every $j\in \ZZ$.  Thus if $a<b\in \ZZ$ are such that $a\leq r_j$ and $b\geq r_{j+1}$, we either have $b-a\geq n(r_j)$ or $b-a \geq n(r_{j+1})$.  It follows from the definition of $n$ that $x_{[a,b)} \notin\EEE'$, since $r_j,r_{j+1}\in [a,b]$, and we conclude that $R$ is $\EEE'$-restricting for $x$.
\iftoggle{arxiv}{}{\qed}
 \end{proof}

%This completes the proof of Proposition \ref{prop:charges}.
%Lemma \ref{lem:infinite-restricting} completes the proof of \ref{charges}.  Before proving liftability, 

\subsection{Finite multiplicity and liftability}\label{sec:decipherable-lifts}

\begin{proposition}\label{prop:finite-mult}
Suppose that $I$ is uniquely decipherable and $\FFF = I^*$ satisfies \ref{gap-2}.
Let $\EEE' = I \cup \Ep \cup \Es \cup (\LLL\setminus\Ep\FFF\Es)$, as in Proposition \ref{prop:charges},
and let $\mu$ be an ergodic measure on $X$ such that $h(\mu) + \int \ph\,d\mu > P(\EEE',\ph)$ and $h(\mu) + \int \ph\,d\mu > P((I')^*,\ph)$ for some $I' \subsetneqq I$.  Then $\#\pi^{-1}(x) < \infty$ for $\mu$-a.e.\ $x\in X$.  
%In particular,  every equilibrium state $\mu$ has the property that $\#\pi^{-1}(x) < \infty$ for $\mu$-a.e.\ $x\in X$.  
\end{proposition}

This section is devoted to Proposition \ref{prop:finite-mult}, which establishes \ref{decipherable-finite} by taking $P' = \max(P(\EEE',\ph), P((I')^*,\ph))$ (note that $P'<P(\ph)$ by \ref{gap-2} and Lemma \ref{lem:hat-I}), and Lemma \ref{lem:finite-lifts} below, which is \ref{finite-lifts}.  
Let $I'\subsetneqq I$ be as in the hypothesis, and fix $v\in I\setminus I'$.  Then $\aa = (v,1) \in A_I$ is a vertex in the graph giving $\Sigma$, and we write
\[
\EEE'' := \{ \pi(\zz_1 \cdots \zz_n) \mid \zz\in \Sigma, n\in \NN, \zz_i\neq \aa \text{ for all } 1\leq i\leq n\}
\]
for the collection of all words in $\LLL$ that can be lifted to a word in the language of $\Sigma$ that avoids $\aa$.  Observe that $\EEE'' \subset \EEE' (I')^* \EEE'$, and thus by Lemma \ref{lem:CD} we get
\[
P(\EEE'',\ph) \leq \max\{ P(\EEE',\ph), P((I')^*,\ph) \} < h(\mu) + \int\ph\,d\mu.
\]
Now consider the sets $C_N := B_N(\EEE'')$ as in \eqref{eqn:Bn}.
%$C_N = \{x \in X \mid x_{[-\ell,k-\ell)} \in \EEE''$ for some $k\geq N$ and $\ell\in [0,k]\}$, which are related to $\EEE''$ in the same way as \eqref{eqn:Bn} relates $B_n$ and $\EEE'$.  Note that the only property of $\EEE'$ used in Lemma \ref{lem:Cn-decay} was that $P(\EEE',\ph)<P(\ph)$, so 
Applying Lemma \ref{lem:Cn-decay} with $\DDD = \EEE''$, there is $N\in \NN$ such that $\mu(X\setminus C_N) > 0$.  Consider the set $E := \{x\in X \mid f^n(x)\in X\setminus C_N$ for infinitely many positive $n$ and infinitely many negative $n\}$.  If $\mu$ is ergodic, then Poincar\'e recurrence implies that $\mu(E)=1$, and so the proof of Proposition \ref{prop:finite-mult} is completed by the following lemma, the idea of which goes back to Bowen and Marcus \cite[p.\ 13--14]{rB78} (see also \cite[p.\ 229]{PP90} and \cite[Theorem 12.8]{oS13}).

\begin{lemma}\label{lem:short-restricting}
For every $x\in E$ we have $\#\pi^{-1}(x) \leq N^2$.
\end{lemma}
\begin{proof}
Suppose there are $N^2+1$ distinct points $\zz^1,\dots, \zz^{N^2+1} \in \Sigma$ such that $\pi(\zz^i)=x$ for every $1\leq i\leq N^2+1$.  Then there is $m\in \NN$ such that the words $\zz^i_{[-m,m]}$ are all distinct.  By the definition of $E$, there are $n_1,n_2\in \ZZ$ such that $n_1 \leq -m < m \leq n_2$ and $f^{n_1}(x), f^{n_2}(x)\in X\setminus C_N$.  By unique decipherability and the definition of $\Sigma$, for every pair of integers $k_1 \in (n_1 - N, n_1]$ and $k_2 \in [n_2, n_2+N)$ there is at most one $\zz\in \pi^{-1}(x)$ such that $\zz_{k_1} = \zz_{k_2} = \aa$.  In particular, from the collection $\zz^1,\dots, \zz^{N^2+1}$, there must be some $\zz^i\in \pi^{-1}(x)$ with the property that $\zz^i_k\neq \aa$ for every $k\in (n_1-N,n_1]$, or $\zz^i_k\neq \aa$ for every $k\in [n_2, n_2+N)$.  But then either  $\pi(\zz^i_{(n_1-N,n_1]}) \in \EEE''$ or $\pi(\zz^i_{[n_2,n_2+N)})\in \EEE''$, contradicting the assumption that $f^{n_1}(x),f^{n_2}(x)\in X\setminus C_N$.
\iftoggle{arxiv}{}{\qed}
\end{proof}

Part \ref{finite-lifts} of Theorem \ref{thm:tower} is the following well-known lemma; for a proof, see \cite[Proposition 13.2]{oS13}.\footnote{The result there is stated in the context of surface diffeomorphisms, but only uses the fact that $\mu$-a.e.\ point has at least one and at most finitely many preimages.}

\begin{lemma}\label{lem:finite-lifts}
If $\mu$ is an ergodic invariant measure on $X$ such that $\mu(\pi(\Sigma))=1$ and $\#\pi^{-1}(x) < \infty$ for $\mu$-a.e.\ $x\in X$, then there is an ergodic invariant measure $\nu$ on $\Sigma$ such that $\pi_*\nu = \mu$, and moreover $h(\nu) = h(\mu)$.
\end{lemma}

\subsection{Verifying the pressure gap and Gibbs property}\label{sec:getting-gibbs}

The remaining items in Theorem \ref{thm:tower} are \ref{D-gap} and \ref{gap-gibbs}, which we prove here as Propositions \ref{prop:D-gap} and \ref{prop:Gibbs-measure}, respectively.

\begin{proposition}\label{prop:D-gap}
If $I\subset \LLL$ is a generating set for $X$ for which $\DDD = \DDD(I) = \{w\in \LLL : \LLL w \LLL \cap I \neq \emptyset\}$ has the property that $P(\DDD,\ph) < P(\ph)$, then $\FFF=I^*$ satisfies \ref{gap-2}.
%$\{w\in \LLL \mid uwv \in I$ for some $u,v\in \LLL\}$
\end{proposition}
\begin{proof}
The collection $\DDD$ consists of all words that appear as subwords of a generator; in particular, $I \subset \DDD$.  
Let $\Ep = \Es = \DDD$; we claim that $\LLL \setminus \Ep\FFF\Es \subset \DDD\DDD \cup \DDD$,
so that $I \cup \Ep \cup \Es \cup (\LLL \setminus \Ep\FFF\Es) \subset \DDD\DDD \cup \DDD$, which will be enough to prove \ref{gap-2} since $P(\DDD\DDD \cup \DDD,\ph) = P(\DDD,\ph)$ by \eqref{eqn:PCD} and Lemma \ref{lem:CD}.
To prove the claim, we show that $\LLL \subset \DDD\FFF\DDD \cup \DDD\DDD \cup \DDD$.  Observe that since $X = \overline{I^\infty}$, for every $w\in \LLL$, the nonempty open set $[w]$ intersects $I^\infty$.  Writing $x\in [w] \cap I^\infty$, there is $(n_k)_{k\in \ZZ} \subset \ZZ$ such that $n_k<n_{k+1}$ and $v^k := x_{[n_k,n_{k+1})} \in I$ for all $k\in \ZZ$.  Let $j\leq k$ be integers such that $n_j \leq 1 < n_{j+1}$ and $n_k \leq \abs{w}+1 < n_{k+1}$.  Then $w = x_{[1,\abs{w}]}$ is a subword of $x_{[n_j,n_{k+1})} = w^j w^{j+1} \cdots w^k$.  If $k=j$ this implies that $w\in \DDD$; if $k=j+1$ then $w\in \DDD\DDD$, and if $k\geq j+2$ then
\[
w=x_{[1,n_{j+1})} w^{j+1} \cdots w^{k-1} x_{[n_k,\abs{w}]} \in \DDD \FFF \DDD,
\]
 which proves the claim and completes the proof of Proposition \ref{prop:D-gap}.
\iftoggle{arxiv}{}{\qed}
\end{proof}

Recall a standard construction of an equilibrium state:  write $\delta_x$ for the point mass at $x\in X$, then for each $w\in \LLL$, let $x(w)\in [w]$ be a point that maximises $S_{\abs{w}}\ph$, and consider the measures defined by
\begin{equation}\label{eqn:nun-mun}
\nu_n = \frac 1{\Lambda_n(\LLL,\ph)} \sum_{w\in \LLL_n} e^{\hat\ph(w)} \delta_{x(w)}, \qquad \qquad
\mu_n = \frac 1n \sum_{k=0}^{n-1} \sigma_*^k \nu_n.
\end{equation}

\begin{proposition}\label{prop:Gibbs-measure}
Every weak* limit point $\mu$ of the sequence $\{\mu_n\}$ is an equilibrium state for $\ph$, and  has the Gibbs property \eqref{eqn:gibbs} for $\ph$ on $\FFF$.
\end{proposition}
\begin{proof}
The first claim is shown in the proof of
%Let $\mu$ be a weak* limit point of the sequence ${\mu_n}$. It is shown in~
\cite[Theorem 9.10]{pW82}.
% that $h(\mu) + \int \ph\,d\mu = P(\ph)$, so $\mu$ is an equilibrium state.
%We can obtain the (non-uniform) Gibbs property for $\mu$ using Proposition \ref{prop:Gibbs}: the following result mimics \cite[Proposition 5.5 and Lemma 5.6]{CT2}, which are formally mildly weaker.
For the second, first fix $n\in \NN$ and $w\in \LLL_n$.  Then for $m > n$ and $1\leq k<m-n$, we have
\[
\nu_m(\sigma^{-k}[w]) = \frac{\Lambda_m(\HHH_m(w,k),\ph)}{\Lambda_m(\LLL,\ph)}
\leq \frac{Q_4 e^{(m-\abs{w})P(\ph) + \hat\ph(w)}}{e^{mP(\ph)}},
\]
where the inequality uses Proposition \ref{prop:Gibbs} for the numerator and Lemma \ref{lem:counting} (specifically the first inequality of \eqref{eqn:countingL}) for the denominator.  Sending $m\to\infty$ gives the upper Gibbs bound in \eqref{eqn:gibbs}.  To prove the lower Gibbs bound, we observe that when $w\in \FFF$, we have
\[
\nu_m(\sigma^{-k}[w]) = \frac{\Lambda_m(\HHH_m(w,k),\ph)}{\Lambda_m(\LLL,\ph)}
\geq \frac{Q_5 e^{(m-\abs{w})P(\ph) + \hat\ph(w)}}{Q_2 e^{mP(\ph)}}
\]
for all sufficiently large $m$; the inequality uses Lemma \ref{lem:counting} for the denominator and Lemma \ref{lem:gcd-counting} for the numerator.  Sending $m\to\infty$ completes the proof of Proposition \ref{prop:Gibbs-measure}.
\iftoggle{arxiv}{}{\qed}
\end{proof}

\subsection{Proof of Corollary \ref{cor:1-sided}}\label{sec:1-sided-corollary}

Let $\tilde X = \{x\in A^\ZZ : x_{[n,\infty)} \in X$ for all $n\in \ZZ\}$ and $\tilde \Sigma = \{\zz\in A_I^\ZZ : \zz_{[n,\infty)} \in \Sigma$ for all $n\in \ZZ\}$. Define $p\colon \tilde X\to X$ by $p(x) = x_{[0,\infty)}$, and $\hat{p}\colon \tilde \Sigma \to \Sigma$ similarly.  Then $p$ is not 1-1 but the induced map $p_* \colon \Ms(\tilde X) \to \Ms(X)$ is an entropy-preserving bijection, and similarly for $\hat{p}_* \colon \MT(\tilde\Sigma) \to \MT(\Sigma)$; see \cite[Proposition 2.1]{CT1} for a proof of this well-known fact.  Now $\tilde\ph = \ph\circ p \in \Ch(\tilde X)$ and $\tilde\Phi = \Phi\circ \hat{p} \in \Ch(\tilde\Sigma)$ are related by $\tilde\Phi = \tilde\ph\circ \tilde\pi$, where $\tilde\pi\colon \tilde \Sigma\to \tilde X$ is the natural 1-block code.  
Because $X$ and $\tilde X$ have the same language, Corollary \ref{cor:2-sided} applies to $(\tilde X,\tilde\ph)$, so
$\tilde\Phi$ is strongly positive recurrent, and there is $P' < P(\tilde\ph)$ such that for every ergodic $\mu\in \Ms(\tilde X)$ with $h(\mu) + \int\ph\,d\mu > P'$, we have $\mu = \pi_* \nu$ for some $\nu\in \MT(\tilde \Sigma)$ with $h(\mu)=h(\nu)$.
Since $Z_n(\Phi,a) = Z_n(\tilde\Phi,a)$, which gives $P_G(\Phi) = P_G(\tilde\Phi)$, and similarly for $Z_n^*$, this implies conclusion \ref{spr} for $\Phi$.  Conclusion \ref{lifts} for $(X,\ph)$ follows since $p_*,\hat{p}_*$ are entropy-preserving bijections (and hence $P(\ph) = P(\tilde\ph)$); existence of an equilibrium state with the Gibbs property for $\ph$ on $\FFF$ follows similarly.

To prove \eqref{eqn:inj}, suppose $\zz,\zz' \in \Sigma$ have $\pi(\zz) = \pi(\zz')$.  Choose $\tilde\zz \in \hat{p}^{-1}(\zz)$ and $\tilde\zz' \in \hat{p}^{-1}(\zz')$.  By Lemma \ref{lem:marking-sets}\ref{lem:piSigma}, $\tilde x = \pi\tilde\zz\in X$ and $\tilde x' = \pi\tilde \zz'\in X$ have bi-infinite $\FFF$-marking sets $J,J'\subset \ZZ$. 
Choose $r\in \NN$ such that $J\cap [0,r) \neq \emptyset$ and $J'\cap [0,r)\neq \emptyset$.  
Because $\tilde x_{[0,\infty)} = \tilde x'_{[0,\infty)}$, $J' \cap [0,\infty)$ is an $\FFF$-marking set for $\tilde x$, and so by Lemma \ref{lem:union-good}, $(J\cup J') \cap [r,\infty)$ is an $\FFF$-marking set for $\tilde x$.  Let $a = \min (J\cap [r,\infty))$, and let $J'' = J \cup (J' \cap [a,\infty))$.  Then $J''$ is a bi-infinite $\FFF$-marking set for $\tilde x$.  By Lemma \ref{lem:bi-inf-union-good}, $J''$ is contained in a bi-infinite maximal $\FFF$-marking set for $\tilde x$, and so  by Lemma \ref{lem:marking-sets}\ref{lem:Jz}, every $j\in J''$ has $\tilde\zz_j = (w,1)$ for some $w\in I$.

A similar argument applies to $\zz'$, so for all sufficiently large $n\in J\cup J'$, there are $w,w'\in I$ with $\abs{w}=\abs{w'}$ such that $\zz_n = (w,1)$ and $\zz'_n=(w',1)$.  Since $I$ is uniquely decipherable by Proposition \ref{prop:11}, this implies that $w=w'$, and we conclude that $\zz_{[n,\infty)} = \zz'_{[n,\infty)}$.  Injectivity on $\Per(\Sigma)$ follows since every $\zz\in \Per(\Sigma)$ is completely determined by $\zz_{[n,\infty)}$ for any $n\in \NN$.

\section{Consequences of strong positive recurrence}
\label{sec:spr-statistics}

In this section we prove Theorem \ref{thm:spr}.  Equality of $P(\ph)$ and $P_G(\Phi)$ is proved in \S\ref{sec:P=}, together with conclusions \ref{unique} and \ref{periodic} on uniqueness and periodic orbits.  The remaining statistical properties are proved in \S\ref{sec:1-sided} for one-sided shifts, and \S\ref{sec:2-sided} for two-sided shifts.  Before proceeding, we recall some well-known facts about the relationship between the one- and two-sided cases.

Let $X$ and $\Sigma$ be two-sided, with alphabets $A$ and $V$, respectively.
As in \S\ref{sec:1-sided-corollary}, define $p\colon A^\ZZ \to A^{\NN\cup \{0\}}$ by $p(x) = x_{[0,\infty)}$, and define $\hat{p} \colon V^\ZZ \to V^{\NN\cup \{0\}}$ similarly.  Let $X^+ = p(X)$ and $\Sigma^+ = \hat{p}(\Sigma)$.  Then $p_* \colon \Ms(X) \to \Ms(X^+)$ and $\hat{p}_* \colon \MT(\Sigma) \to \MT(\Sigma^+)$ are entropy-preserving bijections.
The following result goes back to Sinai \cite[\S3.2]{yS72}; see also  \cite[Lemma 1.6]{rB75}, and
\cite[Theorem 3.1]{yD13} for the countable-state case.

\begin{lemma}\label{lem:cohom}
Given a two-sided countable-state Markov shift $\Sigma$ and $\Psi\in \Ch(\Sigma)$, there is a bounded function $u\in \Ch(\Sigma)$ such that the function $\Psi^+ := \Psi - u + u\circ T \in \Ch(\Sigma)$ only depends on non-negative coordinates; that is, $\Psi^+(\zz) = \Psi^+(\zz')$ whenever $\zz_k = \zz'_k$ for all $k\geq 0$.  The maps $\Psi\mapsto u$ and $\Psi\mapsto \Psi^+$ are linear.
The function $\Psi^+$ can be considered as a function $\Sigma^+\to \RR$, and is H\"older continuous with the same constant and exponent as $\Psi^+\colon \Sigma\to \RR$.  Finally, for any $\zz\in \Sigma$ we have
\begin{equation}\label{eqn:cohom-sum}
\abs{S_n\Psi^+(\hat{p}(\zz)) - S_n\Psi(\zz)} \leq 2\|u\|.
\end{equation}
\end{lemma}

\subsection{Equality of pressures}\label{sec:P=}

\begin{lemma}\label{lem:PGP}
Let $\Sigma$ be a one- or two-sided countable-state Markov shift and $\pi$ a 1-block code from $\Sigma$ to a shift space $X$ on a finite alphabet.  Then for every $\ph\in \Cb(X)$, we have $\Phi = \ph\circ \pi \in \Cb(\Sigma)$ with $\abs{\Phi}_\beta \leq \abs{\ph}_\beta$.  If there is a state $a$ for $\Sigma$ such that $\pi$ is 1-1 on $\Per_n(\Sigma,a)$ for every $n\in \NN$, then $P_G(\Phi) \leq \ulim_{n\to\infty} \frac 1n \log \sum_{x\in \Per_n(X)} e^{S_n\ph(x)} \leq P(\ph)$.
\end{lemma}
\begin{proof}
The estimate on $\abs{\Phi}_\beta$ follows since $\pi$ is a 1-block code and hence does not expand distances.
For the estimate on $P_G(\Phi)$, observe that injectivity of $\pi$ on $\Per_n(\Sigma,a)$ gives
%given $\zz\in \Per_n(\Sigma,a)$, let $w_n(\zz) = \pi(\zz)_{[0,n)}$, so
\[
Z_n(\Phi,a)
= \sum_{\zz \in \Per_n(\Sigma,a)} e^{\sum_{k=0}^{n-1} \ph\circ \pi(T^k \zz)}
%= \sum_{\zz \in \Per_n(\Sigma,a)} e^{\sum_{k=0}^{n-1} \ph(\sigma^k(\pi \zz))}
%\leq \sum_{\zz \in \Per_n(\Sigma,a)} e^{\hat\ph(w(\zz))}.
\leq \sum_{x\in \Per_n(X)} e^{S_n\ph(x)}
\leq \Lambda_n(\ph);
\]
taking logs, dividing by $n$, and sending $n\to\infty$ completes the proof.
%where the last inequality uses injectivity of $\pi$ on $\Per_n(\Sigma,a)$.  
%If $w_n(\zz) = w_n(\zz')$, then $\pi(\zz) = \pi(\zz')$, and since $\pi$ is 1-1 on each $\Per_n(\Sigma,a)$, we conclude that $Z_n(\Phi)$
%we have $w(\zz) \neq w(\zz')$ whenever $\zz,\zz' \in \Per_n(\Sigma,a)$ are distinct, and thus the above inequality implies that $Z_n(\ph\circ\pi,a) \leq \Lambda_n(\ph)$, which completes the proof.
\iftoggle{arxiv}{}{\qed}
\end{proof}

\begin{lemma}\label{lem:PG-var}
If $(\Sigma,T)$ is a topologically transitive one- or two-sided countable-state Markov shift and $\Phi\in \Ch(\Sigma)$ satisfies $P_G(\Phi)<\infty$, then
$P_G(\Phi) = \sup \{ h(\nu) + \int \Phi\,d\nu : \nu\in \MT(\Sigma)\}$, and there is at most one $\nu\in \MT(\Sigma)$ that achieves the supremum.
%\begin{equation}\label{eqn:PG-var}
%P_G(\Phi) = \sup \bigg\{ h(m) + \int \Phi\,dm : m\in \MT(\Sigma) \bigg\}.
%\end{equation}
\end{lemma}
\begin{proof}
The one-sided case follows from \cite[Theorems 4 and 7]{oS99} and \cite[Theorem 1.1]{BS03}.\footnote{See also \cite[Theorem 5.5]{oS15}.  Our assumptions here are stronger than Sarig's; in particular, we only consider $\Phi\in \Ch(\Sigma)$, which are required to be bounded above and below, while Sarig considers locally H\"older potentials that may be unbounded, and thus must restrict the supremum to measures satisfying $\int \Phi\,d\nu > -\infty$.  In our setting there is no need to do this.}
In the two-sided case, let $\Phi^+\in \Ch(\Sigma^+)$ be given by Lemma \ref{lem:cohom}.  Then we have
\[
P_G(\Phi) = P_G(\Phi^+) = \sup_{\nu \in \MT(\Sigma)} \bigg( h(\hat{p}_*\nu) + \int \Phi^+ \,d(\hat{p}_*\nu)\bigg)
= \sup_{\nu \in \MT(\Sigma)} \bigg( h(\nu) + \int \Phi\,d\nu \bigg),
\]
where the first equality uses \eqref{eqn:cohom-sum}, the second uses the one-sided result and the fact that $\hat{p}_*$ is a bijection, and the third uses the fact that $\hat{p}_*$ preserves entropy and that $\Phi,\Phi^+$ are cohomologous so they have the same integrals with respect to any invariant measure.
\iftoggle{arxiv}{}{\qed}\end{proof}

\begin{lemma}\label{lem:PGP2}
If $X$, $\Sigma$, $\pi$, $\ph$ are as in the hypothesis of Theorem \ref{thm:spr}, %so that $\pi$ is 1-1 on $\Per(\Sigma)$ and properties \ref{spr}--\ref{lifts} from Theorem \ref{thm:structure} hold, 
then $\Phi = \ph\circ \pi$ has $P_G(\Phi) = P(\ph)$.  Moreover,  $(X,\sigma,\ph)$ 
%has a unique equilibrium state $\mu$, given by $\mu = \pi_* m$ for the unique measure $m$ achieving the supremum in Lemma \ref{lem:PG-var}.
and $(\Sigma,T,\Phi)$ have unique equilibrium states $\mu$, $m$, related by $\mu = \pi_* m$.
\end{lemma}
\begin{proof}
Observe that $(X,\sigma,\ph)$ has at least one equilibrium state by upper semi-continuity of entropy.   If $\mu$ is any ergodic equilibrium state for $(X,\sigma,\ph)$, then by \ref{lifts} we have $\mu = \pi_* \nu$ for some $\nu\in \MT(\Sigma)$ with $h(\nu) = h(\mu)$.  Along with Lemmas \ref{lem:PGP} and \ref{lem:PG-var}, this gives
\[
P_G(\Phi) \leq P(\ph) = h(\mu) + \int\ph\,d\mu = h(\nu) + \int\Phi\,d\nu \leq P_G(\Phi).
\]
Thus $P_G(\Phi) = P(\ph)$, and $\nu$ achieves the supremum in Lemma \ref{lem:PG-var}.  Since there is at most one $\nu$ achieving this supremum, the result follows.%we conclude that $(X,\sigma,\ph)$ has at most one equilibrium state.
\iftoggle{arxiv}{}{\qed}\end{proof}

Lemma \ref{lem:PGP2} establishes conclusion \ref{unique} of Theorem \ref{thm:main}.
For the result on periodic orbits, let $\mu_n$ be the measures from \eqref{eqn:mun}.  As in \cite[Theorem 9.10]{pW82}, any weak* limit point of the sequence $\mu_n$ is invariant with $h(\mu) + \int\ph\,d\mu \geq\ulim_{n\to\infty} \frac 1n \log \sum_{x\in \Per_n(X)} e^{S_n\ph(x)} = P(\ph)$, where the last equality follows from Lemmas \ref{lem:PGP} and \ref{lem:PGP2}. This shows that every weak*-limit point of the sequence $\mu_n$ is an equilibrium state for $(X,\ph)$; by uniqueness, this proves \ref{periodic}.

\subsection{One-sided shifts}\label{sec:1-sided}

Now we prove statistical properties for one-sided shifts; see \S\ref{sec:2-sided} for the two-sided case.
Let $\Sigma_0,\dots, \Sigma_{p-1}$ be the disjoint closed sets in the spectral decomposition of $\Sigma$, so that $(\Sigma_i, T^p)$ is a topologically mixing countable-state Markov shift upon recoding by cylinders of length $p$.  Under this recoding, $\Sigma_i$ is equipped with the metric
\[
d_p(\zz,\zz') = e^{-k_p(\zz,\zz')},\qquad
k_p(\zz,\zz') = \min\{\abs{k} : \zz_{[pk,(p+1)k)} \neq \zz'_{[pk,(p+1)k)} \},
\]
relative to which we have the following result for 
 $\Phi':=\sum_{j=0}^{p-1} \Phi\circ T^j$.
\begin{lemma}\label{lem:Phi'}
$\Phi'$ is H\"older continuous in the metric $d_p$, with $\abs{\Phi'}_\beta \leq p e^{p\beta} \abs{\Phi}_\beta$.
\end{lemma}
%Let $\Phi' = \sum_{j=0}^{p-1} \Phi\circ T^j$.  
\begin{proof}
For each $\zz,\zz'$ and $0\leq j < p$ we have $d(T^j\zz,T^j\zz') \leq e^{-(k_p(\zz,\zz')-1)p} \leq e^p d_p(\zz,\zz')^p$, and thus $\abs{\Phi'(\zz) - \Phi'(\zz')} \leq \abs{\Phi}_\beta \sum_{j=0}^{p-1} d(T^j \zz,T^j \zz')^\beta
\leq p e^{p\beta} \abs{\Phi}_\beta d_p(\zz,\zz')^{p\beta}$.
\iftoggle{arxiv}{}{\qed}\end{proof}

For each $i$, $\Phi'\in \Ch(\Sigma_i)$ is  strongly positive recurrent and $(\Sigma_i,T^p)$ is topologically mixing.
Let $m\in \MT(\Sigma)$ be the unique equilibrium state for $(\Sigma,T,\Phi)$ from Lemma \ref{lem:PGP2}; then  $m_i = m|_{\Sigma_i} \cdot p$ is the unique equilibrium state for $(\Sigma_i, T^p, \bar\Phi)$.

To obtain the statistical properties claimed in conclusions \ref{bernoulli}--\ref{analytic},
% for the unique equilibrium state $\mu$, 
we use the formulation given by Cyr and Sarig in \cite{CS09}, which is most convenient to our present setting.  What follows could also be done using the machinery of Young towers developed in \cite{lY98,lY99}.

By \cite[Theorem 2.1]{CS09}, strong positive recurrence of $(\Sigma_i,T^d,\Phi')$ implies the \defn{spectral gap property} \cite[Definition 1.1]{CS09}, and then \cite[Theorem 1.1]{CS09} implies the following.%, where we note that elements of $\Cb(\Sigma_i)$ are bounded by definition.
\begin{itemize}
\item \emph{Exponential decay of correlations.}
Given $\beta>0$, there are $\theta\in (0,1)$ and $C>0$ such that for every $\Psi_1\in L^\infty(\Sigma_i,m_i)$, $\Psi_2\in \Cb(\Sigma_i)$ and $n\in \NN$ we have%\footnote{In Theorem 1.1 of \cite{CS09} the bound is given as $K(\Psi_1,\Psi_2) \theta^n$, but an examination of \cite[Lemma 8.3 and \S8.3(b)]{CS09} shows that the dependence is of the form shown here.}
\begin{equation}\label{eqn:edc-2}
%\Cor_n^{m_i}(\Psi_1,\Psi_2) = 
\abs{\int (\Psi_1\circ T^{pn}) \Psi_2 \,dm_i - \int \Psi_1\,dm_i \int \Psi_2\,dm_i} \leq C \|\Psi_1\|_\infty \|\Psi_2\|_\beta \theta^n.
%\leq K(\Psi_1,\Psi_2) \theta^n.
\end{equation}
In \cite{CS09} the upper bound is $K(\Psi_1,\Psi_2)\theta^n$; the bound here uses Lemma \ref{lem:UBP}.
%Note that by Lemma \ref{lem:UBP} we can take $K(\Psi_1,\Psi_2) = C\|\Psi_1\|_\infty \|\Psi_2\|_\beta$ for some $C>0$.
\item \emph{Central limit theorem.}
If $\Psi'\in \Ch(\Sigma_i)$ has $\int \Psi'\,dm_i=0$ and is not equal to $v - v\circ T^p$ for any continuous $v\colon \Sigma_i\to \RR$, then there is $\sigma_{\Psi'}> 0$ such that for all $\tau\in \RR$ we have
\begin{equation}\label{eqn:Psi-clt}
\lim_{n\to\infty} m_i \left\{ \zz\in \Sigma_i \mid \frac 1{\sqrt{n}}  \sum_{k=0}^{n-1} \Psi'(T^{pk} \zz)  \leq \tau \right\}
= 
\frac 1{\sigma_{\Psi'}\sqrt{2\pi}} \int_{-\infty}^\tau e^{-\frac{t^2}{2\sigma_{\Psi'}^2}} \,dt.
\end{equation}
\item \emph{Analyticity of pressure.}  If $\Psi' \in \Cb(\Sigma_i)$, then $t\mapsto P_G^{T^p}(\Phi' + t\Psi')$ is real analytic on a neighbourhood of $0$; we write $P_G^{T^p}$ to emphasize that the quantity here is Gurevich pressure for the action of $T^p$ on the shift whose alphabet is words of length $p$.
\end{itemize}
Now we prove \ref{bernoulli}--\ref{analytic} for $X$.
For \ref{bernoulli}, let $Z_i = \pi(\Sigma_i) \subset X$; these sets may not be disjoint since $\pi$ is not 1-1.  Given $0\leq j < p$, let $X_j := \bigcup_{i=0}^{p-1} Z_i \cap Z_{i+j\pmod p}$.  Then $X_j$ is $\sigma$-invariant and so be ergodicity we have $\mu(X_j)=0$ or $\mu(X_j)=1$ for each $j$.  The set $\{j : \mu(X_j)=1\}$ is closed under addition mod $p$, so $d := \min\{j >0 : \mu(X_j)=1\}$ is a factor of $p$.  Given $0\leq i < d$, let $Y_i = \bigcup_{k=0}^{p/d-1} Z_{i + kd}$; then $\mu(Y_i \cap Y_{i'}) = 0$ for all $i\neq i'$, and $\sigma(Y_i) = Y_{i+1\pmod d}$.  

Note that if \eqref{eqn:inj} is satisfied, then given $\zz,\zz'\in \Sigma$ with $\pi(\zz)=\pi(\zz')$, we have $T^n\zz = T^n\zz'$ for some $n\in \NN$; let $i\in \{0,1,\dots,p-1\}$ be such that $T^n\zz = T^n\zz' \in \Sigma_i$, then $\zz,\zz' \in \Sigma_{i-n\pmod p}$.  %In particular, the sets $Z_i$ are all disjoint when \eqref{eqn:inj} is satisfied, and thus in this case $d=p$.
Thus the sets $Z_i$ are all disjoint, and hence $d=p$ when \eqref{eqn:inj} holds.  
Observe that the argument in this paragraph and the previous one works in both the one- and two-sided cases.

Now we restrict our attention to the one-sided case.
For each $0\leq i < d$, let $\mu_i = \frac dp \sum_{k=0}^{p/d-1} \pi_* m_{i + kd}$, so $\mu_i(Y_i)=1$.  Given $\psi_1,\psi_2 \in \Cb(X)$ with $\int \psi_1\,d\mu_i = \int \psi_2\,d\mu_i = 0$, let 
$\Psi_j = \psi_j\circ \pi \in \Cb(\Sigma)$, so
\begin{equation}\label{eqn:Psi-sum}
\abs{\int (\psi_1\circ \sigma^{dn}) \psi_2 \,d\mu_i} 
\leq \sum_{k=0}^{p/d-1} \abs{\int (\Psi_1 \circ T^{dn}) \Psi_2 \,dm_{i+kd}}.
\end{equation}
Write $dn = pq + r$ for $q\in \NN$ and $r\in \{0,1,\dots, p-1\}$; then 
since $m_{i+kd} = T_*^{kd} m_i$, we have
\begin{multline*}
\abs{\int (\Psi_1\circ T^{dn}) \Psi_2 \,dm_{i+kd}}
= \abs{\int (\Psi_1 \circ T^{dk+pq + r}) (\Psi_2\circ T^{dk}) \,dm_i} \\
%\leq K(\Psi_1 \circ T^{dk+r}, \Psi_2\circ T^{dk}) \theta^q
\leq C\|\Psi_1\circ T^{dk+r}\|_\infty \|\Psi_2\circ T^{dk}\|_\beta \theta^{\frac dp n} 
\leq C\|\Psi_1\|_\infty \|\Psi_2\|_\beta e^{\beta dk} \theta^{\frac dp n},
\end{multline*}
where the first inequality uses \eqref{eqn:edc-2} and the second uses \eqref{eqn:norm-shift} (with $T$ instead of $\sigma$).  Along with \eqref{eqn:Psi-sum}, this gives
%\begin{equation}\label{eqn:edc-1-sided}
$\abs{\int (\psi_1\circ \sigma^{dn}) \psi_2 \,d\mu_i} \leq C\frac pd e^{\beta p} \|\psi_1\|_\infty \|\psi_2\|_\beta \theta^{\frac dp n}$, which proves \ref{bernoulli}.
%\end{equation}
%This proves exponential decay of correlations for $(Y_i,\sigma^d,\mu_i)$. 
%Since $\mu(\bigcup_{i=0}^{d-1} Y_i) = \mu(\pi\Sigma) = 1$, we conclude that $(X,\sigma,\mu)$ has exponential decay of correlations up to the period $d$.

For the central limit theorem in conclusion \ref{clt}, fix $\psi\in \Cb(\Sigma)$ with $\int \psi\,d\mu = 0$.  Let $\Psi = \psi\circ \pi$ and $\Psi' = S_p \Psi$, so $\int \Psi' \,dm_i = 0$ for each $i$.  

\begin{lemma}\label{lem:cohom-const}
If $\Psi' = v-v\circ T^p$ for some continuous $v\colon \Sigma_i\to \RR$, then $\psi = u - u\circ \sigma$ for some $\mu$-integrable $u\colon X\to\RR$.
\end{lemma}
\begin{proof}
Let $f = \sum_{j=1}^{p-1} j \Psi\circ T^{p-1-j}$; then $f - f\circ T = p\Psi - \Psi'$.  Writing $g = \sum_{j=0}^{p-1} v\circ T^j$ gives $g - g\circ T = v - v\circ T^p = \Psi'$, so $(f+g) - (f+g)\circ T = p\Psi$.
By \cite[Theorem 1.6]{sG06}, $v\in L^1(m)$, so $f+g\in L^1(m)$.  Consider the conditional expectation $h = \mathbb{E}[f+g\mid \pi^{-1}\BBB]$, where $\BBB$ is the Borel $\sigma$-algebra on $X$.  By $(\pi^{-1}\BBB)$-measurability of $\Psi = \psi\circ \pi$, we have
\begin{equation}\label{eqn:h-cobdry}
h-h\circ T = \mathbb{E}[(f+g)-(f+g)\circ T\mid \pi^{-1}\BBB] = \mathbb{E}[p\Psi \mid \pi^{-1}\BBB] = p\Psi.
\end{equation}
Since $h$ is $(\pi^{-1}\BBB)$-measurable, there is $u\in L^1(\mu)$ such that $h = pu\circ \pi$, and \eqref{eqn:h-cobdry} gives $u-u\circ \sigma = \psi$ $\mu$-a.e\ (using the fact that $\mu(\pi\Sigma)=1$).  Modifying $u$ on a null set gives the coboundary relationship everywhere.
\iftoggle{arxiv}{}{\qed}\end{proof}

By Lemma \ref{lem:cohom-const}, if $\psi$ is not cohomologous to a constant, then $\Psi' \neq v-v\circ T^p$ for any continuous $v$, so \eqref{eqn:Psi-clt} holds for some $\sigma_{\Psi'}>0$.  Using
%$\frac 1{\sqrt {dn}} S_{dn} \psi(\pi\zz) = \frac 1{\sqrt {dn}} \sum_{k=0}^{n-1} \Psi'(T^{dk}\zz)$, and 
$\mu = \pi_* m = \frac 1p\sum_{i=0}^{p-1} \pi_* m_i$, we get
\begin{multline}\label{eqn:cltmi}
\mu\bigg\{ x : \frac 1{\sqrt{pn}} S_{pn} \psi(x) \leq \tau \bigg\}
= \frac 1p \sum_{i=0}^{p-1} m_i \bigg\{\zz : \frac 1{\sqrt{pn}} \sum_{k=0}^{n-1} \Psi'(T^{pk}\zz) \leq \tau \bigg\}
\\
\to \frac 1{\sigma_{\Psi'}\sqrt{2\pi}} \int_{-\infty}^{\tau\sqrt{p}} e^{-\frac{t^2}{2\sigma_{\Psi'}^2}}\,dt
= \frac {\sqrt p}{\sigma_{\Psi'}\sqrt{2\pi}} \int_{-\infty}^{\tau} e^{-\frac{s^2p}{2\sigma_{\Psi'}^2}}\,ds,
\end{multline}
where the last equality uses the change of variables $s=t/{\sqrt{p}}$.
By comparing $\frac 1{\sqrt{pn}} S_{pn}\psi$ to $\frac 1N S_N\psi$ for $pn \leq N < p(n+1)$ and using continuity in $\tau$ of the last expression in \eqref{eqn:cltmi}, this implies the central limit theorem for $\psi$ with $\sigma_\psi = \sigma_{\Psi'}/{\sqrt{p}}$, completing the proof of \ref{clt}.

Finally, for analyticity %we observe that given $\psi\in \Cb(X)$ and $\DDD\subset \LLL$, the value $P(\DDD,\ph + t\psi)$ depends continuously on $t$, and thus $\Cp,\GGG,\Cs$ continue to satisfy \ref{gap} for $t\approx 0$.  In particular, \eqref{eqn:P=PG} gives $P(\ph + t\psi) = P_G(\Phi + t\Psi)$ for all $t\approx 0$.  
we start by observing that given $\Psi\in \Ch(\Sigma)$, Lemma \ref{lem:PG-var} gives $P_G^{T^p}(S_p \Psi) = p\cdot P_G^T(\Psi)$, and thus $t\mapsto P_G(\Phi + t\Psi)$ is real analytic on a neighborhood of $0$.
Now observe that %property \ref{per-inj} in Theorem \ref{thm:structure} 
injectivity of $\pi$ on $\Per_n(\Sigma,a)$
does not depend on $\ph$, and that properties \ref{spr} and \ref{lifts} are stable under uniformly small perturbations.  More precisely, if $\psi\in \Ch(X)$ has 
$2\|\psi\| < P_G(\Phi) - \ulim \frac 1n \log Z_n^*(\Phi,a)$, then for $\Psi = \psi\circ \pi$ we have
\[
\ulim \frac 1n \log Z_n^*(\Phi+ \Psi,a) \leq \|\psi\| + \ulim \frac 1n \log Z_n^*(\Phi,a) <  P_G(\Phi) - \|\psi\| \leq P_G(\Phi+\Psi),
\]
so $\Phi + \Psi$ is strongly positive recurrent as well.  Similarly, if $P' < P(\ph)$ is as in \ref{lifts} and $2\|\psi\| \leq P(\ph) - P'$, then every ergodic $\mu\in \Ms(X)$ with $h(\mu) + \int (\ph + \psi)\,d\mu > P' + \|\psi\|$ has $h(\mu) + \int \ph\,d\mu > P'$, and thus $\mu = \pi_*\nu$ for some $\nu\in \MT(\Sigma)$ with $h(\nu)=h(\mu)$; since $P' + \|\psi\| < P(\ph) - \|\psi\| \leq P(\ph + \psi)$, this shows that $\ph+\psi$ satisfies \ref{lifts}.  By Lemma \ref{lem:PGP2}, we conclude that for all $\psi \in \Ch(X)$, there is $\eps>0$ such that $P_G(\Phi + t \Psi) = P(\ph + t\psi)$ when $\abs{t}<\eps$, hence $t\mapsto P(\ph + t\psi)$ is real analytic on a neighborhood of $0$, proving \ref{analytic}.

\subsection{Two-sided shifts}\label{sec:2-sided}

Now suppose that $X$ is a two-sided shift space satisfying the hypotheses of Theorem \ref{thm:spr} for some two-sided countable-state Markov shift $\Sigma$, 1-block code $\pi\colon \Sigma\to X$, and $\ph \in \Ch(X)$.
Let $\Phi = \ph\circ \pi \in \Ch(\Sigma)$, and let $\Phi^+ \in \Ch(\Sigma^+)$ be given by Lemma \ref{lem:cohom}.  Since $\Phi$ is strongly positive recurrent, it follows from \eqref{eqn:cohom-sum} that $\Phi^+$ is strongly positive recurrent as well.  Let $\mu$ and $m$ be the unique equilibrium states for $(X,\sigma,\ph)$ and $(\Sigma,T,\Phi)$ from Lemma \ref{lem:PGP2}, and let $m^+ = \hat{p}_* m$, so that $m^+$ is the unique equilibrium state for $(\Sigma^+,T,\Phi^+)$.  Finally, let $\pi^+ \colon \Sigma^+ \to X^+$ be the 1-block code defined by the same map $V\to A$ as $\pi$; then $\pi^+ \circ \hat{p} = p\circ \pi$, and so $\mu^+ := \pi^+_* m^+ = p_* \mu$.

We prove \ref{bernoulli}--\ref{analytic} assuming that $\Sigma^+$ is topologically mixing; the extension to the topologically transitive case is just as above.  Given mixing, the previous section gives $C>0$ such that for every $\psi^+_1,\psi^+_2\in \Cb(X^+)$ with $\int \psi^+_1\,d\mu = \int\psi^+_2\,d\mu = 0$, we have
\begin{equation}\label{eqn:1-edc}
\abs{\int (\psi^+_1 \circ \sigma^n) \psi^+_2\,d\mu^+ } \leq C \| \psi^+_1\| \|\psi^+_2\|_\beta \theta^n
\text{ for all } n\geq 0.
\end{equation}
We deduce exponential decay of correlations for the two-sided shift $(X,\sigma,\mu)$ following \cite[Proposition 2.4]{PP90} (see also  \cite[\S 4]{lY98}).  Given $\psi_1,\psi_2\in \Cb(X)$ with $\int \psi_1\,d\mu = \int \psi_2\,d\mu = 0$ and $k\in \NN$, we approximate $\psi_i$ with functions $\psi_i^k\in \Cb(X)$ that depend only on the coordinates $-k,\dots,k$, as follows.  Choose any function $\xi \colon \LLL_{2k+1} \to X$ such that $\xi(w)_i = w_{i+k+1}$ for all $\abs{i} \leq k$; then define $\zeta_k\colon X\to X$ by $\zeta_k(x) = \xi(x_{[-k,k]})$.  Finally, put $\psi_i^k = \psi_i \circ \zeta_k$.  Observe that $|\psi_i^k|_\beta \leq \abs{\psi_i}_\beta$ since $d(\zeta_k x,\zeta_k y) \leq d(x,y)$ for all $x,y\in X$.  Moreover, $\|\psi_i - \psi_i^k\| \leq \abs{\psi}_\beta e^{-\beta k}$.

Since $\psi_i^k (\sigma^k x)$ only depends on $x_{[0,\infty)} = p(x)$, we can treat $\psi_i^k \circ \sigma^k$ as a function on $X^+$, and thus for every $n\geq 0$, \eqref{eqn:1-edc} and \eqref{eqn:norm-shift} give
\begin{equation}\label{eqn:edc-k}
\abs{\int (\psi_1^k\circ \sigma^n) \psi_2^k \,d\mu}
= \abs{\int ((\psi_1^k \circ \sigma^k) \circ \sigma^n) (\psi_2^k \circ \sigma^k) \,d\mu}
\leq C \|\psi_1\| \|\psi_2 \|_\beta e^{\beta k}\theta^n.
\end{equation}
%By \eqref{eqn:norm-shift}, we have
%There is $\lambda>1$, depending on $\beta$ but not on $\psi_i$ or $n$, such that for all $k$ we have
%$\|\psi_2 \circ \sigma^k\|_\beta \leq \| \psi_2 \|_\beta e^{\beta k}$.
Thus by writing
\[
(\psi_1\circ \sigma^n) \psi_2 = (\psi_1 \circ \sigma^n - \psi_1^k\circ \sigma^n)\psi_2 
+ (\psi_1^k\circ \sigma^n)(\psi_2 - \psi_2^k)
+ (\psi_1^k\circ \sigma^n)(\psi_2^k)
\]
and integrating, we deduce that
\begin{align*}
\abs{\int (\psi_1\circ \sigma^n) \psi_2 \,d\mu}
%&= \bigg| \int \Big((\psi_1 \circ \sigma^n - \psi_1^k\circ \sigma^n)\psi_2 
%\iftoggle{arxiv}{}{\\ &\qquad\qquad}
%+ (\psi_1^k\circ \sigma^n)(\psi_2 - \psi_2^k)
%+ (\psi_1^k\circ \sigma^n)(\psi_2^k)\Big)\,d\mu\bigg|  \\
&\leq \| \psi_1 - \psi_1^k \| \|\psi_2\| + \|\psi_1\| \| \psi_2 - \psi_2^k\| + C \|\psi_1\| \|\psi_2\|_\beta e^{\beta k} \theta^n \\
&\leq \|\psi_1\|_\beta \|\psi_2\|_\beta (2 e^{-\beta k} + C e^{\beta k} \theta^n).
\end{align*}
Fix $\gamma>0$ small enough that $\beta \gamma + \log\theta < 0$, and given $n\in \NN$, choose $k\in \NN$ with $\gamma n/2 \leq k \leq \gamma n$; then the above estimate gives
\[
\abs{\int (\psi_1\circ \sigma^n) \psi_2 \,d\mu}
\leq \|\psi_1\|_\beta \|\psi_2\|_\beta (2 e^{-\beta \gamma n/2} + K e^{(\beta \gamma + \log \theta)n} ),
\]
which proves exponential decay of correlations for $(X,\sigma,\mu)$.

%Roughly speaking, the idea is to approximate $\psi_1,\psi_2\in \Cb(X)$ with $\psi_1^k,\psi_2^k$ that depend only on coordinates $-k,\dots,k$ (for example, one can obtain $\psi_i^k$ as a conditional average of $\psi_i$ over $[-k,k]$-cylinders).\footnote{In either case, one must ensure that a single constant $K$ can be chosen to work for all $k$ in \eqref{eqn:edc-2} and \eqref{eqn:edc-3}.} In \cite{PP90} this is done for functions on $X$, while in \cite{lY98} it is done at the level of the tower; although the notation is different there, the idea is that one takes a function $\psi\colon X\to \RR$, considers $\Psi = \psi\circ \pi\colon \Sigma\to \RR$, and then approximates $\Psi$ with $\Psi_k$.  Ultimately one is able to reduce to the one-sided case; we omit the details as there is nothing new in our setting.  Note that for the two-sided result one must assume that \emph{both} test functions are H\"older.

For the central limit theorem, consider $\psi\in \Cb(X)$ with $\int\psi\,d\mu=0$.  Let $\Psi=\psi\circ \pi\in \Cb(\Sigma)$, and let $\Psi^+, u$ be as in Lemma \ref{lem:cohom}.  
%If $\Psi^+$ is cohomologous to a constant, then $\Psi$ is as well, so there is $f\in L^2(\Sigma,\pi_*^{-1}\mu)$ such that $\Psi = f - f\circ T$.  Define $g\in L^2(\Sigma,\mu)$ by $g(x) = f(\pi^{-1}x)$ on $\pi(\Sigma)$, and $g=0$ elsewhere; then $\psi = g-g\circ \sigma$ on $\pi(\Sigma)$, so $\psi$ is cohomologous to a constant.
As in the previous section, if $\psi$ is not cohomologous to a constant then neither is $\Psi^+$.
By \eqref{eqn:cohom-sum} we have
\begin{equation}\label{eqn:cohom-sum-2}
\abs{\frac 1{\sqrt{n}} \sum_{k=0}^{n-1} \Psi^+(T^k\hat{p}(\zz)) - \frac 1{\sqrt{n}} \sum_{k=0}^{n-1} \Psi(T^k \zz)} \leq \frac{2\|u\|}{\sqrt{n}}.
\end{equation}
Write $G_n^\mu(\tau) = \mu \{x\in X \mid \frac 1{\sqrt{n}} \sum_{k=0}^{n-1} \psi(\sigma^k x) \leq \tau\}$, and similarly for $G_n^m(\tau)$ (summing over $T$-orbits on $\Sigma^+$).  Then \eqref{eqn:cohom-sum-2} gives
$
G_n^m\left(\tau - \frac{2\|u\|}{\sqrt{n}}\right) \leq G_n^\mu(\tau) \leq  G_n^m\left(\tau + \frac{2\|u\|}{\sqrt{n}}\right),
$
and thus $G_n^\mu(\tau)$ converges to the right-hand side of \eqref{eqn:Psi-clt} (this uses continuity of that expression). 

For analyticity, we observe that the linearity in Lemma \ref{lem:cohom} gives $(\Phi + t\Psi)^+ = \Phi^+ + t\Psi^+$, so $P_G(\Phi + t\Psi) = P_G(\Phi^+ + t\Psi^+)$ is analytic in $t$ near $0$, and as in the one-sided case we have $P(\ph + t\psi) = P_G(\Phi + t\Psi)$ for $\abs{t}$ sufficiently small, proving conclusion \ref{analytic} of Theorem \ref{thm:main}.

Finally, it follows from \cite{yD13} that $(\Sigma,T,m)$ is Bernoulli up to a period, and thus its factor $(X,\sigma,\mu)$ is also Bernoulli up to a period by \cite{dO70}.

\section{Proof of Theorem \ref{thm:get0spec}}\label{sec:get0spec}

Now we turn our attention to Theorem \ref{thm:get0spec} and assume that we have $\GGG\subset \LLL$ satisfying \ref{spec}, \ref{gap}, \ref{inter}, \ref{union}.  We must produce $\FFF \subset \LLL$ satisfying \ref{free}, \ref{gap-2}, and \ref{overlaps}, such that $\mu$ has the Gibbs property for $\ph$ on $\GGG$ if and only if $\mu$ has the Gibbs property for $\ph$ on $\FFF$.  We first dispense with a trivial case.  We say that $\GGG\subset \LLL$ is \defn{periodic} if there is a periodic sequence $x\in X$ such that every $w\in \GGG$ appears somewhere in $x$.

\begin{proposition}\label{prop:periodic}
If $(X,\ph)$ is such that there is a periodic $\GGG\subset \LLL$ satisfying \ref{spec} and \ref{gap}, then there is $\FFF \subset \LLL$ satisfying \ref{free}, \ref{gap-2}, and \ref{overlaps}.
\end{proposition}
\begin{proof}
Let $x\in X$ be periodic such that every $w\in \GGG$ appears in $x$.  Let $d\in \NN$ be the least period of $x$ and let $\FFF = \{x_{[1,kd]} \mid k\in \NN\}$.  Then $\FFF$ has \ref{free} and \ref{overlaps} (the second assertion uses the fact that $d$ is minimal).  Let $\Cp,\Cs$ be given by \ref{gap} and put $\Ep =\Cp \LLL_{\leq d} \cap \LLL$, $\Es = \LLL_{\leq d} \Cs \cap \LLL$.
%\begin{align*}
%\Ep &= \Cp \LLL_{\leq d} \cap \LLL = \{wv\in \LLL \mid w\in \Cp, \abs{v}\leq d\},
%\\
%\Es &= \LLL_{\leq d} \Cs \cap \LLL = \{vw\in \LLL \mid \abs{v}\leq d, w\in %\Cs\}.
%\end{align*}
Then given any $u^p\in \Cp, v\in \GGG, u^s\in \Cs$, we note that there are $i\in [1,d]$ and $j\in (\abs{v}-d,\abs{v}]$ such that $v_{[i,j]}\in \FFF$, and hence
$
u^p v u^s = (u^p v_{[1,i)}) v_{[i,j]} (v_{(j,\abs{v}]} u^s) \in \Ep \FFF \Es
$.
Together with the observation that $P(\Ep,\ph) = P(\Cp,\ph)$, $P(\Es,\ph) = P(\Cs,\ph)$, and $I=\FFF\setminus \FFF\FFF$ is finite, this establishes \ref{gap-2}.% for $\Ep$, $\FFF$, $\Es$.
\iftoggle{arxiv}{}{\qed}\end{proof}

For non-periodic $\GGG$, the key to producing $\FFF$ satisfying \ref{free} is the following definition.

\begin{definition}\label{def:triple}
Given $\GGG\subset \LLL$ satisfying \ref{spec},
we say that $(r,c,s)$ is a \defn{synchronising triple} for $\GGG$ if $r,s\in \GGG$, $c\in \LLL_{\leq \tau}$, and given any $r'\in \LLL r \cap \GGG$ and $s'\in s\LLL\cap \GGG$, we have $r'cs'\in \GGG$.
In this case we write %$\BBB^{r,s} = \LLL r \cap s\LLL \cap \GGG$, and 
$\FFF^{r,c,s} = c(\LLL r \cap s\LLL \cap \GGG)$.\footnote{Note that this is not the same thing as defining $\FFF^{r,c,s}$ to be all words in $\GGG$ that are of the form $csxr$, since we allow $s$ and $r$ to overlap.}
% = \{cw \mid w\in \BBB^{r,s}\}$.
\end{definition}

\begin{figure}[htbp]
\includegraphics[width=220pt]{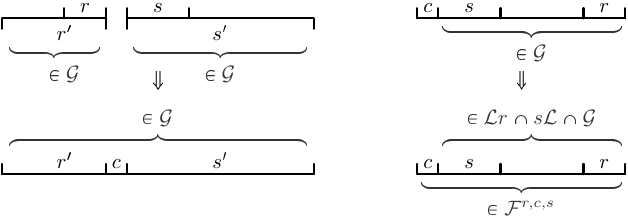}
\caption{A synchronising triple $(r,c,s)$ and the collection %s $\BBB^{r,s}$, 
$\FFF^{r,c,s}$ it generates.}
\label{fig:sync-trip}
\end{figure}

Theorem \ref{thm:get0spec} follows immediately from Proposition \ref{prop:periodic} and the following two results, which we prove in \S\ref{sec:get-free} and \ref{sec:get-gap}, respectively.

%Recall that we will produce $\FFF$ satisfying \ref{free} by finding a synchronising triple $(r,c,s)$ for $\GGG$; that is, $r,s\in \GGG$ and $c\in \LLL_{\leq \tau}$ such that for every $r' \in \LLL r\cap \GGG$ and $s' = s\LLL \cap \GGG$ we have $r'cs'\in \GGG$; then we will take $\FFF = \FFF^{r,c,s} = \{cw \mid w\in \LLL r \cap s\LLL \cap \GGG\}$.  By a careful choice of $r,s$ we will also use \ref{inter} for $\GGG$ to deduce \ref{overlaps} for $\FFF$.  This is the content of the following result, which implies Proposition \ref{prop:get-free} and is proved  in \S\ref{sec:get-free}.

\begin{proposition}\label{prop:get-free-2}
Every $\GGG\subset \LLL$ satisfying \ref{spec} has a synchronising triple $(r,c,s)$, and both $r$ and $s$ can be chosen to be arbitrarily long.  In addition, we have:
\begin{enumerate}[leftmargin=*, widest=d, label=\textup{(\alph{*})}]
\item\label{a} If $(r,c,s)$ is any synchronising triple for $\GGG$, then $\FFF^{r,c,s}$ satisfies \ref{free}.
\item\label{b} If $(r,c,s)$ is any synchronising triple for $\GGG$, then a measure $\mu$ is Gibbs for $\ph$ on $\GGG$ if and only if it is Gibbs for $\ph$ on $\FFF^{r,c,s}$.
\item\label{c} If $\GGG$ is not periodic, then it has a synchronising triple $(r,c,s)$ with $r,s$ arbitrarily long and 
\begin{equation}\label{eqn:short-overlaps}
[rcs] \cap \sigma^{-k}[rcs] = \emptyset \text{ for every }1\leq k\leq \max\{\abs{rc},\abs{cs}\}.
\end{equation}
If in addition $\GGG$ satisfies \ref{s-spec} for some $\tau\in \NN$, then we can take $\abs{c}=\tau$ and get $\gcd\{ \abs{w} \mid w\in \FFF^{r,c,s}\} = \gcd \{ \abs{v} + \tau \mid v\in \GGG\}$.
\item\label{d} If $\GGG$ satisfies \ref{inter} for some $L\in \NN$, and $(r,c,s)$ is any synchronising triple satisfying \eqref{eqn:short-overlaps} and  $\abs{r},\abs{s}\geq L$, then $\FFF^{r,c,s}$ satisfies \ref{overlaps}.
\end{enumerate}
\end{proposition}

%After Proposition \ref{prop:get-free-2} is proved, the rest of the proof of Theorem \ref{thm:get0spec} is devoted to establishing \ref{gap-2} via the following result, which we prove in \S\ref{sec:get-gap}.

\begin{proposition}\label{prop:get-gap}
Let $X$ be a shift space on a finite alphabet and $\ph\in \Cb(X)$ for some $\beta>0$.  Suppose $\GGG\subset \LLL(X)$ satisfies \ref{spec}, \ref{gap}, \ref{inter}, \ref{union}.  If $(r,c,s)$ is any synchronising triple for $\GGG$ satisfying \eqref{eqn:short-overlaps} and $\abs{r},\abs{s}\geq L$, where $L$ is large enough for both \ref{inter} and \ref{union} to hold, then $\FFF = \FFF^{r,c,s} = c(s\LLL \cap \LLL r \cap \GGG)$ satisfies \ref{gap-2}:  for $I=\FFF\setminus\FFF\FFF$ we have $P(I,\ph) < P(\ph)$, and there are $\Ep,\Es\subset \LLL$ such that $P(\Ep \cup \Es \cup (\LLL\setminus \Ep \FFF \Es),\ph)<P(\ph)$.
\end{proposition}

\subsection{Producing a collection of words with free concatenation}\label{sec:get-free}

In this section we prove Proposition \ref{prop:get-free-2}.  We start by establishing existence of arbitrarily long synchronising triples in \S\ref{sec:sync-trip}.  In \S\ref{sec:free-gibbs} we prove parts \ref{a} and \ref{b}.  In \S\ref{sec:short-overlaps} we prove part \ref{c}.  In \S\ref{sec:overlaps} we prove \ref{d}.

\subsubsection{Existence of a synchronising triple}\label{sec:sync-trip}

The following lemma mimics the proof from \cite{aB88} that specification implies synchronised. 

\begin{lemma}\label{lem:reset}
Suppose $\GGG$ satisfies \ref{spec}.  Then given $v,w\in \GGG$, there are $q \in v \LLL \cap \GGG$, $p\in \LLL w \cap \GGG$, and $c\in \LLL_{\leq\tau}$ such that $(p,c,q)$ is a synchronising triple for $\GGG$.
%$pcq\in \GGG$, and given any $up,qu'\in \GGG$, we have $upcqu'\in \GGG$.
If in addition $\GGG$ satisfies \ref{s-spec} for some $\tau\in\NN$, then $(p,c,q)$ can be chosen so that $\abs{c}=\tau$.
\end{lemma}
\begin{proof}
Define a partial order on $\GGG\times \GGG$ by writing $(v,w) \succ (v',w')$ if $v\in v' \LLL$ and $w\in \LLL w'$.
Given $v,w\in \GGG$, let $C(w,v) = \{ c\in \LLL_{\leq \tau} : w'cv' \in \GGG$ for all $(v',w') \prec (v,w)\}$.  It follows from the definition that $C(\bar w, \bar v) \subset C(w,v)$ whenever $(\bar v,\bar w) \succ (v,w)$.
By \ref{spec}, $C(w,v)$ is nonempty for all $v,w\in \GGG$.\footnote{This is the only place in the paper where we require the conclusion of \ref{spec} for $v'\neq v$ and $w'\neq w$; everywhere else, it would suffice to only require $vuw\in \GGG$ in \ref{spec}.}
Since $C(w,v)$ is finite, there is a pair $(p,q) \in \GGG\times \GGG$ such that $(p,q) \succ (v,w)$ and $C(\bar w, \bar v) = C(q,p)$ for all $(\bar v,\bar w) \succ (p,q)$, which implies that $(p,c,q)$ is a synchronising triple for $\GGG$.  For the claim about the case when $\GGG$ satisfies \ref{s-spec}, it suffices to replace $\LLL_{\leq\tau}$ with $\LLL_\tau$ in the definition of $C(w,v)$ and then repeat this argument.
\iftoggle{arxiv}{}{\qed}\end{proof}

%\begin{figure}[htbp]
%\includegraphics[width=220pt]{get-sync}
%\caption{Producing a synchronising triple.}
%\label{fig:get-sync}
%\end{figure}

%The triple $(p,c,q)$ produced in Lemma \ref{lem:reset} is a synchronising triple for $\GGG$.   
Note that $p,q$ can be taken arbitrarily long by choosing long words $v,w$ in the lemma.

\subsubsection{Free concatenation and Gibbs properties}\label{sec:free-gibbs}

Now we take an arbitrary synchronising triple $(p,c,q)$ and let $\FFF = \FFF^{p,c,q} = c(\LLL p \cap q\LLL \cap \GGG)$.  Given any $v,w\in \FFF$ there are $v',w'\in \LLL p \cap q\LLL \cap \GGG$ such that $v=cv'$, $w=cw'$, and since $(p,c,q)$ is a synchronising triple we have $v'cw'\in \LLL p \cap q\LLL \cap \GGG$ as well, so $vw=cv'cw' \in \FFF$. Thus $\FFF$ satisfies \ref{free}.

For part \ref{b} of Proposition \ref{prop:get-free-2}, we show that a measure $\mu$ has the Gibbs property for $\ph$ on $\FFF$ if and only if it has the Gibbs property for $\ph$ on $\GGG$.
Note that the upper bound in \eqref{eqn:gibbs} is required to hold for all $w$, so it suffices to check the lower bound.

Suppose $\mu$ is Gibbs for $\ph$ on $\GGG$ with constant $Q_1$, and $(p,c,q)$ is a synchronising triple.  Then we have $pw\in \GGG$ for each $w\in \FFF^{p,c,q}$, and in particular
\[
\mu[w] \geq \mu[pw] \geq Q_1^{-1} e^{-\abs{pw}P(\ph) + \hat\ph(pw)}
\geq Q_1^{-1} e^{-\abs{w} P(\ph) + \hat\ph(w)} e^{-\abs{p}(P(\ph) + \|\ph\|)},
\]
so $\mu$ is Gibbs for $\ph$ on $\FFF^{p,c,q}$.  Conversely, if $\mu$ is Gibbs for $\ph$ on $\FFF^{p,c,q}$ then given $w\in \GGG$ there are $u,v\in \LLL_{\leq \tau}$ such that $cquwvp\in c(q \LLL \cap \LLL p \cap \GGG) = \FFF^{p,c,q}$, and hence
\begin{align*}
\mu[w] &\geq \mu[cquwvp] \geq Q_1^{-1} e^{-\abs{cquwvp}P(\ph) + \hat\ph(cquwvp)} \\
&\geq Q_1^{-1} e^{-(\abs{cq} + \abs{p} + 2\tau)(P(\ph)+\|\ph\|)}
e^{-\abs{w} P(\ph) + \hat\ph(w)},
\end{align*}
so $\mu$ is Gibbs for $\ph$ on $\GGG$.  This completes the proof of Proposition \ref{prop:get-free-2}\ref{a}--\ref{b}.

%\begin{remark}\label{rmk:GH}%
%Every word in $\FFF = \FFF^{r,c,s}$ can be extended to a word in $\GGG$ in a uniform number of symbols, and vice versa.  On the one hand, given any $w\in \FFF$, we have $w = cv$ for some $v\in s\LLL \cap \GGG$, hence $rw=rcv\in \GGG$.  On the other hand, if $w\in \GGG$ then there are $u,v\in \LLL_{\leq \tau}$ such that $suwvr\in \BBB^{r,s} = s\LLL \cap \LLL r \cap \GGG$, and hence $csuwvr\in \FFF$.  
%\end{remark}

\subsubsection{A synchronising triple with no long overlaps}\label{sec:short-overlaps}

Now we prove Proposition \ref{prop:get-free-2}\ref{c}.  Assume that $\GGG\subset \LLL$ satisfies \ref{spec} and is not periodic; this has the following consequence.\footnote{A very similar result is proved in \cite[Proposition 2.4]{CT1}; see \S6.3 there.  Our hypotheses here are different; neither set of conditions implies the other.}

\begin{lemma}\label{lem:good-entropy}
Suppose $\GGG$ has \ref{spec} and is not periodic, and let $(p,c,q)$ be 
a synchronising triple for $\GGG$.  Then there is $\ell\in \NN$ such that $\#\GGG_{\ell m - \abs{c}} \geq 2^m$ for all $m\in \NN$.
\end{lemma}
\begin{proof}
Let $\FFF = \FFF^{p,c,q}$.
First we show that there are $v,w\in \FFF$ such that $\abs{w}\geq \abs{v}$ and $w\notin v\LLL$. Suppose this was false; 
then we can define $x\in A^\NN$ by $x_i = w_i$ whenever $w\in \FFF_{\geq i}$.  This is well-defined since $v_i=w_i$ whenever $v,w\in \FFF_{\geq i}$.  Moreover, given $v\in \FFF$ we have $vv\cdots v\in \FFF$ for arbitrarily long concatenations of $v$ with itself, so $x_{i+k\abs{v}} = v_i = x_i$ for any $k\geq 0$ and $1\leq i \leq \abs{v}$.  It follows that $x$ is periodic, and that every $v\in \FFF$ appears as a prefix of $x$.  But then $x$ contains every $w\in \GGG$ as well, since as in \S\ref{sec:free-gibbs} there are $u,v\in \LLL_{\leq \tau}$ such that $cquwvp\in \FFF$.

By the previous paragraph there are  $v,w\in \FFF$ such that $v_i \neq w_i$ for some $i\leq \min(\abs{v},\abs{w})$.  Let $\ell = \abs{v}\cdot \abs{w}$; let $u^1 = vv\cdots v$ and $u^2=ww\cdots w$, where we concatenate $\abs{w}$ copies of $v$ and $\abs{v}$ copies of $w$ so that $\abs{u^1} = \abs{u^2} = \ell$.  By construction of $v,w$ we have $u^1\neq u^2$.  Now for every $m\in \NN$ and $y\in \{1,2\}^m$, we have $u^{y_1} \cdots u^{y_m}\in \FFF_{m\ell}$.  Moreover, different choices of $y$ yield different words in $\FFF_{m\ell}$, so that $\#\FFF_{m\ell} \geq 2^m$.  Since $\FFF \subset c\GGG$, we get $\#\GGG_{m\ell-\abs{c}} \geq 2^m$.
\iftoggle{arxiv}{}{\qed}\end{proof}

Now let $(p,c,q)$ be a synchronising triple for $\GGG$ with $\abs{p},\abs{q}\geq L$, where $L$ is the constant from \ref{inter}.  Let $\ell$ be as in Lemma \ref{lem:good-entropy}, and choose $\alpha>0$ such that $\alpha \ell \log(\#A) < \log 2$, where we recall that $A$ is the alphabet.
There are arbitrarily large $m,n\in \NN$ such that
\begin{equation}\label{eqn:mn}
\frac \ell\alpha m + \frac{2\tau + \abs{pcq}}\alpha \leq \ell n - \abs{c} < 2^m.
\end{equation}
Say that a word $w\in \LLL$ is \defn{$k$-periodic} if
\begin{equation}\label{eqn:w-periodic}
w_{i+k} = w_i \text{ for every } 1\leq i\leq \abs{w} - k.
\end{equation}

\begin{lemma}\label{lem:v-not-in-w}
For sufficiently large $m,n$ satisfying \eqref{eqn:mn}, there are $v\in \GGG_{\ell m - \abs{c}}$ and $w\in \GGG_{\ell n - \abs{c}}$ such that $v$ is not a subword of $w$, and $w$ is not $k$-periodic for any $1\leq k \leq \alpha\abs{w}$.
\end{lemma}
\begin{proof}
Consider the collection $\PPP^\alpha = \{ w\in \LLL \mid w$  is $k$-periodic for some $1\leq k\leq \alpha\abs{w}\}$.
If $w$ is $k$-periodic then it is determined by its first $k$ entries, so 
\begin{equation}\label{eqn:Dn}
\#\PPP_N^\alpha \leq \sum_{k=1}^{\lfloor \alpha N \rfloor} (\#A)^k
\leq (\#A)^{\alpha N} \sum_{j=0}^\infty (\# A)^{-j} = e^{\alpha \log(\# A) N} \left(\frac {\#A}{\#A - 1}\right).
\end{equation}
Write $\gamma := (\log 2)/\ell - \alpha \log\#A$ and note that $\gamma>0$ by the choice of $\alpha$.  When $N=\ell n - \abs{c}$ for some $n\in \NN$, Lemma \ref{lem:good-entropy} gives $\#\GGG_N \geq 2^n \geq 2^{N/\ell}$, and so 
\begin{equation}
\frac{\#\GGG_N}{\#\PPP_N^\alpha}
\geq \left(\frac {\#A - 1}{\#A}\right)
e^{(\log 2)\frac{N}\ell - \alpha \log(\# A)N}
\geq \left(\frac {\#A - 1}{\#A}\right) e^{\gamma N}.
\end{equation}
For $n$ sufficiently large this gives $\#\GGG_N > \#\PPP_N^\alpha$, so there is $w\in \GGG_N = \GGG_{\ell n - \abs{c}}$ that is not $k$-periodic for any $1\leq k \leq \alpha \abs{w}$.  
To put it another way: for every sufficiently large 
$n$ there is $w\in \GGG_{\ell n - \abs{c}}$ such that
\begin{equation}\label{eqn:w-not-periodic}
\text{for every $1\leq k\leq \alpha\abs{w}$ there is $1\leq j\leq \abs{w} - k$ with $w_{k + j} \neq w_k$}.
\end{equation}
Now let $m,n\in \NN$ be such that \eqref{eqn:mn} is satisfied and \eqref{eqn:w-not-periodic} holds for some $w\in \GGG_{\ell n - \abs{c}}$.  Note that $w$ contains at most $\abs{w}$ subwords of length $\ell m - \abs{c}$, while $\#\GGG_{\ell m - \abs{c}} \geq 2^m > \abs{w}$ by Lemma \ref{lem:good-entropy} and \eqref{eqn:mn}.  Thus there is $v\in \GGG_{\ell m - \abs{c}}$ such that $w_{[i,i+\abs{v})} \neq v$ for every $1\leq i\leq \abs{w}-\abs{v}$; that is, $v$ is not a subword of $w$.
\iftoggle{arxiv}{}{\qed}
\end{proof}

Let $v,w$ be the words provided by Lemma \ref{lem:v-not-in-w}.  By \ref{spec} there are $u,u'\in \LLL_{\leq \tau}$ such that $r=vup\in \GGG$ and $s=qu'w\in \GGG$.  Then $(r,c,s)$ is again a synchronising triple for $\GGG$; we show that it satisfies condition \eqref{eqn:short-overlaps}, which can be thought of as forbidding `long overlaps' of $rcs$ with itself.  Suppose that $k>0$ is such that there is $x\in[rcs] \cap \sigma^{-k}[rcs]$.  Then we have
\begin{equation}\label{eqn:xrcs}
x_{[0,\abs{rcs})} = x_{[k,k+\abs{rcs})} = rcs = vupcqu'w.
\end{equation}
%We want to choose $v,w$ such that  \eqref{eqn:xrcs} forces $k \geq \max(\abs{vupc},\abs{cqu'w})$.  We will choose $v$ to be (much) longer than $p,q$, and $w$ to be (much) longer than $v$.  
Figure \ref{fig:short-overlaps} illustrates the three possible ranges of $k$ that we must deal with:
\begin{enumerate}
\item $1\leq k \leq \abs{vupcqu'}$, so $v$ starts in $x$ before $w$ starts in $\sigma^k(x)$; 
% \leq \alpha\abs{w}$;
\item $\abs{vupcqu'} < k < \abs{upcqu'w} = \abs{rcs} - \abs{v}$, so $v$ is a subword of $w$;
\item $k\geq \abs{upcqu'w} \geq \max(\abs{vupc},\abs{cqu'w})$.
\end{enumerate}

\begin{figure}[htbp]
\includegraphics[width=280pt]{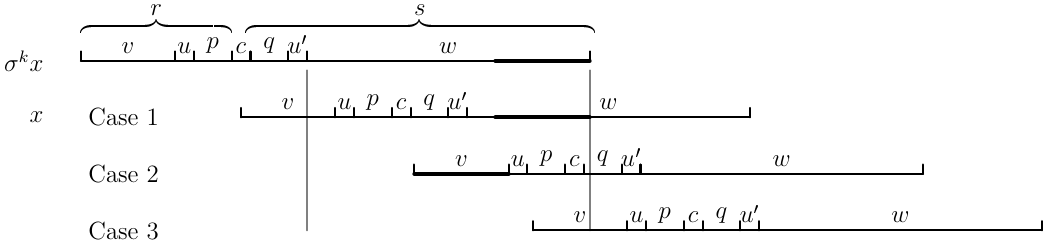}
\caption{Cases 1 and 2 are forbidden by our choice of $v,w$.  Case 3 is permissible.}
\label{fig:short-overlaps}
\end{figure}

The first half of \eqref{eqn:mn} gives $\abs{vupcqu'} \leq \ell m + 2\tau + \abs{pcq} \leq \alpha(\ell n - \abs{c}) = \alpha\abs{w}$.  In particular, if Case 1 occurs then we have $w_{[1,\abs{w}-k]} = w_{[k+1,\abs{w}]}$ for some $k\leq \abs{vupcqu'} \leq \alpha\abs{w}$, so $w$ is $k$-periodic, contradicting Lemma \ref{lem:v-not-in-w}.  Thus Case 1 does not occur.  Case 2 is forbidden since by Lemma \ref{lem:v-not-in-w}, $v$ does not appear as a subword of $w$.  Thus \eqref{eqn:xrcs} is not satisfied for any $1\leq k\leq \abs{rcs} - \abs{v}$.  Since $\abs{rc} \leq \abs{cs} = \abs{rcs} - \abs{r} \leq \abs{rcs} - \abs{v}$, this proves \eqref{eqn:short-overlaps}.

For the final claim in part \ref{c}, note that when $\GGG$ satisfies \ref{s-spec} for some $\tau\in\NN$, Lemma \ref{lem:reset} gives a synchronising triple $(p,c,q)$ with $\abs{c}=\tau$.  As above, this extends to a synchronising triple $(r,c,s)$ satisfying \eqref{eqn:short-overlaps}.
Let $F = \{\abs{w} : w\in \FFF^{r,c,s}\} \subset \NN$ and $G = \{\abs{w} : w\in \GGG\} \subset \NN$.  Clearly $F \subset \abs{c} + G = \tau+G$, so $\gcd(F)$ is a multiple of $d := \gcd(\tau+G)$.
On the other hand, $\tau + G \subset \NN$ is closed under addition by \ref{s-spec}, so there is $N\in \NN$ such that $d \NN  \cap [N,\infty) \subset \tau + G$.  For every $w\in \GGG$ we have $csuwvr\in \FFF^{r,c,s}$ for some $u,v\in \LLL_\tau$, and so $F \supset \tau + G + (2\tau + \abs{r} + \abs{s}) \supset d \NN \cap [N+2\tau + \abs{r}+\abs{s},\infty)$, which gives $\gcd(F) = d$.

\subsubsection{Absence of long overlaps implies \ref{overlaps}}\label{sec:overlaps}

Now we prove part \ref{d} of Proposition \ref{prop:get-free-2}.
Let $(r,c,s)$ be a synchronising triple for $\GGG$ satisfying \eqref{eqn:short-overlaps}.  We show that
\[
\FFF := \FFF^{r,c,s} = c\BBB^{r,s} = c(s\LLL \cap \LLL r \cap \GGG)
\]
satisfies \ref{overlaps} if $\GGG$ satisfies \ref{inter}.  Note that $\FFF$ satisfies \ref{free} by part \ref{a}.

\begin{figure}[htbp]
\includegraphics[width=200pt]{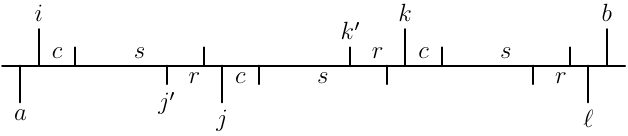}
\caption{Establishing \ref{overlaps}.}
\label{fig:get-overlaps}
\end{figure}

Suppose $x\in X$ and $i\leq j < k \leq \ell$ are integers such that $x_{[i,k)},x_{[j,\ell)}\in \FFF$ and there are $a<j$ and $b>k$ such that $x_{[a,j)},x_{[k,b)}\in \FFF$ (see Figure \ref{fig:get-overlaps}).  We must show that $x_{[j,k)}\in \FFF$.
Let $j' = j-\abs{r}$ and $k'=k-\abs{r}$; then we have
\[
x_{[j',j'+\abs{rcs})} = x_{[k',k'+\abs{rcs})} = rcs,
\]
so  \eqref{eqn:short-overlaps} gives $k-j = k'-j' \geq \max(\abs{rc},\abs{cs})$.  Thus $x_{[j,j+\abs{c})} = c$ and $x_{[j+\abs{c},k)} \in s\LLL\cap \LLL r$.  

It remains only to show that $x_{[j+\abs{c},k)}\in \GGG$.  For this we observe that $x_{[i,k)}\in \FFF$ implies $x_{[i+\abs{c},k)}\in \GGG$, and $x_{[j,\ell)}\in \FFF$ implies $x_{[j+\abs{c},\ell)}\in \GGG$.  Note that $i+\abs{c} \leq j+\abs{c} < k \leq \ell$, 
and that $k-j \geq \max(\abs{rc},\abs{cs})$ implies $k-(j+\abs{c}) \geq \max(\abs{r},\abs{s}) \geq L$,
so by \ref{inter} we have $x_{[j+\abs{c},k)}\in\GGG$.  It follows that $x_{[j,k)}\in\FFF$, which establishes \ref{overlaps} for $\FFF$.  This completes the proof of Proposition \ref{prop:get-free-2}.

\subsection{Construction of $\Ep$ and $\Es$}\label{sec:get-gap}

In this section we prove Proposition \ref{prop:get-gap}, modulo a few lemmas that are proved in \S\ref{sec:endcaps}.  The argument here uses ideas similar to those in the proof of Lemma \ref{lem:hat-I}, but is more complicated.

%In the non-uniform setting we still follow the basic plan from \S\ref{sec:uniform}, but a number of complications arise.
Suppose $\GGG$ satisfies \ref{spec}, \ref{gap}, \ref{inter}, \ref{union} and let $r,c,s$ be a synchronising triple satisfying \eqref{eqn:short-overlaps} and such that $\abs{r},\abs{s} \geq L$, where $L$ is large enough for both \ref{inter} and \ref{union} to hold.
Let $\FFF = \FFF^{r,c,s}$ and $I = \FFF\setminus \FFF\FFF$; we must prove that $P(I,\ph)<P(\ph)$ and produce $\Ep,\Es\subset \LLL$ such that
\begin{equation}
\label{eqn:gapEFE}
P(\Ep \cup \Es \cup (\LLL\setminus \Ep \FFF \Es),\ph) < P(\ph). \end{equation}
To this end, given $w\in \GGG$, consider the set of times where the synchronising triple appears in a `good' position relative to the start of the word:
\begin{equation}\label{eqn:Gsw}
\Gs(w) := \{i\in [\abs{r},\abs{w}-\abs{cs}] \mid w_{[1,i]}\in \GGG \text{ and }
w_{(i-\abs{r}, i+\abs{cs}]}= rcs \}.
\end{equation}
Let $S(w) = \{i\in \Gs(w) \mid w_{(i+\abs{c},\abs{w}]}\in \GGG\}$ be the set of such times that are also `good' relative to the end of the word.  Then consider the collection
\[
\EEE := \{w\in \GGG \mid  S(w)=\emptyset\}.
\]
Note that $\FFF \subset c\GGG$, and if $cw\in \FFF$ has $S(w) \neq \emptyset$, then for $i\in S(w)$ we have
\[
w_{[1,i]} \in \GGG,\quad
w_{(i-\abs{r},i+\abs{cs}]} = rcs,\quad
w_{(i+\abs{c},\abs{w}]}\in \GGG
\quad
\Rightarrow
\quad
cw = cw_{[1,i]} c w_{(i+\abs{c},\abs{w}]} \in \FFF\FFF;
\]
this shows that $I = \FFF\setminus \FFF\FFF\subset c\EEE$, so $P(I,\ph) \leq P(\EEE,\ph)$ by Lemma \ref{lem:CD}.  %Thus it remains to describe $\LLL\setminus \Ep \FFF \Es$, and then to estimate $P(\EEE,\ph)$.

Let $\Ep := \Cp \EEE$ and $\Es := \EEE \Cs$.  
%Then $P(\Ep \cup \Es,\ph) \leq \max\{P(\EEE,\ph), P(\Cp,\ph), P(\Cs,\ph)\}$ by Lemma \ref{lem:CD}.  
We must describe $\LLL\setminus \Ep \FFF \Es$.
By \ref{inter} and \eqref{eqn:short-overlaps}, given $w\in \GGG$ and $i,j\in S(w)$, we have $w_{(i,j]} \in \FFF$.  
Given $v\in \Cp \GGG \Cs$ we write $v = u^p w u^s$ for $u^p\in \Cp$, $w\in \GGG$, and $u^s\in \Cs$. 
  Then we have the following trichotomy: either
\begin{itemize}
\item $\#S(w)=0$, so $w\in \EEE$ and $v\in \Cp \EEE \Cs = \Cp \Es$; or
\item $\#S(w)= 1$, so $w\in \EEE\EEE$ and $v \in \Cp \EEE\EEE \Cs = \Ep \Es$; or
\item $\#S(w)>1$, in which case we take $i = \min S(w)$ and $j=\max S(w)$ to obtain $w_{[1,i]} \in \EEE$, $w_{(i,j]} \in \FFF$, and $w_{(j,\abs{w}]}\in \EEE$.  In particular, this gives $v\in \Cp \EEE \FFF \EEE \Cs \subset \Ep \FFF \Es$.
\end{itemize}
We conclude that $\LLL\setminus  \Ep\FFF\Es \subset \Cp \Es \cup\Ep \Es \cup (\LLL\setminus \Cp\GGG\Cs)$, so Lemma \ref{lem:CD} gives
\[
P(\Ep \cup \Es \cup (\LLL\setminus \Ep\FFF\Es),\ph) \leq 
\max\{P(\EEE,\ph), P(\Cp,\ph), P(\Cs,\ph)\}.
\]
By \ref{gap} and the fact that $P(I,\ph) \leq P(\EEE,\ph)$, we see that in order to prove \ref{gap-2} it suffices to show that $P(\EEE,\ph) < P(\ph)$.

As in \eqref{eqn:GH+}, let $\GGG^+ := \{w_{[1,i]} \mid w\in \GGG, 1\leq i\leq \abs{w}\}$, and similarly $\GGG^- := \{w_{[i,\abs{w}]} \mid w\in \GGG, 1\leq i\leq \abs{w}\}$.  The following lemma (proved in \S\ref{sec:endcaps}) says that with very few exceptions, words in $\GGG^+$ admit a decomposition with no prefix, and words in $\GGG^-$ admit a decomposition with no suffix.  In the proof, and below, we will use the collection
\begin{equation}\label{eqn:C}
\CCC := \Cp \cup \Cs \cup (\LLL\setminus \Cp \GGG\Cs).
\end{equation}

\begin{lemma}\label{lem:endcaps}
If $\GGG$ satisfies \ref{spec}, \ref{gap}, and \ref{inter} with some $\Cp,\Cs\subset \LLL$,  then $P(\GGG^+ \setminus (\GGG \Cs),\ph) < P(\ph)$, and similarly $P(\GGG^- \setminus (\Cp \GGG),\ph) < P(\ph)$.  In particular, there are $\xi>0$ and $Q_7>0$ such that for every $n\in \NN$ we have
\begin{equation}\label{eqn:GCG}
\begin{aligned}
\Lambda_n(\GGG^+ \setminus (\GGG\Cs),\ph) &\leq Q_7 e^{n(P(\ph)-\xi)}, \\
\Lambda_n(\GGG^- \setminus (\Cp\GGG),\ph) &\leq Q_7 e^{n(P(\ph)-\xi)}.
\end{aligned}
\end{equation}
\end{lemma}

\begin{figure}[htbp]
\includegraphics[width=300pt]{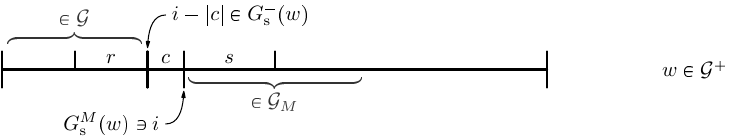}
\caption{Candidates for good occurrences of $(r,c,s)$.}
\label{fig:GsM}
\end{figure}

Given $w\in \GGG^+$, define $\Gs(w)$ as in \eqref{eqn:Gsw}.  For $M\in \NN$ and $w\in \GGG^+$, let
\[
\GsM(w) := \{i\in [\abs{rc},\abs{w}-\abs{s}] \mid i - \abs{c} \in \Gs(w) \text{ and } w_{(i,i+M]} \in \GGG\},
\]
as illustrated in Figure \ref{fig:GsM}.
The indices $i\in \GsM(w)$ are `candidates' for good occurrences of the synchronising triple $(r,c,s)$.  Note that here $i$ represents the position where $s$ starts, rather than the position where $r$ ends; this differs from our convention with $\Gs(w)$ but will be more convenient in what follows.  To guarantee that an index $i\in \GsM(w)$ represents a genuinely good occurrence (that is, $i-\abs{c}\in S(w)$), we will eventually need the added property that there is no long obstruction beginning at $i$; that is, $w_{(i,i']}\notin \Cp$ for any $i' \geq i + M - L$.  By Lemma \ref{lem:endcaps} this will (typically) guarantee existence of $j\in (i, i+M-L]$ such that $w_{(j,\abs{w}]}\in \GGG$; this in turn will allow us to apply \ref{union} and deduce that $w_{(i,\abs{w}]}\in \GGG$ and hence $i-\abs{c}\in S(w)$.

\begin{figure}[htbp]
\includegraphics[width=300pt]{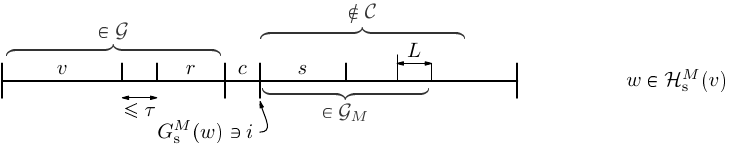}
\caption{Obtaining $i\in S(w)$ from $i\in \GsM(w)$.}
\label{fig:HsM}
\end{figure}

To show that nearly every word $w$ has indices in $\GsM(w)$, we will work left to right, estimating the probability that the synchronising word appears soon in a good position, conditioned on the symbols we have seen so far.  %Write $\CCC := \Cp \cup \Cs \cup (\LLL \setminus \Cp \GGG \Cs)$.  
For each $v\in \GGG$, let $\HHH^+(v) := \GGG^+ \cap v\LLL$ and consider the following collection, illustrated in Figure \ref{fig:HsM}:
\begin{multline*}
\HsM(v) := \{w\in \HHH^+(v) \mid \text{there is } i\in [\abs{v}+\abs{rc}, \abs{v}+\abs{rc}+\tau] \cap \GsM(w) \\
\text{ such that } w_{(i,i']}\notin \CCC \text{ for any } i'\in [i+M-L,\abs{w}]\},
\end{multline*}
Here $\CCC$ is as in \eqref{eqn:C}.
The following is proved in \S\ref{sec:chunk-of-H} and is the source of the exponential decay that we eventually obtain.

\begin{lemma}\label{lem:chunk-of-H}
There is $\gamma>0$ such that there are arbitrarily large values of $M\in \NN$ such that the following holds for every $v\in \GGG$ and $n\geq \abs{v} + \tau + \abs{rc} + M$:
\begin{equation}\label{eqn:chunk-of-H}
\Lambda_n(\HsM(v),\ph) \geq \gamma \Lambda_n(\HHH^+(v),\ph).
\end{equation}
\end{lemma}

In order to apply Lemma \ref{lem:chunk-of-H}, we need to show that for $w\in \GGG^+$, we have $w_{[1,i]}\in \GGG$ `often enough'.  Thanks to Lemma \ref{lem:endcaps}, we can do this by controlling how many words can have long segments covered by a small number of words in $\CCC$.   To this end, given $\delta,\beta >0$, consider the collection
\begin{multline*}
\hC^{\beta,\delta} := \{w\in \LLL \mid \text{there are } \{ (i_a,i_a'] \subset [1,\abs{w}]\}_{a=1}^ B \text{ with } B\leq 2 \lfloor \beta\abs{w} \rfloor\\
\text{such that } w_{(i_a,i_a']}\in \CCC \text{ and } \# \textstyle\bigcup_{a=1}^B (i_a, i_a'] \geq 2\delta \abs{w} \}
\end{multline*}
of words for which at least $2\delta$ of the length of the word can be covered by a small ($\leq 2 \lfloor \beta\abs{w} \rfloor$) number of subwords lying in $\CCC$.  When $\beta \ll \delta$, the following estimate (proved in \S\ref{sec:hatC}) shows that $\hC^{\beta,\delta}$ has small pressure, and gives a concrete estimate on the partition sum.

\begin{lemma}\label{lem:hatC}
For every $\delta>0$ there is $\beta>0$ such that $P(\hC^{\beta,\delta},\ph) < P(\ph)$.  In particular, there are $Q_8>0$ and $\theta<1$ such that
\begin{equation}\label{eqn:hatC}
\Lambda_n(\hC^{\beta,\delta},\ph) \leq Q_8 \theta^n e^{nP(\ph)}
\end{equation}
for every $n\in \NN$.  Moreover, because $\hC^{\beta',\delta} \subset \hC^{\beta,\delta}$ for every $0 < \beta' < \beta$, \eqref{eqn:hatC} remains true if $\beta$ is replaced with any smaller positive number.
\end{lemma}

%Given $M\in \NN$, put $T = 2(M+\abs{rc}+\tau)$.  As in \S\ref{sec:uniform}, we will estimate the proportion of words that `reset' with the synchronising triple in a good position by time $kT$ for each $k\in \NN$. 

Let $\delta = 1/8$, and let $\beta>0$ be such that Lemma \ref{lem:hatC} holds.  Choose an integer $M\geq \max(\beta^{-1},2(\abs{rc}+\tau))$ such that
Lemma \ref{lem:chunk-of-H} holds and 
such that writing $T = 2(M+\abs{rc}+\tau)$, the quantity
 $\zeta := (\max(\theta,e^{-\xi}))^T$ satisfies\footnote{Recall that $\xi, \gamma, \theta$ are provided by Lemmas \ref{lem:endcaps}, \ref{lem:chunk-of-H}, and \ref{lem:hatC}, respectively.  Note that in Lemma \ref{lem:chunk-of-H}, $\gamma$ is independent of the choice of $M$.}
\begin{equation}\label{eqn:zeta-gamma}
Q_2 (Q_7 +Q_8) \zeta(1-\theta)^{-1}(1-\zeta)^{-1} < \gamma.
\end{equation}
Observe our choice of $M$ guarantees that $\frac 1{2T} < \beta$ and that $T\leq 3M$; we will use this in Lemma \ref{lem:AkZk} below.  
For every $k\in \NN$ and $m \geq kT$, Lemmas \ref{lem:endcaps} and \ref{lem:hatC} give
\begin{equation}\label{eqn:QQ}
\begin{aligned}
\Lambda_m(\GGG^+ \setminus (\GGG \Cs),\ph) &\leq Q_7 \zeta^k e^{mP(\ph)}, \\
\Lambda_m(\hC^{\beta,\delta},\ph) &\leq Q_8 \zeta^k e^{mP(\ph)}.
\end{aligned}
\end{equation}
Consider the intervals
$
I_j := (jT, (j+\frac 12)T]
$
for $j\in \NN$, and write
\begin{multline}\label{eqn:AAAk}
\AAA^k := \{w\in \GGG^+ \mid \text{for all } 0\leq j < k \text{ and } i\in I_j \cap \GsM(w), \\
\text{there is } 
i'\in [i+M-L,\abs{w}] \text{ such that } w_{(i,i']}\in \CCC \}
\end{multline}
for the set of words which \defn{avoid} genuinely good occurrences of $(r,c,s)$ in the first $k$ intervals $I_j$; in particular, any candidate good occurrence within these intervals is `ruined' sometime before the end of $w$ by a long obstruction from $\CCC$.
Note that $\AAA^{k+1} \subset \AAA^k$, and that if $v\in \AAA^k$ and $\abs{v}\geq kT$, then $\HHH^+(v) \subset\AAA^k$.  On the other hand, there may be $w\in \AAA^k$ with $w_{[1,kT]}\notin \AAA^k$.

Our goal is to relate $\Lambda_{n}(\AAA^{k+1},\ph)$ and $\Lambda_{n}(\AAA^k,\ph)$.  We do this by decomposing (most of) $\AAA^k$ into collections $\HHH^+(v)$ where $v\in \GGG$, and then applying Lemma \ref{lem:chunk-of-H}.  Given $k\in \NN$, write
\begin{equation}\label{eqn:mkZk}
\begin{aligned}
m_k &:= (k+\tfrac 12)T - (\tau + \abs{rc}), \\
\ZZZ^k &:= \{v\in \GGG \mid \abs{v} \in [kT,m_k], v_{[1,kT]}\in \AAA^k, v_{[1,i]}\notin \GGG \text{ for any } kT\leq i < \abs{v}\}.
\end{aligned}
\end{equation}
We will use the following observations.
\begin{enumerate}[leftmargin=*, widest=2]
\item The collections $\{\HHH^+(v) \mid v\in \ZZZ^k\}$ are disjoint; no word in $\ZZZ^k$ is a prefix of any other.  Indeed, given $v'\in \ZZZ^k$ and $v\in v'\LLL \cap \GGG$ with $v\neq v'$, we have $kT \leq \abs{v'} < \abs{v}$ and $v_{[1,\abs{v'}]} = v' \in \GGG$, so $v\notin \ZZZ^k$.
\item Given $v\in \ZZZ^k$ and $n\geq (k+1)T$, we have $\HsM(v)_n \subset \HHH_n^+(v) \setminus \AAA^{k+1}$.   To see this, observe that $\HsM(v) \subset \HHH^+(v)$ by definition, and that moreover any $w\in \HsM(v)$ has some $i\in \GsM(w)$ such that
\begin{itemize}
\item $i\in [\abs{v}+\abs{rc},\abs{v}+\abs{rc}+\tau] \subset I_k$ (this last inclusion motivates our choice of $m_k$), and
\item $w_{(i,i']}\notin \CCC$ for any $i'\in [i+M-L,\abs{w}]$.
\end{itemize}
The existence of such an $i$ guarantees that $w\notin \AAA^{k+1}$.
\end{enumerate}
The following lemma, proved in \S\ref{sec:AkZk}, relates $\AAA^k$ to $\bigsqcup_{v\in \ZZZ^k} \HHH^+(v)$ up to a small term whose partition sum is well controlled by Lemmas \ref{lem:endcaps} and \ref{lem:hatC}.

\begin{lemma}\label{lem:AkZk}
For every $k\geq 0$ and $n\geq m_k$, we have
\begin{equation}\label{eqn:AkZk}
\bigsqcup_{v\in \ZZZ^k} \HHH_n^+(v) \subset \AAA_n^k
\subset \bigg( \bigsqcup_{v\in \ZZZ^k} \HHH_n^+(v) \bigg) \cup \YYY_n^k,
\end{equation}
where 
\[
\YYY_{n}^k := 
 (\GGG^+\setminus(\GGG\Cs))_{m_k} \LLL_{n-m_k} \cup 
 \bigg(\bigcup_{j=0}^{k-1} \bigcup_{\ell=kT}^n \AAA^j_{jT} \hC^{\beta,\delta}_{\ell-jT} \LLL_{n-\ell}\bigg).
\]
Moreover, $\YYY^k := \bigcup_{n\in \NN} \YYY_n^k$ satisfies the bound (with $\theta,\zeta$ as in \eqref{eqn:hatC}--\eqref{eqn:QQ})
\begin{equation}\label{eqn:Ynk}
\Lambda_{n}(\YYY^k,\ph) \leq Q_2 \cdot \frac{Q_7+Q_8}{1-\theta} \sum_{j=0}^{k-1} \Lambda_{jT}(\AAA^j,\ph) \zeta^{k-j} e^{(n-jT)P(\ph)}.
\end{equation}
\end{lemma}

Now we estimate $\Lambda_{n}(\AAA^k,\ph)$.
Given $v\in \ZZZ^k$ and $n\geq (k+1)T$, we recall from above that $\HsM(v)_n \subset \HHH_n^+(v) \setminus \AAA_n^{k+1}$, and so Lemma \ref{lem:chunk-of-H} gives
\begin{equation}\label{eqn:LnH+}
\Lambda_{n}(\HHH^+(v) \cap \AAA^{k+1},\ph)
\leq \Lambda_{n}(\HHH^+(v) \setminus \HsM(v),\ph)
\leq (1-\gamma) \Lambda_{n}(\HHH^+(v),\ph).
\end{equation}
Summing over $v\in \ZZZ^k$ and using $\AAA^{k+1} \subset \AAA^k$ together with the second half in \eqref{eqn:AkZk} gives
\begin{equation}\label{eqn:Ak1}
\Lambda_{n}(\AAA^{k+1},\ph)\leq \bigg(\sum_{v\in \ZZZ^k} \Lambda_{n}(\HHH^+(v) \cap \AAA^{k+1}, \ph)\bigg) + \Lambda_{n}(\YYY^k,\ph).
\end{equation}
Using the first inclusion in \eqref{eqn:AkZk} together with \eqref{eqn:LnH+} gives
%, the sum in brackets is bounded above by
\[
\sum_{v\in \ZZZ^k} \Lambda_{n}(\HHH^+(v) \cap \AAA^{k+1}, \ph) \leq
\sum_{v\in \ZZZ^k} (1-\gamma) \Lambda_n(\HHH^+(v),\ph)
\leq (1-\gamma) \Lambda_n(\AAA^k,\ph),
\]
and so from \eqref{eqn:Ak1} we obtain
\begin{equation}\label{eqn:Ak2}
\Lambda_n(\AAA^{k+1},\ph) \leq (1-\gamma) \Lambda_n(\AAA^k,\ph) + \Lambda_n(\YYY^k,\ph).
\end{equation}
Let
$a_k := \sup_{n\geq kT} \Lambda_n(\AAA^k,\ph) e^{-nP(\ph)}$
so that $\Lambda_{n}(\AAA^k,\ph) \leq a_k e^{nP(\ph)}$ for all $n\geq kT$.  Lemma \ref{lem:counting} gives $a_k \leq a_0 \leq Q_2$; we will prove  that $a_k \leq Q_2 \eta^k$ for some $\eta < 1$.  Write \eqref{eqn:Ynk} as
\[
\Lambda_n(\YYY^k,\ph) \leq Q_2\frac{Q_7+Q_8}{1-\theta} \sum_{j=0}^{k-1} a_j \zeta^{k-j} e^{nP(\ph)},
\]
and multiply both sides of \eqref{eqn:Ak2} by $e^{-nP(\ph)}$ to get
\begin{equation}\label{eqn:ak}
a_{k+1} \leq (1-\gamma) a_k + Q_2\frac{Q_7+Q_8}{1-\theta} \sum_{j=0}^{k-1} a_j \zeta^{k-j}.
\end{equation}
Rewriting \eqref{eqn:zeta-gamma} as
$(1-\gamma) + Q_2(Q_7+Q_8) \zeta(1-\theta)^{-1}(1-\zeta)^{-1} < 1$,
choose $\eta<1$ such that
\begin{equation}\label{eqn:eta-gamma}
(1-\gamma) \eta^{-1} + Q_2\frac{Q_7+Q_8}{1-\theta} \cdot \frac{\zeta}{\eta(\eta - \zeta)} < 1.
\end{equation}
Let $b_k = a_k \eta^{-k}$ and observe that $b_0 = a_0 \leq Q_2$.  Suppose that $k$ is such that $b_j \leq Q_2$ for every $0\leq j\leq k$.  Then \eqref{eqn:ak} gives
\begin{align*}
b_{k+1} &= a_{k+1} \eta^{-(k+1)}
\leq (1-\gamma)a_k \eta^{-(k+1)} + Q_2\frac{Q_7+Q_8}{1-\theta} \eta^{-(k+1)} \sum_{j=0}^{k-1} b_j \eta^j \zeta^{k-j} \\
&\leq (1-\gamma) b_k \eta^{-1} 
+ Q_2\frac{Q_7+Q_8}{1-\theta} \cdot \frac 1\eta \sum_{j=0}^{k-1} Q_2 \bigg( \frac{\zeta}{\eta}\bigg)^{k-j} \\
&\leq Q_2\bigg(
(1-\gamma)\eta^{-1} + Q_2\frac{Q_7+Q_8}{1-\theta} \cdot \frac 1\eta \cdot \frac {\frac{\zeta}\eta}{1-\frac{\zeta}\eta}\bigg) < Q_2,
\end{align*}
where the final inequality uses \eqref{eqn:eta-gamma}.  It follows by induction that $b_k\leq Q_2$ for every $k\in \NN$, and thus $a_k \leq Q_2 \eta^k$ for every $k$.  In particular, this gives 
\begin{equation}\label{eqn:Ak}
\Lambda_{n}(\AAA^k,\ph) \leq Q_2 \eta^k e^{nP(\ph)}
\quad\text{ for every }0\leq kT\leq n.
\end{equation}
It remains to relate $\EEE_{n}$ to $\AAA_n^{k}$.
Given $n\in \NN$ large, choose $k_n \in [\frac n{3T}, \frac n{2T}]\cap \NN$.  We claim that
\begin{equation}\label{eqn:EnAn}
\EEE_n \subset \AAA_n^{k_n} \cup \bigg( \bigcup_{m=0}^{k_n T} \LLL_m (\GGG^- \setminus \Cp \GGG)_{n-m} \bigg).
\end{equation}
Indeed, given $w\in \GGG_n$, suppose that $w$ is not contained in the right-hand side of \eqref{eqn:EnAn}; that is, $w\notin \AAA_n^{k_n}$ and $w_{(m,\abs{w}]}\in \Cp\GGG$ for every $0\leq m\leq k_n T$.  Then by the definition of $\AAA_n^{k_n}$, there are $j<k_n$ and $i\in \GsM(w) \cap I_j$ such that $w_{(i,i']}\notin \CCC$ for any $i'\geq i+M-L$.  Since $w_{(i,\abs{w}]}\in \CCC \GGG$, this implies that there is $\ell\in [i, i+M-L]$ such that $w_{(\ell,\abs{w}]}\in \GGG$.  Applying \ref{union} to $w_{(i,i+M]}$ and $w_{(\ell,\abs{w}]}$, we conclude that $w_{(i,\abs{w}]}\in \GGG$, and hence $i-\abs{c}\in S(w)$.\footnote{This is the only place in the paper where we use \ref{union}.} 
In particular, this proves that $S(w)\neq \emptyset$, so $w\notin \EEE_n$, establishing \eqref{eqn:EnAn}.

Using \eqref{eqn:EnAn} together with \eqref{eqn:GCG} and \eqref{eqn:Ak}, we now have the estimate
\begin{align*}
\Lambda_n(\EEE,\ph) &\leq \Lambda_n(\AAA^{k_n},\ph) + \sum_{m=0}^{k_n T} \Lambda_m(\LLL,\ph) \Lambda_{n-m}(\GGG^-\setminus \Cp \GGG, \ph) \\
&\leq Q_2 \eta^{k_n} e^{nP(\ph)} + \sum_{m=0}^{k_n T} Q_2 e^{mP(\ph)} Q_7 e^{(n-m)(P(\ph) - \xi)} \\
&\leq Q_2 e^{nP(\ph)} \bigg( \eta^{\frac n{3T}} + Q_7 \sum_{\ell = n-k_nT}^\infty e^{-\ell\xi}\bigg) 
\leq Q_2 e^{nP(\ph)} \big( %e^{\frac 1{3T}(\log\eta) n} 
\eta^{\frac n{3T}}
+ Q_7(1-e^{-\xi})^{-1} e^{-\frac n2 \xi} \big).
\end{align*}
Taking logs, dividing by $n$, and sending $n\to\infty$ gives
\[
P(\EEE,\ph) \leq P(\ph) + \max ( \tfrac 1{3T}\log \eta, -\tfrac \xi 2) < P(\ph).
\]
This completes the proof of Proposition \ref{prop:get-gap}, and hence of Theorem \ref{thm:get0spec}, modulo the proofs of Lemmas \ref{lem:endcaps}--\ref{lem:AkZk}, which we give in the next section.
%now in \S\S\ref{sec:endcaps}--\ref{sec:AkZk}.

\subsection{Proofs of Lemmas \ref{lem:endcaps}--\ref{lem:AkZk}}
\subsubsection{Proof of Lemma \ref{lem:endcaps}}\label{sec:endcaps}

We prove the lemma for $\GGG^+ \setminus \GGG \Cs$; the other claim follows from a symmetric argument.
It suffices to show that $P(\GGG^+ \setminus \GGG\Cs,\ph) < P(\ph)$, which we do by finding $\RRR,\SSS\subset \LLL$ and $c\in \NN$ with $\GGG_{\geq c}^+ \setminus \GGG\Cs \subset \RRR\cup \SSS$, proving that $P(\RRR,\ph) < P(\ph)$ and $P(\SSS,\ph) < P(\ph)$, and applying \eqref{eqn:PCD}.

Fix $\eps>0$ small enough that $P(\CCC,\ph) + \eps < P(\ph)$, so there is $C>0$ with 
\begin{equation}\label{eqn:LnC}
\Lambda_n(\CCC,\ph) \leq C e^{n(P(\ph) - \eps)} \text{ for every } n\in \NN.
\end{equation}
Then fix $\alpha>0$ such that $\alpha(P(\ph)+\|\ph\|) < \eps$.
Let $N,Q_3,n_0$ be as in Lemma \ref{lem:counting}, so that given any $n\geq c := \tau + \lceil n_0/\alpha\rceil$ we can choose $m_n \in [\alpha n - \tau - N, \alpha n - \tau ]$ for which $\Lambda_{m_n}(\GGG,\ph) \geq Q_3 e^{m_n P(\ph)}$.  Note that $m_n + \tau  \leq \alpha n$ and $\frac 1n m_n \to \alpha$. 

Given $w\in \GGG_{\geq c}^+$, let $x\in \LLL$ be such that $wx\in \GGG$.  For every $v\in \GGG$, it follows from \ref{spec} that there is $u=u(v,w)\in \LLL_{\leq \tau}$ such that $vuwx\in \GGG$.  From now on we will consider $x$ as a function of $w$, and $u$ as a function of $w,v$.
Consider the collection $\RRR \subset \GGG_{\geq c}^+$ given by
\[
\RRR := \{w\in \GGG_{\geq c}^+ \mid \text{ for all } v\in \GGG_{m_{\abs{w}}} \text{ there is } 1\leq k\leq \abs{w} \text{ such that } vu(w_{[1,k]}) \in \CCC \}.
\]
In particular, for every $n\geq c$ and $w\in \GGG_n^+ \setminus \RRR$, there are $v\in \GGG_{m_n}$ and $u\in \LLL_{\leq \tau}$ such that $vuw\in \GGG^+$ and $vu(w_{[1,k]})\notin \CCC$ for every $1\leq k\leq \abs{w}$.  We will consider $v,u$ as functions of $w$ whenever $w\in \GGG^+\setminus \RRR$. Let
\[
\SSS := \{ w\in \GGG_{\geq c}^+ \setminus \RRR \mid \abs{w}\geq L \text{ and there is } 1\leq j \leq \abs{vu} + L \text{ such that } (vuw)_{(j,\abs{vuw}]}\in \CCC \}.
\]
We will demonstrate below that $P(\RRR \cup \SSS,\ph) < P(\ph)$.  First we observe that if $w\in \GGG_{\geq c}^+ \setminus (\RRR \cup \SSS)$, then we have $vuw\notin \CCC$ and hence $vuw \in \Cp \GGG \Cs$, so that in particular there are $1\leq i \leq j \leq \abs{vuw}$ such that
\[
(vuw)_{[1,i]}\in \Cp, \qquad (vuw)_{(i,j]}\in \GGG, \qquad (vuw)_{(j,\abs{vuw}]}\in \Cs.
\]
By the choice of $v$ and $u$, we have $i\leq \abs{vu}$, and by the definition of $\SSS$, we have $j > \abs{vu} + L$.  In particular, writing $\ell = j - \abs{vu}$, we see that $w_{(\ell,\abs{w}]} = (vuw)_{(j,\abs{vuw}]}\in \Cs$, and also $w_{[1,\ell]}$ is the intersection of the two words $(vuwx)_{(i,j]}\in \GGG$ and $(vuwx)_{(\abs{vu},\abs{vuwx}]} = wx\in \GGG$.  Since $\ell\geq L$ this gives $w_{[1,\ell]}\in \GGG$ (by \ref{inter}) and thus $w\in \GGG \Cs$.

Having proved that $\GGG_{\geq c}^+ \setminus \GGG\Cs \subset \RRR \cup \SSS$, it remains to estimate $P(\RRR,\ph)$ and $P(\SSS,\ph)$.
To estimate $\Lambda_n(\RRR,\ph)$, we will estimate the partition sum of the collection $\{vuw \mid v\in \GGG_{m_n}, w\in \RRR_n\}$ in two different ways.  First, note that every such $vuw$ has length between $n+m_n$ and $n+m_n+\tau$, so along the same lines as in the proof of Lemma \ref{lem:counting}, we have
\begin{multline}\label{eqn:Gmn1}
\sum_{v\in \GGG_{m_n}} \sum_{w\in \RRR_n} e^{\hat\ph(vuw)} 
\geq \sum_{v\in \GGG_{m_n}} \sum_{w\in \RRR_n} e^{\hat\ph(v)} e^{\hat\ph(w)} e^{-(\tau\|\ph\| + \absd{\ph})} \\
\geq e^{-(\tau\|\ph\| + \absd{\ph})}  \Lambda_{m_n}(\GGG,\ph) \Lambda_n(\RRR,\ph)
\geq e^{-(\tau\|\ph\| + \absd{\ph})} Q_3 e^{m_n P(\ph)} \Lambda_n(\RRR,\ph).
\end{multline}
On the other hand, for $v\in \GGG_{m_n}$ and $w\in \RRR_n$, we have $vuw\in \CCC_{m_n + \abs{u} +k} \LLL_{n - k}$ for some $1\leq k\leq n$, and thus as in \eqref{eqn:CD},
\begin{equation}\label{eqn:Gmn2}
\begin{aligned}
\sum_{v\in \GGG_{m_n}} \sum_{w\in \RRR_n} e^{\hat\ph(vuw)} 
&\leq \sum_{k=1}^n \sum_{t=0}^\tau \sum_{x\in \CCC_{m_n + t + k}} \sum_{y\in \LLL_{n- k}} e^{\hat\ph(x)} e^{\hat\ph(y)} \\
&\leq \sum_{k=1}^n \sum_{t=0}^\tau Ce^{(m_n + t + k)(P(\ph)-\eps)}
Q_2 e^{(n-k)P(\ph)}.
\end{aligned}
\end{equation}
Observe that for every $1\leq k\leq n$ and $0\leq t\leq \tau$ we have
\begin{multline*}
(n-k)P(\ph) + (m_n + t + k)(P(\ph)-\eps)  \\
= (n+m_n + t)P(\ph) - (m_n + t + k)\eps 
\leq (n+m_n+\tau)P(\ph) - m_n \eps
\end{multline*}
and so \eqref{eqn:Gmn2} gives
\[
\sum_{v\in \GGG_{m_n}} \sum_{w\in \RRR_n} e^{\hat\ph(vuw)} 
\leq n(\tau + 1) Q_2 C e^{(n+m_n+\tau)P(\ph)} e^{-m_n\eps},
\]
which together with \eqref{eqn:Gmn1} gives
\begin{align*}
\Lambda_n(\RRR,\ph) &\leq e^{\tau\|\ph\| + \absd{\ph}} Q_3^{-1} e^{-m_nP(\ph)} n(\tau+1) Q_2 C e^{(n+m_n+\tau)P(\ph)} e^{-m_n\eps} \\
&\leq e^{\tau\|\ph\| + \absd{\ph}} Q_3^{-1} n(\tau + 1) Q_2 C e^{(n+\tau)P(\ph)} e^{-m_n\eps},
\end{align*}
and we conclude that $P(\RRR,\ph) \leq P(\ph) - \alpha\eps < P(\ph)$.

Now we consider $\Lambda_n(\SSS,\ph)$.  For every $w\in \SSS_n$ we have $v\in \GGG_{m_n}$ and $u\in \LLL_t$ for some $t\leq \tau$ such that $vuw \in \LLL_j \CCC_{\abs{vuw}-j}$ for some $j\leq \abs{vu} +L$.  This gives
\begin{multline}\label{eqn:LnS}
\Lambda_n(\SSS,\ph) = \sum_{w\in \SSS_n} e^{\hat\ph(w)}
\leq \sum_{w\in \SSS_n} e^{(\tau + m_n)\|\ph\| + \absd{\ph}} e^{\hat\ph(vuw)} \\
\leq e^{\alpha n \|\ph\| + \absd{\ph}} 
\sum_{t=0}^\tau \sum_{j=1}^{m_n + t + L} \Lambda_j(\LLL,\ph) \Lambda_{m_n+t+n-j}(\CCC,\ph).
\end{multline}
For each choice of $t,j$ we get
\begin{multline*}
\Lambda_j(\LLL,\ph) \Lambda_{m_n + t + n - j}(\CCC,\ph)
\leq Q_2 e^{jP(\ph)} C e^{(m_n + t + n - j)(P(\ph) - \eps)}
\\
= Q_2 C e^{(m_n + t) P(\ph)} e^{nP(\ph)} e^{-\eps(n + m_n + t - j)}
\leq Q_2 C e^{\alpha n P(\ph)} e^{nP(\ph)} e^{-\eps \ell}
\end{multline*}
where $\ell = n + m_n + t - j \geq n - L$, and so
\[
\sum_{j=1}^{m_n + t + L}\Lambda_j(\LLL,\ph) \Lambda_{m_n + t + n - j}(\CCC,\ph)
\leq Q_2 C e^{\alpha nP(\ph)} e^{nP(\ph)} e^{-\eps(n-L)} (1-e^{-\eps})^{-1}.
\]
 Together with \eqref{eqn:LnS}, this gives
\[
\Lambda_n(\SSS,\ph) \leq e^{\alpha n (P(\ph) + \|\ph\|)} e^{\absd{\ph}} Q_2 C e^{nP(\ph)} (\tau + 1) e^{-\eps(n-L)} (1-e^{-\eps})^{-1},
\]
and we conclude that $P(\SSS,\ph) \leq P(\ph) + \alpha(P(\ph) + \|\ph\|) - \eps < P(\ph)$, where the last inequality uses our choice of $\alpha$.  This proves Lemma \ref{lem:endcaps}.

\subsubsection{Proof of Lemma \ref{lem:chunk-of-H}}\label{sec:chunk-of-H}

We first describe the values of $M$ that we will use, then give a computation that proves \eqref{eqn:chunk-of-H} (and shows that $\gamma$ is independent of $M$).
By Lemma \ref{lem:counting} there are arbitrarily large values of $\ell$ such that $\Lambda_\ell(\GGG,\ph) \geq Q_3 e^{\ell P(\ph)}$.  Given such an $\ell$, define $\pi\colon \GGG_\ell \to \bigsqcup_{i=0}^\tau s\LLL \cap \GGG_{\abs{s}+\ell+i}$ using \ref{spec} by $\pi(w) = suw$ where $u=u(w)\in \LLL_{\leq \tau}$, and $s$ is from the synchronising triple.  Then we have
\begin{align*}
\sum_{i=0}^\tau \Lambda_{\abs{s} + \ell  + i}(s\LLL \cap \GGG,\ph) 
&\geq \sum_{w\in \GGG_\ell } e^{\hat\ph(suw)}
\geq \sum_{w\in \GGG_\ell } e^{\hat\ph(su)} e^{\hat\ph(w)} e^{-\absd{\ph}} \\
&\geq e^{-(\abs{s}+\tau)\|\ph\|-\absd{\ph}} \Lambda_\ell (\GGG,\ph),
\end{align*}
so in particular there is $M\in [\abs{s}+\ell ,\abs{s}+\ell +\tau]$ such that
\begin{align*}
\Lambda_M(s\LLL \cap \GGG,\ph) &\geq (\tau+1)^{-1}e^{-(\abs{s}+\tau)\|\ph\|-\absd{\ph}}Q_3 e^{\ell P(\ph)} \\
&\geq (\tau+1)^{-1}e^{-(\abs{s}+\tau)(\|\ph\|+P(\ph))-\absd{\ph}}Q_3 e^{MP(\ph)}.
\end{align*}
Putting $Q_9 = (\tau+1)^{-1}e^{-(\abs{s}+\tau)(\|\ph\|+P(\ph))-\absd{\ph}}Q_3$, we will use this in the form
\begin{equation}\label{eqn:sLG}
\Lambda_M(s\LLL \cap \GGG,\ph) \geq Q_9 e^{MP(\ph)}.
\end{equation}
Note that $M$ can be taken arbitrarily large, and that $Q_9$ only depends on $Q_3,\tau,\abs{s},\ph$.

Now we fix $v\in \GGG$ and  $n\geq \abs{v} + \tau + \abs{rc} + M$.
By \ref{spec} there is $p\in \LLL_{\leq \tau}$ such that $vpr\in \GGG$.  Write $i = \abs{vprc}$ and $m=n-i \geq M$.  For any $s'\in s\LLL \cap \GGG_M$ and $w\in \HHH^+_m(s')$, the synchronising property of $(r,c,s)$ gives
\[
vprcw \in \widetilde\HsM(v,i) := \{x\in \HHH^+_n(vprc) \mid i \in \GsM(x)\};
\]
%where we use the synchronising property of $(r,c,s)$.  
that is, $i$ represents a `candidate' good occurrence of the synchronising triple 
in $vprcw$.  We first estimate how many such words $vprcw$ there are, and then show that in most of them, $i$ must be a genuinely good occurrence; that is, there is typically not a long subword in $\CCC$ beginning in position $i$.
To start, Proposition \ref{prop:Gibbs} gives
\begin{equation}\label{eqn:H+s'}
\Lambda_m(\HHH^+(s'),\ph) \geq Q_4^{-1} e^{(m-M)P(\ph)} e^{\hat\ph(s')}
\end{equation}
so we have
\begin{align*}
\Lambda_n(\widetilde\HsM(v,i),\ph) &\geq \sum_{s'\in s\LLL\cap \GGG_M} \sum_{w\in \HHH^+_m(s')} e^{\hat\ph(vprcw)} 
\geq \sum_{s'\in s\LLL\cap \GGG_M} \sum_{w\in \HHH^+_m(s')} e^{\hat\ph(vprc)} e^{\hat\ph(w)} e^{-\absd{\ph}} \\
&\geq e^{\hat\ph(v) - \abs{prc}\|\ph\| - 2\absd{\ph}} \sum_{s'\in s\LLL\cap \GGG_M} \Lambda_m(\HHH^+(s'),\ph).
\end{align*}
Using the bounds from \eqref{eqn:sLG} and \eqref{eqn:H+s'}, we get
\begin{multline*}
\sum_{s'\in s\LLL\cap \GGG_M} \Lambda_m(\HHH^+(s'),\ph)
\geq \sum_{s'\in s\LLL\cap \GGG_M} Q_4^{-1} e^{(m-M)P(\ph)} e^{\hat\ph(s')} \\
\geq Q_4^{-1} e^{(m-M)P(\ph)} \Lambda_M(s\LLL\cap \GGG,\ph) 
\geq Q_4^{-1} e^{mP(\ph)} Q_9,
\end{multline*}
which gives
\begin{equation}\label{eqn:wHsM}
\Lambda_n(\widetilde\HsM(v,i),\ph) \geq e^{\hat\ph(v)} e^{mP(\ph)}e^{- \abs{prc}\|\ph\| - 2\absd{\ph}}Q_4^{-1}Q_9.
\end{equation}
The upper bound in Proposition \ref{prop:Gibbs} gives
\begin{equation}\label{eqn:H+v}
\Lambda_n(\HHH^+(v),\ph) \leq Q_4 e^{(n-\abs{v})P(\ph)} e^{\hat\ph(v)} \leq Q_4 e^{(\tau + \abs{rc})P(\ph)} e^{mP(\ph)} e^{\hat\ph(v)};
\end{equation}
together with \eqref{eqn:wHsM} this gives
\begin{equation}\label{eqn:wHsM-2}
\Lambda_n(\widetilde\HsM(v,i),\ph) \geq 2\gamma \Lambda_n(\HHH^+(v),\ph)
\end{equation}
for $\gamma = \frac 12 e^{-(\tau+\abs{rc})(P(\ph)+\|\ph\|) - 2\absd{\ph}} Q_4^{-2}Q_9$.
Applying Proposition \ref{prop:Gibbs} again gives 
\begin{equation}\label{eqn:H+v2}
\Lambda_n(\HHH^+(v),\ph) \geq Q_4^{-1} e^{(n-\abs{v})P(\ph)} e^{\hat\ph(v)}
\geq Q_4^{-1} e^{mP(\ph)}e^{\abs{rc}P(\ph)} e^{\hat\ph(v)}.
\end{equation}
We must exclude those words $x\in \widetilde\HsM(v,i)$ for which $x_{(i,i']}\in \CCC$ for some $i'\in [i+M-L,n]$.  Let $\BBB(v,i)$ be the set of such $x$; note that 
$\BBB(v,i) \subset \bigcup_{i'=i+M-L}^n vprc\CCC_{i'-i} \LLL_{n-i'}$, so taking $C,\eps>0$ such that $\Lambda_k(\CCC,\ph) \leq Ce^{k(P(\ph)-\eps)}$ for all $k\in \NN$, we have
\begin{align*}
\Lambda_n(\BBB(v,i),\ph) &\leq \sum_{i'=i+M-L}^n e^{\hat\ph(vprc)} \Lambda_{i'-i}(\CCC,\ph)\Lambda_{n-i'}(\LLL,\ph) \\
&\leq e^{\hat\ph(v)} e^{\tau\|\ph\| + \hat\ph(rc)} \sum_{i'=i+M-L}^n C e^{(i'-i)(P(\ph)-\eps)} Q_2 e^{(n-i')P(\ph)} \\
&\leq e^{\hat\ph(v)} e^{\tau\|\ph\|+\hat\ph(rc)} e^{(n-i)P(\ph)} CQ_2 e^{-(M-L)\eps} (1-e^{-\eps})^{-1}.
\end{align*}
Since $e^{(n-i)P(\ph)} = e^{mP(\ph)}$, 
by \eqref{eqn:H+v2} we can take $M$ large enough that for every $n$ and $v$ we get
$
\Lambda_n(\BBB(v,i)) \leq \gamma \Lambda_n(\HHH^+(v),\ph)
$;
%for every $n$ and $v$, which 
together with \eqref{eqn:wHsM-2} this shows that
\[
\Lambda_n(\widetilde\HsM(v,i) \setminus \BBB(v,i),\ph) \geq \gamma \Lambda_n(\HHH^+(v),\ph).
\]
Finally, $\widetilde\HsM(v,i) \setminus \BBB(v,i) \subset \HsM(v)$, completing the proof of Lemma \ref{lem:chunk-of-H}.

\subsubsection{Proof of Lemma \ref{lem:hatC}}\label{sec:hatC}

Given $\beta,\delta>0$ and $w\in \hC^{\beta,\delta}_n$, let $i_a,i_a'$ be as in the definition of $\hC^{\beta,\delta}$ for $a=1,\dots, B \leq 2 \lfloor \beta n \rfloor$.   Write $I_a = (i_a, i_a'] \cap \NN$ and observe that if some $i\in [1,\abs{w}]$ is contained in $I_a$ for three distinct choices of $a$, then one of the corresponding intervals $I_a$ is contained in the union of the other two, and hence can be removed from the collection without changing $\bigcup_a I_a$.  Thus without loss of generality we may assume that every $i\in [1,\abs{w}]$ is contained in at most two of the $I_a$, and by re-indexing if necessary, we have
\[
i_1 < i_1' \leq i_3 < i_3' \leq \cdots, \qquad i_2 < i_2' \leq i_4 < i_4' \leq \cdots.
\]
Either $\sum_{a\text{ even}} (i_a' - i_a) \geq \delta n$ or $\sum_{a \text{ odd}} (i_a' - i_a) \geq \delta n$.  Write $j_b,j_b'$ for the indices in the larger sum, so $b=1,\dots, B'$ where $B'  \leq \lfloor \beta n \rfloor$ (since $B\leq 2\lfloor \beta n \rfloor$).  We see that $w\in \LLL_{j_1} \CCC_{j_1' - j_1} \LLL_{j_2 - j_1'} \CCC_{j_2' - j_2} \cdots \CCC_{j_b' - j_b} \LLL_{n - j_b'}$; in particular,
\[
\hC^{\beta,\delta}_n \subset \bigcup_{B'=1}^{\lfloor \beta n\rfloor} \bigcup_{\vec{j},\vec{j}'} \LLL_{j_1} \CCC_{j_1' - j_1} \LLL_{j_2 - j_1'} \CCC_{j_2' - j_2} \cdots \CCC_{j_b' - j_b} \LLL_{n - j_b'},
\]
where the inner union is over all sequences $0\leq j_1 < j_1' \leq j_2 \cdots \leq j_{B'}' \leq n$ such that $\sum_b (j_b' - j_b) \geq \delta n$.  Assume that $\beta < \frac 12$, so $2B' \leq 2\lfloor \beta n \rfloor \leq n$; then the number of such sequences is at most\footnote{The factor of 2 in the top half comes since we allow $j_b' = j_{b+1}$; to associate each such sequence to a strictly increasing sequence we can duplicate each of the numbers $0,1,\dots, n$.}
$
\binom{2(n+1)}{2B'} \leq \binom{2n+2}{2\lfloor \beta n \rfloor} \leq 4\binom{2n}{2\lfloor \beta n \rfloor} \leq 8n e^{h(\beta)2n},
$
where we use Lemma \ref{lem:entropy-bound}.
For each such sequence we can write $\ell(\vec{j},\vec{j}') = \sum_b (j_b' - j_b) \geq \delta n$ and use \eqref{eqn:CD}, Lemma \ref{lem:counting}, and \eqref{eqn:LnC} to get
\begin{multline*}
\Lambda_n( \LLL_{j_1} \CCC_{j_1' - j_1} \LLL_{j_2 - j_1'} \CCC_{j_2' - j_2} \cdots \CCC_{j_b' - j_b} \LLL_{n - j_b'},\ph) \\
\leq Q_2^{B'+1} e^{(n - \ell(\vec{j},\vec{j}'))P(\ph)} C^{B'} e^{\ell(\vec{j},\vec{j}')(P(\ph) - \eps)} 
\leq Q_2 (Q_2 C)^{\beta n} e^{nP(\ph)} e^{-\delta\eps n},
\end{multline*}
where $C,\eps$ depend on $\CCC$, but not on $\vec{j}$ or $\vec{j}'$.
Summing over all choices of $B'$ and of $\vec{j},\vec{j}'$ gives
\[
\Lambda_n(\hC^{\beta,\delta},\ph) 
\leq 8\beta n^2 e^{h(\beta) 2n } Q_2 e^{(\log(Q_2 C))\beta n} e^{nP(\ph)} e^{-\delta \eps n},
\]
and we conclude that
\[
P(\hC^{\beta,\delta},\ph) \leq 2h(\beta) + \beta\log(Q_2 C) + P(\ph) - \delta \eps.
\]
For small enough $\beta$ this is $< P(\ph)$, which proves Lemma \ref{lem:hatC}.

\subsubsection{Proof of Lemma \ref{lem:AkZk}}\label{sec:AkZk}

First we prove the inclusions in \eqref{eqn:AkZk}, then we prove the estimate in \eqref{eqn:Ynk}.  For the first inclusion in \eqref{eqn:AkZk}, if $v\in \ZZZ^k$ and $w\in \HHH^+(v)$, then $w_{[1,kT]} = v_{[1,kT]} \in \AAA^k$ by the definition of $\ZZZ^k$, so $w\in \AAA^k$.

The second inclusion requires more work.  Let $m_k$ be as in \eqref{eqn:mkZk},  %As in the lemma, write $m_k = (k+\frac 12)T - (\tau + \abs{rc})$.  
and let $w\in \AAA_n^k$.
If $w_{[1,m_k]}\notin\GGG \Cs$, then $w\in \YYY^k$, so we may assume that $w_{[1,m_k]}\in \GGG\Cs$.  Thus there is $i\in [1, m_k]$ such that $w_{[1,i]}\in \GGG$ and $w_{(i,m_k]}\in \Cs \subset \CCC$.  Let $k' = \lfloor i/T\rfloor$, so that $0\leq k'\leq k$, and let $j\leq k'$ be maximal such that $w_{[1,jT]}\in \AAA^j$.

First suppose that $k'=k$.  Thus $i\geq kT$, and so taking $i'\geq kT$ to be minimal such that $v := w_{[1,i']}\in \GGG$, we see that $v\in \ZZZ^k$ and $w\in \HHH^+(v)$.  So we move on to the case $k'<k$.  

\begin{lemma}\label{lem:wii}
Suppose $w\in \AAA^k$, $0\leq k' < k$, and $i\in [k'T,(k'+1)T)$ are such that $w_{[1,i]}\in \GGG$ and $w_{(i,m_k]} \in \CCC$, and $j\leq k'$ is maximal such that $w_{[1,jT]}\in \AAA^j$.
Then there is $\ell\in [k'T,\abs{w}]$ such that  $(jT,\ell]$ contains $B= k'-j$ intervals $\{(i_a,i_a']\}_{a=1}^B$ such that 
\begin{equation}\label{eqn:wii}
w_{(i_a,i_a']}\in \CCC \text{ for every $a$, and }
\#\textstyle\bigcup_{a=1}^B(i_a,i_a'] \geq \tfrac 12 (\ell - jT).
\end{equation}
\end{lemma}
\begin{proof}
If $j=k'$ then taking $\ell = k'T$ suffices.  So assume $j<k'$. 
Given $1\leq a\leq B=k'-j$, we have $w_{[1,(j+a)T]}\notin \AAA^{j+a}$ by maximality of $j$.  Thus there are $j_a < j+a$ and $i_a \in I_{j_a} = (j_a T, (j_a+\frac 12)T]$ such that $w_{(i_a,i']}\notin \CCC$ for all $i' \in [i_a+M-L,(j+a)T]$.  Note that $j\leq j_a$ because $w_{[1,jT]}\in \AAA^j$, and so $jT < i_a \leq (j_a + \frac 12)T \leq (j+a-\frac 12)T$.

On the other hand, $w\in \AAA^k$ and so for each $a$ there is $i_a'\in [i_a+M-L,\abs{w}]$ such that $w_{(i_a,i_a']}\in \CCC$.  By the previous paragraph we must have $i_a' > (j+a)T$.  Let $\ell = \max_a i_a'$.  Because $i_a \leq k'T$ for all $a=1,\dots, B$, we have
\[
\bigcup_{a=1}^B (i_a,i_a'] \supset \bigg(\bigcup_{a=1}^B \big((j+a-\tfrac 12)T, (j+a)T\big] \bigg) \cup (k'T, \ell],
\]
which proves \eqref{eqn:wii}.
\iftoggle{arxiv}{}{\qed}\end{proof}

Let $\ell$ be given by Lemma \ref{lem:wii}.
If $\ell \geq kT$ then we have $w\in \YYY^k$, since $w_{(jT,\ell]} \in \hC^{\beta,\delta}$ for $\delta = \frac 18$ and $\beta \geq \frac 1{2T}$, by the estimate $B = k'-j \leq \frac 1 T(\ell - jT)$.  

Now consider the case $\ell < kT$; we claim that in this case we have $w_{(jT,m_k]}\in \hC^{\beta,\delta}$.  Put $i_0 = i$ and $i_0' = m_k$; consider the collection of intervals $\{(i_a,i_a']\}_{a=0}^B \subset (jT, m_k]$.  Note that %$k' < k$ since $k'T \leq \ell$, and so
\[
B+1 = k' - j + 1 \leq k - j \leq \tfrac 1T (m_k - jT).
\]
It remains to show that $\#\bigcup_{a=0}^B (i_a,i_a'] \geq \tfrac 14 (m_k - jT)$.  From \eqref{eqn:wii} we observe that $\#\bigcup_{a=1}^B(i_a,i_a'] \geq \frac 12 (\ell - jT)$.  Moveover, this union is contained in $(jT,\ell]$, so writing $\ell' = \max(i,\ell)$, we use the general inequality $\frac{w+x}{y+z} \geq \min(\frac wy, \frac xz)$ to get
\begin{equation}\label{eqn:mkjT}
\frac{\#\bigcup_{a=0}^B(i_a,i_a']}{m_k-jT} 
= \frac{(\#\bigcup_{a=1}^B(i_a,i_a']) + \#(\ell',m_k]}{(\ell - jT) + (m_k-\ell)}
\geq \min\bigg( \frac 12, \frac{m_k-\ell'}{m_k-\ell}\bigg).
\end{equation}
Since $i < (k'+1)T$, $\ell\geq k'T$, we get $\ell' = \max(i,\ell) \leq \ell + T$, so $\ell' - \ell \leq T$.  Recall from \eqref{eqn:mkZk} that $m_k = kT + \frac T2 -  (\tau + \abs{rc})$; since $i < (k'+1)T \leq kT$ (using $k'<k$) and $\ell < kT$, we get $\ell' < kT$ and hence 
$m_k-\ell' \geq \frac T2 - (\tau + \abs{rc}) = M \geq \frac T3$, where we use the bound following \eqref{eqn:zeta-gamma} to get $T\leq 3M$.  Thus \eqref{eqn:mkjT} gives
\[
\frac{m_k-\ell'}{m_k-\ell} = \frac{m_k-\ell'}{(m_k-\ell') + (\ell'-\ell)} = \frac 1{1+\frac{\ell'-\ell}{m_k-\ell'}} \geq \frac 1{1+\frac T{T/3}} = \frac 14.
\]
This proves that $w_{(jT,m_k]}\in \hC^{\beta,\delta}$ and completes the proof of 
\eqref{eqn:AkZk}.

To complete the proof of Lemma \ref{lem:AkZk} we must prove the estimate in \eqref{eqn:Ynk}.  We do this using
\eqref{eqn:CD}, Lemma \ref{lem:counting}, and 
\eqref{eqn:QQ} to get
\[
\Lambda_n((\GGG^+ \setminus (\GGG\Cs))_{m_k} \LLL_{n - m_k},\ph)
\leq Q_7 \zeta^k e^{m_k P(\ph)} Q_2 e^{(n-m_k)P(\ph)}
=   Q_7 Q_2 \zeta^k e^{nP(\ph)}
\]
for the first part of $\YYY^k$, and for the second part, \eqref{eqn:CD}, Lemma \ref{lem:counting}, and \eqref{eqn:hatC} to get
\begin{align*}
\Lambda_n\bigg( \bigcup_{j=0}^{k-1} \bigcup_{\ell=kT}^n 
&\AAA_{jT}^j \hC_{\ell-jT}^{\beta,\delta}\LLL_{n-\ell},\ph\bigg)
\leq \sum_{j=0}^{k-1} \sum_{\ell = kT}^n \Lambda_{jT}(\AAA^j,\ph)
\Lambda_{\ell-jT}(\hC^{\beta,\delta},\ph) \Lambda_{n-\ell}(\LLL,\ph) \\
&\leq \sum_{j=0}^{k-1} \sum_{\ell = kT}^n
\Lambda_{jT}(\AAA^j,\ph) Q_8 \theta^{\ell - jT} e^{(\ell - jT)P(\ph)} Q_2 e^{(n-\ell)P(\ph)} \\
&\leq Q_8 Q_2 \sum_{j=0}^{k-1} \Lambda_{jT}(\AAA^j,\ph) \frac{\theta^{(kT-jT)}}{1-\theta} e^{(n-jT)P(\ph)}.
\end{align*}
Adding the two estimates (and loosening the first) gives \eqref{eqn:Ynk}.\footnote{Strictly speaking, in \eqref{eqn:Ynk} we could put $Q_7$ outside the fraction, and only multiply it by the term in the sum corresponding to $j=0$, but the looser estimate leads to less cumbersome bookkeeping and is sufficient for our purposes.}  This proves Lemma \ref{lem:AkZk}.

\section{Proofs of other results}\label{sec:other-proofs}

%In this section we give the remaining proofs.\foot{Number the subsections and make previous references precise}  First we prove Theorem \ref{thm:CGC}, then the characterisation of strong positive recurrence in \S\ref{sec:spr}, then Propositions \ref{prop:factors} and \ref{prop:positive-entropy} on factors, and we finish by proving the claims in Remark \ref{rmk:hE}.

\subsection{Proof of Theorem \ref{thm:CGC}}\label{sec:CGC}

%Before giving the proof of Theorem \ref{thm:CGC}, we prove a lemma  that we will use several times below.  
We start with a lemma.
Given  $\AAA \subset \LLL$, let 
\[
\hat{P}(\AAA,\ph) := \sup_{n\in \NN} \frac 1n \log \Lambda_n(\AAA,\ph)
\]
and consider the collection $\AAA^* = \{w^1\cdots w^k\in \LLL \mid w^i \in \AAA \text{ for all } i\}$.

\begin{lemma}\label{lem:D*-pressure}
Let $\ph$ be H\"older and $\AAA\subset \LLL_{\geq M}$ for some $M\in \NN$.  Then
\begin{equation}\label{eqn:D*-pressure}
P(\AAA^*,\ph)
\leq \hat{P}(\AAA,\ph) + h(\tfrac 1M),
\end{equation}
where $h(\delta) = -\delta\log\delta - (1-\delta)\log (1-\delta)$. 
\end{lemma}
\begin{proof}
Given $n,k\in \NN$, let
\[
\mathbf{N}_k = \{(n_1,n_2,\dots, n_k)\in \NN^k \mid \sum n_i = n \text{ and } n_i\geq M \text{ for all } 1\leq i\leq k\}.
\]
Note that $\mathbf{N}_k$ is empty for all $k > \frac nM$.  Now we have
\begin{align*}
\Lambda_n(\AAA^*,\ph) &= \sum_{k\in \NN} \sum_{\mathbf{N}_k} \sum_{\substack{w^1\in \AAA_{n_1}, \cdots w^k\in \AAA_{n_k} \\ w^1\cdots w^k\in \LLL}} e^{\hat\ph(w^1\cdots w^k)}
\leq \sum_k\sum_{\mathbf{N}_k} \sum_{w^1\in \AAA_{n_1}, \cdots w^k\in \AAA_{n_k}} e^{\sum_{i=1}^k \hat\ph(w^i)} \\
&\leq \sum_k\sum_{\mathbf{N}_k} \prod_{i=1}^k  \Lambda_{n_i}(\AAA,\ph) 
\leq \sum_{k=1}^{\lfloor n/M \rfloor} (\#\mathbf{N}_k) e^{n\hat{P}(\AAA,\ph) }
\end{align*}
Recalling Lemma \ref{lem:entropy-bound}, we see that for every $k\leq \frac nM$ we have
$
\#\mathbf{N}_k \leq \textstyle\binom{n}{k} \leq n e^{h(\frac 1M)n}
$,
which gives the estimate
\[
\Lambda_n(\AAA^*,\ph) \leq  n^2 M^{-1} e^{h(\frac 1M)n} e^{n\hat{P}(\AAA,\ph)}. % + \frac nM \absd{\ph}}.
\]
We conclude that $P(\AAA^*,\ph) \leq \hat{P}(\AAA,\ph) + h(\frac 1M)$, as claimed.
\iftoggle{arxiv}{}{\qed}\end{proof}

Now we prove Theorem \ref{thm:CGC}.
As in the statement of the theorem, let $\CCC^\pm$ satisfy \eqref{eqn:persistent2} and \ref{spec*}, and fix $\eps>0$ small enough that $P(\CCC^+ \cup \CCC^-,\ph) + 3\eps < P(\ph)$.  We start by choosing parameters $M,\tau,N$ to satisfy certain pressure estimates; then we use these to define $\Cp,\Cs$.  Finally, we define $\GGG$ and verify \ref{spec}--\ref{stay-good}.

\medskip
\noindent
\textsc{Step 1} (\emph{choosing $M,\tau,N$}):
Choose $M\in \NN$ such that $h(\frac 1M) < \eps$ and 
\begin{equation}\label{eqn:hatP}
\hat P(\CCC^-_{\geq M},\ph) < P(\CCC^-,\ph) + \eps,
\qquad
\hat P(\CCC^+_{\geq M},\ph) < P(\CCC^+,\ph) + \eps.
\end{equation}
Let $\tau = \tau(M)$ be as in \ref{spec*} and consider the collections
\begin{align*}
\DDD^- &:= \{ w\in \LLL \mid \text{there exists } x\in \LLL_{\leq \tau + M} \text{ such that } wx\in \CCC^- \}, \\
\DDD^+ &:= \{ w\in \LLL \mid \text{there exists } x\in \LLL_{\leq \tau + M} \text{ such that } xw\in \CCC^+ \}.
\end{align*}
These have the same pressures as $\CCC^\pm$; indeed,
\begin{multline*}
\Lambda_n(\DDD^-,\ph) = \sum_{i=0}^{\tau + M} \sum_{\substack{w\in \DDD^-_n \\ \abs{x(w)}=i}} e^{\hat\ph(w)} 
\leq \sum_{i=0}^{\tau+M} \sum_{v\in \CCC^-_{n+i}} e^{\hat\ph(v) + \absd{\ph} + i\|\ph\|} 
\\
\leq (\tau+M+1) e^{\absd{\ph} + (\tau+M)\|\ph\|} e^{(n+\tau+M)\hat{P}(\CCC^-_{\geq n},\ph)}.
\end{multline*}
Sending $n\to\infty$ gives $P(\DDD^-,\ph) \leq P(\CCC^-,\ph)$, and the estimate for $\DDD^+$ is similar.  Thus we can choose $N\geq M$ large enough that $\frac{\log 2}N < \eps$ and
\begin{equation}\label{eqn:hatP2}
\hat P(\DDD^-_{\geq N},\ph) < P(\CCC^-,\ph) + \eps,\qquad
\hat P(\DDD^+_{\geq N},\ph) < P(\CCC^+,\ph) + \eps.
\end{equation}
For use in Step 3 below we record the fact that \eqref{eqn:persistent2} holds for $\DDD^\pm$: that is,
\begin{equation}\label{eqn:D-persists}
(vw\in \DDD^+ \Rightarrow v\in \DDD^+) \text{\quad and \quad}
(vw\in \DDD^- \Rightarrow w\in \DDD^-).
\end{equation}
This follows quickly from \eqref{eqn:persistent2} for $\CCC^\pm$: if $vw\in \DDD^+$ then there is $x\in \LLL_{\leq \tau+M}$ such that $xvw\in \CCC^+$, and since \eqref{eqn:persistent2} gives $xv\in \CCC^+$, we get $v\in \DDD^+$.  The case $vw\in \DDD^-$ is similar.

\medskip
\noindent
\textsc{Step 2} (\emph{definition of $\Cp,\Cs$}):
Consider the collections
\[
\Cp := (\CCC^-_{\geq M} \cup \DDD_{\geq N}^+)^*, 
\qquad
\Cs := (\CCC^+_{\geq M} \cup \DDD_{\geq N}^-)^*.
\]
By Lemma \ref{lem:D*-pressure} we have
\begin{equation}\label{eqn:PCP}
P(\Cp,\ph) \leq \hat P(\CCC^-_{\geq M} \cup \DDD_{\geq N}^+,\ph) + h(\tfrac 1M) .
\end{equation}
To estimate the $\hat P$ term we note that for $n\in [M,N)$ we have
\[
\Lambda_n(\CCC_{\geq M}^- \cup \DDD_{\geq N}^+,\ph)
= \Lambda_n(\CCC_{\geq M}^-,\ph) \leq e^{n\hat P(\CCC_{\geq M}^-,\ph)}
\leq e^{n(P(\CCC^-,\ph)+\eps)}
\]
using \eqref{eqn:hatP}, while for $n\geq N$ we have
\begin{multline*}
\Lambda_n(\CCC_{\geq M}^- \cup \DDD_{\geq N}^+,\ph)
\leq \Lambda_n(\CCC_{\geq M}^-,\ph) + \Lambda_n(\DDD_{\geq N}^+,\ph) \\
\leq 2e^{n\max \{ \hat{P}(\CCC_{\geq M}^-,\ph), \hat{P}(\DDD_{\geq N}^+,\ph)\}}
\leq 2e^{n(\max\{P(\CCC^-,\ph),P(\CCC^+,\ph)\} + \eps)}
\end{multline*}
using \eqref{eqn:hatP} and \eqref{eqn:hatP2}.  We conclude that
\[
\hat P(\CCC_{\geq M}^- \cup \DDD_{\geq N}^+,\ph)
\leq P(\CCC^- \cup \CCC^+,\ph) + \eps + \tfrac{\log 2}N
< P(\CCC^- \cup \CCC^+,\ph) + 2\eps,
\]
Together with \eqref{eqn:PCP} and the estimate on $h(\frac 1M)$ this gives
\begin{equation}\label{eqn:PCP2}
P(\Cp,\ph) < P(\CCC^- \cup \CCC^+,\ph) + 3\eps < P(\ph).
\end{equation}
The estimate for $P(\Cs,\ph)$ is similar.  %Sending $\eps\to 0$, the right-hand side can be made $<P(\ph)$.

\medskip
\noindent
\textsc{Step 3} (\emph{definition of $\GGG$}):
Now we describe $\GGG$ such that \ref{spec} and \ref{stay-good} hold and we have $P(\LLL\setminus \Cp \GGG \Cs,\ph) < P(\ph)$.  Let
\begin{multline}\label{eqn:GGG}
\GGG := \{w\in \LLL \setminus (\DDD^+ \cup \DDD^-) \mid w_{[1,i]} \notin \CCC^-,  w_{(\abs{w}-i,\abs{w}]} \notin \CCC^+ 
\text{ for all } M \leq i \leq \abs{w}, 
\\ \text{ and }
w_{[1,i]} \notin \DDD^+, w_{(\abs{w}-i,\abs{w}]} \notin \DDD^- \text{ for all } i\geq N \}.
\end{multline}
Given $w\in \LLL$, decompose $w$ as $w = u^p v u^s$ by beginning with $v=w$ and $u^p = u^s = \emptyset$, and then proceeding as follows.
\begin{enumerate}[leftmargin=*, widest=3]
\item Choose the smallest $i\in [1,\abs{v}]$ such that $v_{[1,i]} \in \CCC^-_{\geq M} \cup \DDD^+_{\geq N}$ (if such an $i$ exists); then replace $u^p$ with $u^p v_{[1,i]}$ and replace $v$ with $v_{(i,\abs{v}]}$.  Iterate this step until no such $i$ exists; note that $u^p\in \Cp$.
\item Take the resulting word $v$ and choose the smallest $i\in [1,\abs{v}]$ such that $v_{(\abs{v}-i,\abs{v}]}\in \CCC^+_{\geq M} \cup \DDD^-_{\geq N}$ (if such an $i$ exists); then replace $u^s$ with $v_{(\abs{v}-i,\abs{v}]} u^s$ and $v$ with $v_{[1,\abs{v}-i]}$.  Iterate this step until no such $i$ exists; note that $u^s\in \Cs$.
\item Observe that the resulting word $v$ satisfies $v\in \GGG \cup \DDD^+ \cup \DDD^-$ by the definition of $\GGG$.
\end{enumerate}
We conclude that $\LLL\setminus \Cp \GGG \Cs \subset \Cp (\DDD^+ \cup \DDD^-) \Cs$, and in particular,\footnote{In fact, the estimates given here and earlier show that we can make $P(\Cp \cup \Cs \cup (\LLL\setminus \Cp\GGG\Cs),\ph)$ as close to $P(\CCC^+ \cup \CCC^-,\ph)$ as we like by taking $M,N$ large.}
\[
P(\LLL\setminus \Cp \GGG \Cs,\ph) \leq \max \{ P(\Cp,\ph), P(\DDD^+,\ph), P(\DDD^-,\ph), P(\Cs,\ph) \} < P(\ph).
\]
It remains to show that $\GGG$ satisfies \ref{spec} and \ref{stay-good}.  For \ref{spec}, first note that $\GGG \subset \GGG(\CCC^\pm,M)$, and so by \ref{spec*}, for every $v,w\in \GGG$ there is $u\in \LLL_{\leq \tau}$ such that $vuw\in \LLL$.  Fix any such $u$ and observe that for all  $v' = v_{[i,\abs{v}]}$ and $w' = w_{[1,j]}$ we have $v'uw'\in \LLL$ since it is a subword of $vuw$.  When $v',w'\in \GGG$, the following lemma implies that $v'uw'\in \GGG$, which establishes \ref{spec}.

\begin{lemma}\label{lem:L-to-G}
If $v',w'\in \GGG$ and $u\in \LLL_{\leq \tau}$ are such that $v'uw'\in \LLL$, then $v'uw'\in \GGG$.
\end{lemma}
\begin{proof}
%We claim that \foot{********}\changed{if $v'\in \GGG$ is a final segment of $v$ and $w'\in \GGG$ is an initial segment of $w$, then $v'uw'\in \GGG$, which will establish \ref{spec}.  To show this, first observe that $v'uw'\in \LLL$ since it is subword of $vuw\in \LLL$.  Now 
We prove that
%in fact $vuw\in \GGG$, which will establish \ref{spec}.  To show this, we prove that
\begin{enumerate}[leftmargin=*, widest=3,
label={\textup{\textbf{$\langle$\arabic{*}$\rangle$}}}]
\item\label{1} $v'uw'\notin \DDD^+$ and $v'uw'\notin \DDD^-$;
\item\label{2} $(v'uw')_{[1,i]}\notin \CCC^-$ and $(v'uw')_{(\abs{v'uw'}-i,\abs{v'uw'}]}\notin \CCC^+$ for all $i\geq M$;
\item\label{3} $(v'uw')_{[1,i]}\notin \DDD^+$ and $(v'uw')_{(\abs{v'uw'}-i,\abs{v'uw'}]}\notin \DDD^-$ for all $i\geq N$.
\end{enumerate}
In each case we prove only the first assertion; the second follows by a symmetrical argument.

For \ref{1}, we see that $v'\notin\DDD^+$ implies $v'uw'\notin\DDD^+$ by \eqref{eqn:D-persists}.
%given any $x\in \LLL_{\leq \tau+M}$ we have $xv\notin \CCC^+$ since $v\notin \DDD^+$, and hence $xvuw\notin \CCC^+$ by \eqref{eqn:persistent2}; this proves $vuw\notin \DDD^+$.  
For \ref{2} we consider $(v'uw')_{[1,i]}$ in the following three cases.
\begin{itemize}
\item $M\leq i\leq \abs{v'}$.  Then $v'\in \GGG$ gives $(v'uw')_{[1,i]} = v'_{[1,i]} \notin\CCC^-$.
\item $\abs{v'} < i \leq \abs{v'u}+M$.  Then since $v'\notin \DDD^-$ and $i-\abs{v'} \leq \tau+ M$, we must have $(v'uw')_{[1,i]} =v'(uw')_{[1,i-\abs{v'}]}\notin\CCC^-$.
\item $i>\abs{v'u} + M$.  Then $w'\in \GGG$ gives $w'_{[1,i-\abs{v'u}]}\notin\CCC^-$, so \eqref{eqn:persistent2} gives $(v'uw')_{[1,i]}\notin\CCC^-$.
\end{itemize}
For \ref{3}, if $N\leq i < \abs{v'}$, then $(v'uw')_{[1,i]} = v'_{[1,i]}\notin \DDD^+$ by the definition of $\GGG$; if $i\geq \abs{v'}$, then $v'\notin \DDD^+$ implies $(v'uw')_{[1,i]}\notin\DDD^+$ by \eqref{eqn:D-persists}.  
%given any $x\in \LLL_{\leq \tau+M}$ and $i\geq N$, there are two cases: either $\abs{v}\leq i$, in which case $xv \notin \CCC^+$ (since $v\notin\DDD^+$) gives $x(vuw)_{[1,i]} \notin \CCC^+$ by \eqref{eqn:persistent2}; or $\abs{v}>i$, in which case $x(vuw)_{[1,i]} = x(v_{[1,i]})\notin \CCC^+$ since $v\in \GGG$, and thus $(vuw)_{[1,i]}\notin \DDD^+$.
Thus $v'uw'\in \GGG$, which % verifies \ref{spec}.
proves Lemma \ref{lem:L-to-G}.
\iftoggle{arxiv}{}{\qed}\end{proof}

The proof of \ref{stay-good} has a similar flavour.  If $uvw\in \LLL$ and $uv,vw\in \GGG$, we show that $uvw\in \GGG$, and that $v\in \GGG$ if $\abs{v}\geq N$.  As above, we verify the conditions involving $\CCC^-$ and $\DDD^+$; the other conditions %involving $\CCC^+$ and $\DDD^-$ 
follow from symmetric arguments.  We start with $uvw$.
\begin{enumerate}[leftmargin=*, widest=3,
label={\textup{\textbf{$\langle$\arabic{*}$\rangle$}}}]
\item Since $uv\notin \DDD^+$, we have $uvw\notin \DDD^+$ by \eqref{eqn:D-persists}.
%For $x\in \LLL_{\leq \tau+M}$, $uv\notin \DDD^+$ gives $xuv\notin \CCC^+$, so $xuvw\notin \CCC^+$ by \eqref{eqn:persistent2}, hence $uvw\notin \DDD^+$.
\item For $i\geq M$, we check $(uvw)_{[1,i]}\notin \CCC^-$ in the following three cases.
\begin{itemize}
\item $M \leq i\leq \abs{uv}$.  Then $uv\in \GGG$ gives $(uvw)_{[1,i]} = (uv)_{[1,i]} \notin \CCC^-$.
\item $\abs{uv} < i \leq \abs{uv}+M$.  Then $uv\notin \DDD^-$ implies that $(uvw)_{[1,i]} = (uv)w_{[1,i-\abs{uv}]}\notin \CCC^-$ since $i-\abs{uv}\leq M$.
\item $i>\abs{uv} + M \geq \abs{u}+M$.  Then $vw\in \GGG$ gives $(vw)_{[1,i-\abs{u}]} \notin \CCC^-$, hence $(uvw)_{[1,i]}\notin \CCC^-$ by \eqref{eqn:persistent2}.
\end{itemize}
\item If $N\leq i\leq \abs{uv}$, then $uv\in \GGG$ gives $(uvw)_{[1,i]} = (uv)_{[1,i]}\notin \DDD^+$.  If $i\geq \abs{uv}$, then $uv\notin\DDD^+$ gives $(uvw)_{[1,i]}\notin \DDD^+$ by \eqref{eqn:D-persists}.
%Given any $x\in \LLL_{\leq \tau+M}$ and $i\geq N$, there are two cases.  If  $\abs{uv}\leq i$, then $uv\notin \DDD^+$ gives $xuv\notin \CCC^+$ and hence $x(uvw)_{[1,i]}\notin\CCC^+$ by \eqref{eqn:persistent2}.  If  $\abs{uv}\geq i$, then $uv\in \GGG$ gives $x(uvw)_{[1,i]} = x(uv)_{[1,i]} \notin \CCC^+$.  In both cases we deduce that $(uvw)_{[1,i]}\notin \DDD^+$.
\end{enumerate}
We conclude by showing that $v\in \GGG$ whenever $\abs{v} \geq N$.
\begin{enumerate}[leftmargin=*, widest=3,
label={\textup{\textbf{$\langle$\arabic{*}$\rangle$}}}]
\item %Given  $x\in \LLL_{\leq \tau+M}$, we have $xv = x(vw)_{[1,\abs{v}]}\notin\CCC^+$ since $vw\in \GGG$ and $\abs{v}\geq N$; thus $v\notin \DDD^+$.
Since $vw\in \GGG$ and $\abs{v}\geq N$, we have $v = (vw)_{[1,\abs{v}} \notin \DDD^+$.
\item Given $i\geq M$ we have $v_{[1,i]} = (vw)_{[1,i]} \notin \CCC^-$ since $vw\in \GGG$.
\item Given $i\geq N$ we have $v_{[1,i]} = (vw)_{[1,i]} \notin \DDD^+$ since $vw\in \GGG$.
\end{enumerate}
This establishes \ref{stay-good} for $\GGG$ and completes the proof of Theorem \ref{thm:CGC}.

\subsection{Shifts of quasi-finite type}\label{sec:qft}

To prove Theorem \ref{thm:QFT}, we first prove \eqref{eqn:persistent2} for $\CCC^+ = \CCC^\ell$ and $\CCC^- = \CCC^r$.  Given $vw\in \CCC^+ = \CCC^\ell$, let $u\in \LLL$ be such that $(vw)_{[2,\abs{vw}]} u \in \LLL$ but $(vw)u\notin \LLL$.  Then $v_{[2,\abs{v}]} (wu)\in \LLL$ but $v(wu)\notin \LLL$, so $v\in \CCC^+ = \CCC^\ell$.  The proof for $\CCC^- = \CCC^r$ is similar.

Now we show that $\CCC^-=\CCC^r$ and $\CCC^+ = \CCC^\ell$ always form a complete list of obstructions to specification as long as $X$ is topologically transitive.  Fix  $M\in \NN$ and let $\tau\in \NN$ be such that for every $v,w\in \LLL_{\leq M}$ there is $u\in \LLL$ with $\abs{u}\leq \tau$ such that $vuw\in \LLL$; note that such a $\tau$ exists because $X$ is transitive and $\LLL_{\leq M}$ is finite.  Then given any $v,w\in \GGG(\CCC^\pm,M)$, there is $u\in \LLL_{\leq \tau}$ such that $v_{[\abs{v}-M+1,\abs{v}]} u w_{[1,M]} \in \LLL$.  Since $v,w\in \GGG(\CCC^\pm,M)$, we have that $v_{[\abs{v}-M,\abs{v}]}\notin \CCC^+ = \CCC^\ell$, and hence $v_{[\abs{v} - M, \abs{v}]} u w_{[1,M]} \in \LLL$.  Proceeding inductively and using the fact that $v_{[\abs{v}-i,\abs{v}]}\notin \CCC^\ell$ for any $i>M$, we conclude that $vuw_{[1,M]}\in \LLL$.  A similar induction using $w_{[1,i]}\notin \CCC^- = \CCC^r$ yields $vuw\in \LLL$.  Since $M$ was arbitrary, this gives \ref{spec*}.

The proof for $\CCC^- = \emptyset$ and $\CCC^+ = \CCC^\ell$ in the topologically exact case is similar.  Given $M\in \NN$, exactness gives $\tau\in \NN$ such that for every $v\in \LLL_{\leq M}$, we have $\sigma^{\tau + M}[v] = X^+$.  In particular, for all $v\in \LLL_{\leq M}$ and $w\in \LLL$, we have $\sigma^{\tau + M}[v] \supset [w]$, so there is $u\in \LLL_\tau$ such that $vuw\in \LLL$.  Then given any $v,w\in \GGG(\CCC^\pm,M)$, there is $u\in \LLL_{ \tau}$ such that $v_{(\abs{v}-M,\abs{v}]} uw\in \LLL$, and the same inductive argument as before shows that $vuw\in \LLL$, so \ref{spec*} holds.

\subsection{Synchronised shifts}\label{sec:sync}

To deduce Theorem \ref{thm:sync} from Theorems \ref{thm:structure} and \ref{thm:spr-a}, let $s$ be a synchronising word and let $\GGG = \LLL \cap s\LLL \cap \LLL s$ be the set of words that start and end with $s$ (though $s$ may overlap itself).   Choose $c\in \LLL$ such that $scs\in \LLL$, and let $\tau = \abs{c}$.  For every $v,w\in \GGG$ we have $vcw\in \GGG$ by the definition of a synchronising word, so $\GGG$ satisfies \ref{s-spec}; since the gluing word $c$ does not depend on the choice of $v,w\in \GGG$, the statement in \ref{s-spec} involving $v'$ and $w'$ holds automatically.  Writing $\Cp = \Cs = \LLL(Y) = \LLL \setminus \LLL s \LLL$ for the collection of words that do not contain $s$ as a subword, every $w\in \LLL$ is either contained in $\Cp$, or has $w=u^pvu^s \in \Cp\GGG\Cs$ by marking the first and last occurrences of $s$ as a subword of $w$.  Thus $P(\Cp \cup \Cs \cup (\LLL \setminus \Cp \GGG \Cs),\ph) = P(Y,\ph)$, and the hypothesis that $P(Y,\ph)<P(\ph)$ implies \ref{gap}.  Finally, taking $L=|s|$ we see that \ref{stay-good} is immediately satisfied by the definition of $\GGG$.

\subsection{Proof of Theorem \ref{thm:factors}}\label{sec:passing-to-factors}

Theorem \ref{thm:factors} is a consequence of Theorem \ref{thm:CGC} and the following two propositions.

\begin{proposition}\label{prop:factors}
If $\tilde X$ is a shift factor of $X$, then $\hspec(\tilde X) \leq \hspec(X)$.
\end{proposition}

\begin{proposition}\label{prop:positive-entropy}
If there is $\GGG\subset \LLL$ satisfying \ref{spec} such that every $w\in \LLL$ has $\LLL w \LLL \cap \GGG \neq\emptyset$, then every subshift factor $\tilde X$ of $X$ has $h(\tilde X)>0$ or is a single periodic orbit.  In particular, if $\gcd \{ k \mid  \Per_k(X) \neq \emptyset %X \text{ contains a periodic point with period } k
\} = 1$, then every non-trivial subshift factor of $X$ has positive entropy.
\end{proposition}

\begin{proof}[Proof of Proposition \ref{prop:factors}]
Let $(\tilde X,\tilde\sigma)$ be a shift factor of $(X,\sigma)$.  We prove that $\hspec(\tilde X) \leq \hspec(X)$ by showing that if $\CCC^{\pm} \subset \LLL$ satisfy \eqref{eqn:persistent2} and \ref{spec*}, then there are $\tilde\CCC^{\pm}\subset\tilde\LLL$ satisfying the same conditions and with the property that $h(\tC^- \cup \tC^+) \leq h(\CCC^- \cup \CCC^+)$.

We follow the proof of \cite[Proposition 2.2]{CT1}: given two shifts $X,\tilde X$ on finite alphabets $A,\tilde A$ with a factor map $\pi\colon X\to \tilde X$, there is some $m\in \NN$ and $\theta\colon \LLL_{2m+1}(X) \to \tilde A$ such that $\pi(x)_n = \theta(x_{[n-m,n+m]})$ for every $x\in X$ and $n\in \ZZ$.  Writing $\Theta\colon \LLL_{n+2m} \to \tilde \LLL_n$ for the map induced by $\theta$, we consider $\CCC^\pm\subset \LLL$ satisfying \eqref{eqn:persistent2} and \ref{spec*}, and put
$
\tC^- = \Theta(\CCC^-)$, % \qquad 
$\tC^+ = \Theta(\CCC^+).
$
Since $\#\tC^-_n \leq \#\CCC^-_{n+2m}$, and similarly for $\tC^+$, we get $h(\tC^-\cup\tC^+) \leq h(\CCC^- \cup \CCC^+)$.

To prove \eqref{eqn:persistent2} for $\tC^+$, observe that given $\tilde{w}\in \tC^+$ and $1\leq i\leq \abs{w}$, there is $w\in \CCC^+$ such that $\tilde{w} = \Theta(w)$, and in particular $\tilde{w}_{[1,i]} = \Theta(w_{[1,i+2m]}) \in \Theta(\CCC^+) = \tC^+$ since $\CCC^+$ satisfies \eqref{eqn:persistent2}.  The proof for $\tC^-$ is similar.

Finally, every $\GGG(\tC^\pm,M)$ has specification in the sense of \ref{spec*}: given $\tilde{v}\in \GGG(\tC^\pm,M)$, let $v \in \LLL$ be such that $\tilde{v} = \Theta(v)$.  Then $\tilde{v}_{[1,i]}\notin \tC^-$ for all $i> M$ implies that $v_{[1,j]} \notin \CCC^-$ for all $j > M+2m$, and we similarly deduce that $v_{[i,\abs{v}]}\notin\CCC^+$ for $i\leq \abs{v}-M-2m$, hence $v\in \GGG(\CCC^\pm,M+2m)$.

Now given $\tilde{v},\tilde{w}\in \GGG(\tC^\pm,M)$, there are $v,w\in \GGG(\CCC^\pm,M+2m)$ with $\Theta(v)=\tilde{v}$ and $\Theta(w)=\tilde{w}$.  By \ref{spec*} for $\CCC^\pm$ there is $u\in \LLL$ such that $\abs{u}\leq \tau(M+2m)$ and $vuw\in \LLL$.  Thus $\Theta(vuw) = \tilde v \cdot \Theta(v_{(\abs{v}-2m,\abs{v}]} u w_{[1,2m]}) \cdot \tilde w \in \tilde\LLL$, so $\GGG(\tC^\pm,M)$ satisfies \ref{spec*} with $\tilde\tau(M) = \tau(M+2m) + 2m$.
\iftoggle{arxiv}{}{\qed}\end{proof}

\begin{proof}[Proof of Proposition \ref{prop:positive-entropy}]
If $\GGG$ is periodic in the sense of 
Proposition \ref{prop:periodic}, then $\LLL$ is periodic as well,
since $\LLL w\LLL \cap \GGG \neq \emptyset$ for every $w\in \LLL$,  so $X$ is a single periodic orbit.  Thus if $X$ is not a single periodic orbit, Lemma \ref{lem:good-entropy} applies to give $h(\GGG)>0$, hence $h(X)>0$.

If $\tilde X$ is a shift factor of $X$ with factor map $\Theta \colon \LLL_{n+2m} \to \tilde \LLL_n$, then taking $\tilde\GGG = \Theta(\GGG)$ we see that $\tilde\GGG$ has \ref{spec}; indeed, given any $\tilde{v},\tilde{w}\in \tilde\GGG$ we take $v,w\in \GGG$ such that $\Theta(v)=\tilde{v}$ and $\Theta(w)=\tilde{w}$, then there is $u\in \LLL_{\leq \tau}$ such that 
%$vuw \in \GGG$, and we see that $\Theta(vuw) = \tilde v\tilde u \tilde w\in \tilde \GGG$, where $\abs{\tilde u} \leq \tau + 2m$, hence $\tilde \GGG$ has \ref{spec}. 
if $v',w'\in \GGG$ have $v\in \LLL v'$ and $w\in w'\LLL$, then $v'uw'\in \GGG$.  In particular, $vuw\in \GGG$, and $\Theta(vuw) = \tilde v\tilde u \tilde w\in \tilde \GGG$, where $\abs{\tilde u} \leq \tau + 2m$.  Now if $\tilde{v}', \tilde{w}'\in \tilde\GGG$ have
$\tilde v \in \tilde\LLL \tilde{v}'$ and $\tilde w \in \tilde{w}'\tilde \LLL$, then there are $v',w'$ as above with $\Theta(v')=\tilde{v}'$ and $\Theta(w')=\tilde{w}'$, and thus $\tilde{v}' \tilde{u} \tilde{w}' = \Theta(v'uw') \in \tilde\GGG$, so $\tilde\GGG$ has \ref{spec}.

Moreover, given any $\tilde w\in \tilde \LLL$ we have $\tilde w = \Theta(w)$ for some $w\in \LLL$, and thus there are $u,v\in \LLL$ such that $uwv\in \GGG$, hence $\tilde u \tilde w \tilde v := \Theta(uwv) \in \tilde\GGG$.
We have shown that $\tilde\GGG$ satisfies \ref{spec} and has the property that $\tilde\LLL \tilde w \tilde \LLL \cap \tilde\GGG \neq \emptyset$ for all $\tilde w\in \tilde \LLL$.  It remains only to show that if $\gcd\{k \mid \Per_k(X) \neq \emptyset\} = 1$, then $\tilde X$
is not a single non-trivial periodic orbit.  For this it suffices to observe that if $x\in \Per_k(X)$, then $\sigma^k(x) = x$ and hence $\tilde\sigma^k(\Theta(x)) = \Theta(\sigma^k(x)) = \Theta(x)$, hence $\Theta(x)$ is periodic with period a factor of $k$.  If $\tilde X$ is a single periodic orbit with least period $p$, then $p$ divides $k$ whenever $\Per_k(X)\neq \emptyset$, and hence $p=1$.
\iftoggle{arxiv}{}{\qed}\end{proof}

\subsection{Proof of claim in \S\ref{sec:shift-spaces}: $\Delta_a[\Phi{]}>0$ iff \eqref{eqn:SPR}}\label{sec:spr-pf}

The definition of strong positive recurrence in \cite{oS01}  involves positivity of a certain \defn{discriminant} $\Delta_a[\Phi]$.  More precisely, one defines the \defn{induced pressure function} $\gamma(p) := P_G(\overline{\Phi + p})$ for $p\in \RR$; this function can take finite or infinite values, and we will not need to use its definition, only its properties as proved in \cite{oS01}.  Writing $p_a^*[\Phi] = \sup \{p\mid \gamma(p)<\infty\}$, the discriminant is $\Delta_a[\Phi] = \sup\{\gamma(p) \mid p<p_a^*[\Phi]\}$.  By \cite[Proposition 3]{oS01}, the function $\gamma(p)$ is continuous and strictly increasing on $(-\infty,p_a^*[\Phi]]$.
We show that $\Delta_a[\Phi]>0$ iff \eqref{eqn:SPR} holds.  By \cite[(2),(4),(6)]{oS01}, we have
$\Delta_a[\Phi] = \gamma(p_a^*[\Phi])$ and
%\begin{align}
\begin{equation}
%\label{eqn:Delta}
%\Delta_a[\Phi] &= \gamma(p_a^*[\Phi]), \\
\label{eqn:p*}
p_a^*[\Phi] = -\ulim \frac 1n \log Z_n^*(\Phi,a), \qquad
%\label{eqn:PG}
P_G(\Phi) = \begin{cases} -p(\Phi) & \Delta_a[\Phi] \geq 0, \\
-p_a^*[\Phi] & \Delta_a[\Phi] < 0, \end{cases}
\end{equation}
where $p(\Phi)$ is the unique solution of $\gamma(p)=0$, which exists iff $\Delta_a[\Phi]\geq 0$.

Now if $\Delta_a[\Phi] < 0$, then \eqref{eqn:p*} %--\eqref{eqn:PG} 
shows that $P_G(\Phi) = \ulim \frac 1n Z_n^*(\Phi,a)$, so that \eqref{eqn:SPR} fails.  It remains to consider the case when $\Delta_a[\Phi] \geq 0$, so the two sides of \eqref{eqn:SPR} are given by $-p_a^*[\Phi]$ and $-p(\Phi)$, where $\gamma(p(\Phi))=0$ and $\gamma(p_a^*[\Phi]) = \Delta_a[\Phi]$.  Since $\gamma(p)$ is strictly increasing in $p$, it follows that $\Delta_a[\Phi] > 0$ iff $p(\Phi) < p_a^*[\Phi]$, which is equivalent to \eqref{eqn:SPR}.

\subsection{Proof of Lemma \ref{lem:UBP}}\label{sec:UBP}

Given $\psi_1\in \BBB_1$, define a family of continuous linear maps $\{K_i^{\psi_1} \colon \BBB_2 \to \RR\}_{i\in I}$ by $K_i^{\psi_1}(\psi_2) = F_n(\psi_1,\psi_2) \Delta(i)^{-1}$.  Then for every $i\in I$ we have $|K_i^{\psi_1}(\psi_2)| \leq K(\psi_1,\psi_2)$, so the family $\{K_i^{\psi_1}\}_{i\in I}$ is pointwise bounded.
%For every index $n$ and every $\psi_1\in \BBB_1$, define a continuous linear map $K_{n,\psi_1} \colon \BBB_2 \to \RR$ by $K_{n,\psi_1}(\psi_2) = F_n(\psi_1,\psi_2) \theta^{-\abs{n}}$.  Then 
%$K_n$ is a continuous linear operator (by continuity and bilinearity of $F_n$), and 
%for every $n$ we have $\abs{K_n^{\psi_1}(\psi_2)}  \leq K(\psi_1,\psi_2)$, so the family $K_n^{\psi_1}$ is pointwise bounded.  
By the uniform boundedness principle, $C_0(\psi_1) := \sup_{i\in I} \|K_i^{\psi_1}\| < \infty$.  Now we define a family of continuous linear maps $\{K'_{i,\psi_2} \colon \BBB_1\to \RR\}_{i\in I, \psi_2\in \BBB_2}$ by $K'_{i,\psi_2}(\psi_1) = F_n(\psi_1,\psi_2) \Delta(i)^{-1} \|\psi_2\|_2^{-1}$.
For each $i$ and $\psi_2$ we have $|K'_{i,\psi_2}(\psi_1)| = |K_i^{\psi_1}(\psi_2)| /\|\psi_2\|_2 \leq \|K_i^{\psi_1}\| \leq C_0(\psi_1)$, so the family $\{K'_{i,\psi_2}\}_{i\in I, \psi_2\in \BBB_2}$ is pointwise bounded.  Applying the uniform boundedness principle again gives $C := \sup_{i\in I, \psi_2 \in \BBB_2} \|K'_{i,\psi_2}\| < \infty$, and thus for all $\psi_1\in \BBB_1$, $\psi_2\in \BBB_2$, and $i\in I$, we have
\[
\abs{F_i(\psi_1,\psi_2)} = |K'_{i,\psi_2}(\psi_1)| \|\psi_2\|_2 \Delta(i)
\leq \|K'_{i,\psi_2}\| \|\psi_1\|_1 \|\psi_2\|_2 \Delta(i)
\leq C \|\psi_1\|_1 \|\psi_2\|_2 \Delta(i).
\]

\subsection{Proof of claims in Remark \ref{rmk:hE}}\label{sec:hE}

%Recall the example described in Remark \ref{rmk:hE}; fix $k \geq 4$ and let $X$ be the SFT on $A=\{1,\dots,k\}$ such that the allowed transitions are $a\to a+1 \pmod k$ and $a\to a+2 \pmod k$.
Suppose $\FFF \subset \LLL(X)$ satisfies \ref{free}.  Let $B = \{w_1 \mid w\in \FFF\}$ and $C=\{w_{\abs{w}} \mid w\in \FFF\}$.  Then by \ref{free} we have $c\to b$ for every $c\in C$ and $b\in B$.  Since each $a\in A$ has exactly two followers (two choices of $b$ such that $a\to b$) and no two choices of $a$ have the same set of two followers, 
one of $B,C$ must be a singleton, call it $\{a\}$.   Then every word in $\FFF$ either starts or ends with $a$. % By symmetry we can assume that $3\leq a\leq k-1$.
Let
$\DDD := \{w\in \LLL \mid w_j \neq a \text{ for all }1\leq j\leq \abs{w}\}$.
Then for any choice of $\Ep,\Es$ and any $w\in \DDD$, we have either $w\in \LLL \setminus \Ep \FFF \Es$, or $w = u^p v u^s$ for some $u^p \in \Es$, $v\in \FFF$, and $u^s\in \Es$; $v$ must be the empty word since $w$ does not contain the symbol $a$, so $\DDD \subset (\LLL \setminus \Ep \FFF \Es) \cup \Ep \Es$, and by Lemma \ref{lem:CD} we get
\[
h(\DDD) \leq \max \{h(\LLL \setminus \Ep \FFF \Es), h(\Ep), h(\Es)\} \leq h(\EEE),
\]
where $\EEE = I \cup \Ep \cup \Es \cup (\LLL \setminus \Ep \FFF \Es)$.
Each state has two followers, so there are always two choices for the next symbol, giving $\#\LLL_n = k 2^{n-1}$ and $h(X)=\log 2$.  We estimate $h(\DDD)$: given $u\in \{1,2\}^n$, define $\pi(u)\in \DDD_{n+1}$ by $\pi(u)_1 = a+1$ (working mod $k$ throughout) and
$\pi(u)_{i+1} = \pi(u)_i + u_i$ unless the right-hand side is $a$, in which case we set $\pi(u)_{i+1}$ to be whichver of $a\pm1$ is legal.
%either $\pi(u)_i=a-2$ and we put $\pi(u)_{i+1} = a-1$, or $\pi(u)_i = a-1$ and we put $\pi(u)_{i+1} = a+1$.
%\[
%\pi(u)_{i+1} = \begin{cases} a-1 & \text{if } \pi(u)_i = a-2, \\
%  a+1 & \text{if } \pi(u)_i = a-1, \\
%    \pi(u)_i + u_i  & \text{otherwise,}
%             \end{cases}
%\]
%where we work mod $k$.  
Given $u,v\in \{1,2\}^n$, we have $\pi(u)=\pi(v)$ if and only if $u_i = v_i$ for all $i$ such that $\pi(u)_i \notin \{a-2,a-1\}$, and since $\pi(u)_i \in \{a-2,a-1\}$ occurs at most twice in each $k/2$ consecutive values of $i$,  every $w\in \DDD_{n+1}$ has $\#\pi^{-1}(w) \leq 2^{2\frac n{k/2}}$.  Thus $\# \DDD_{n+1} \geq 2^n 2^{-4n/k}$, giving $h(\EEE) \geq h(\DDD) \geq (1-\frac 4k) \log 2$.  
%Note that 
%This shows that for every choice of $\Ep, \FFF, \Es$, we have $h(\EEE)> 0$,\footnote{When $k=4$ we have $1-\frac 4k=0$, but a more careful inspection shows that $h(\DDD)>0$.} and that the entropy gap between $h(\EEE)$ and $h(X)$ can be forced to be arbitrarily small by taking $k$ large.

\appendix

\section{List of conditions}\label{sec:conditions}

For ease of reference, we list here the various versions of \ref{spec}, \ref{gap}, \ref{stay-good} that appear throughout the paper.  First we list variants of \ref{spec}, which all give a specification property.
\begin{enumerate}[leftmargin=*, widest=I,
label = \textup{\textbf{[\Roman{*}{]}}}] %spec
\item 
There is $\tau\in \NN$  such that for all $v,w\in \GGG$, there is $u\in \LLL$ with $\abs{u} \leq \tau$ such that $v'uw'\in\GGG$ whenever $v'\in\GGG$ is a suffix of $v$ and $w'\in\GGG$ is a prefix of $w$,
%There is $\tau\in \NN$ such that for every $v,w\in \GGG$ there is $u\in \LLL_{\leq \tau}$ for which $v'uw'\in \GGG$ whenever $v',w'\in \GGG$ are such that $v\in \LLL v'$ and $w\in w'\LLL$;
\end{enumerate}
\begin{enumerate}[leftmargin=*, widest=I,
label = \textup{\textbf{[\Roman{*}$'${]}}}] %s-spec
\item
There is $\tau\in \NN$  such that for all $v,w\in \GGG$, there is $u\in \LLL$ with $\abs{u}= \tau$ such that $v'uw'\in\GGG$ whenever $v'\in\GGG$ is a suffix of $v$ and $w'\in\GGG$ is a prefix of $w$.
%There is $\tau\in \NN$ such that  for all $v,w\in \GGG$ there is $u\in \LLL_\tau$ for which $v'uw'\in \GGG$ whenever $v',w'\in \GGG$ are such that $v\in \LLL v'$ and $w\in w'\LLL$;
\end{enumerate}
\begin{enumerate}[leftmargin=*, widest=I,
label = \textup{\textbf{[\Roman{*}$_0${]}}}] %free
\item
Given any $v,w\in \FFF$ we have $vw\in \FFF$.
\end{enumerate}
\begin{enumerate}[leftmargin=*, widest=I,
label = \textup{\textbf{[\Roman{*}$^*${]}}}] %spec*
\item
For every $M\in \NN$ there is $\tau=\tau(M)$ such that for all $v,w\in \GGG(\CCC^\pm,M)$ there is $u\in \LLL$ with $\abs{u}\leq \tau$ such that $vuw\in \LLL$.
\end{enumerate}

Observe that \ref{s-spec} is stronger than \ref{spec}, and \ref{free} is stronger than both of them.  Theorem \ref{thm:get0spec} is devoted to going from \ref{spec} to \ref{free}, without control on $\gcd\{\abs{w}\mid w\in \FFF\}$; \ref{s-spec} gives control of this gcd.  Condition \ref{spec*} is used in Theorem \ref{thm:CGC}; $\GGG(\CCC^\pm,M)$ is defined in \eqref{eqn:GMC}.  

The variants of \ref{gap} control the pressure of prefix and suffix collections.
\begin{enumerate}[leftmargin=*, widest=II,
label = \textup{\textbf{[\Roman{*}{]}}}] %gap
\setcounter{enumi}{1}
\item
There are $\Cp,\Cs \subset \LLL$ such that $P(\Cp \cup \Cs \cup (\LLL \setminus \Cp \GGG \Cs),\ph) < P(\ph)$.
\end{enumerate}
\begin{enumerate}[leftmargin=*, widest=II,
label = \textup{\textbf{[\Roman{*}$'${]}}}] %gap-2
\setcounter{enumi}{1}
\item
 $P(I,\ph) < P(\ph)$, and there are $\Ep,\Es \subset \LLL$ with $P(\Ep \cup \Es \cup (\LLL \setminus \Ep\FFF\Es),\ph) < P(\ph)$.
\end{enumerate}
Observe that \ref{gap-2} is stronger than \ref{gap}, since it imposes a condition on the set of generators $I=\FFF\setminus \FFF\FFF$ in addition to the collection $\FFF = I^*$.  Note that \ref{gap-2} only makes sense for collections satisfying \ref{free} so that we can talk about a `set of generators'.

The variants of \ref{stay-good} impose conditions on how $\GGG$ behaves under intersections and unions.  We start with two equivalent formulations of \ref{stay-good}, then list other related conditions.
\begin{enumerate}[leftmargin=*, widest=III,
label = \textup{\textbf{[\Roman{*}{]}}}] %stay-good
\setcounter{enumi}{2}
\item
There is $L\in \NN$ such that if $u,v,w\in \LLL$ have $\abs{v}\geq L$, $uvw\in \LLL$, $uv,vw\in \GGG$, then $v, uvw\in \GGG$.%\footnote{The condition on $uvw$ can be relaxed slightly; see Theorem \ref{thm:get0spec}.}
\end{enumerate}
\begin{enumerate}[leftmargin=*, widest=III,
label = \textup{\textbf{[\Roman{*}{]}}}]
\setcounter{enumi}{2}
\item There is $L\in \NN$ such that if $x\in X$ and $i\leq j \leq k\leq \ell\in \ZZ$ are such that $k-j \geq L$ and $x_{[i,k)},x_{[j,\ell)}\in \GGG$, then $x_{[j,k)}, x_{[i,\ell)} \in \GGG$.
\end{enumerate}
\begin{enumerate}[leftmargin=*, widest=III,
label = \textup{\textbf{[\Roman{*}$^*${]}}}] %overlaps
\setcounter{enumi}{2}
\item
If $x\in X$ and $i\leq j\leq k\leq \ell$ are such that $x_{[i,k)},x_{[j,\ell)}\in \FFF$, and there are $a<j$ and $b>k$ such that $x_{[a,j)},x_{[k,b)}\in \FFF$, then $x_{[j,k)}\in \FFF$.
\end{enumerate}
\begin{enumerate}[leftmargin=*, widest=a,
label = \textup{\textbf{[III$_\mathrm{\alph{*}}${]}}}] %inter
\item
There is $L$ 
such that if $uv,vw\in \GGG$, $\abs{v}\geq L$, and $uvw\in \LLL$, then $v\in \GGG$.
\end{enumerate}
\begin{enumerate}[leftmargin=*, widest=b,
label = \textup{\textbf{[III$_\mathrm{\alph{*}}${]}}}] %union
\setcounter{enumi}{1}
\item
There is $L$ such that if $uv,vw\in \GGG$, $\abs{v}\geq L$, and $xuvw\in \GGG$ for some $x\in \LLL$, then $uvw\in \GGG$.
\end{enumerate}
Condition \ref{overlaps} neither implies nor follows from \ref{stay-good}.  Conditions \ref{inter} and \ref{union} both follow from \ref{stay-good}, but do not imply it, because of the freedom in choosing $x$ in \ref{union}.  As seen in Theorem \ref{thm:get0spec}, these conditions are actually sufficient to prove the main result, which is important for the applications in \cite{CP}.

\iftoggle{arxiv}{}{
\begin{acknowledgement}
I am grateful to the anonymous referees for many comments that improved the exposition and for spotting errors in earlier versions of the result on factors and of Lemma \ref{lem:reset}; the latter, which was also pointed out to me by Qu Congcong, necessitated a change in the formulation of condition \ref{spec} from previous versions.
I am also grateful to Omri Sarig for clarifying aspects of strong positive recurrence as they appear in \S\ref{sec:shift-spaces}, and to Dominik Kwietniak for introducing me to cocyclic subshifts and \cite{jK00,jK04}.
This work was partially supported by NSF grants DMS-1362838 and DMS-1554794.
\end{acknowledgement}
}

%\backmatter
\ifarxiv\bibliographystyle{amsalpha}	
\else\bibliographystyle{alpha}
\fi
\bibliography{spec-towers}

\providecommand{\bysame}{\leavevmode\hbox to3em{\hrulefill}\thinspace}
\providecommand{\MR}{\relax\ifhmode\unskip\space\fi MR }
\providecommand{\MRhref}[2]{%
  \href{http://www.ams.org/mathscinet-getitem?mr=#1}{#2}
}
\providecommand{\href}[2]{#2}
\begin{thebibliography}{BCLS15}

\bibitem[AHS16]{AHS}
Artur Avila, Pascal Hubert, and Alexandra Skripchenko, \emph{Diffusion for
  chaotic plane sections of 3-periodic surfaces}, Invent. Math. \textbf{206}
  (2016), no.~1, 109--146. \MR{3556526}

\bibitem[BCFT]{BCFT}
Keith Burns, Vaughn Climenhaga, Todd Fisher, and Daniel~J. Thompson,
  \emph{Unique equilibrium states for geodesic flows in non-positive
  curvature}, Geom.\ Funct.\ Anal., 44 pages, to appear.

\bibitem[BCLS15]{BCLS}
Martin Bridgeman, Richard Canary, Fran{\c{c}}ois Labourie, and Andres
  Sambarino, \emph{The pressure metric for {A}nosov representations}, Geom.
  Funct. Anal. \textbf{25} (2015), no.~4, 1089--1179. \MR{3385630}

\bibitem[Ber88]{aB88}
Anne Bertrand, \emph{Specification, synchronisation, average length}, Coding
  theory and applications ({C}achan, 1986), Lecture Notes in Comput. Sci., vol.
  311, Springer, Berlin, 1988, pp.~86--95. \MR{960710 (89i:94022)}

\bibitem[BG11]{BG11}
A.~I. Bufetov and B.~M. Gurevich, \emph{Existence and uniqueness of a measure
  with maximal entropy for the {T}eichm\"uller flow on the moduli space of
  abelian differentials}, Mat. Sb. \textbf{202} (2011), no.~7, 3--42.
  \MR{2857792 (2012h:37068)}

\bibitem[BH86]{BH86}
F.~Blanchard and G.~Hansel, \emph{Syst\`emes cod\'es}, Theoret. Comput. Sci.
  \textbf{44} (1986), no.~1, 17--49. \MR{858689 (88m:68029)}

\bibitem[BL13]{BL13}
Henk Bruin and Renaud Leplaideur, \emph{Renormalization, thermodynamic
  formalism and quasi-crystals in subshifts}, Comm. Math. Phys. \textbf{321}
  (2013), no.~1, 209--247. \MR{3089670}

\bibitem[Bow72]{rB72}
Rufus Bowen, \emph{The equidistribution of closed geodesics}, Amer. J. Math.
  \textbf{94} (1972), 413--423. \MR{0315742 (47 \#4291)}

\bibitem[Bow74]{rB74}
\bysame, \emph{Some systems with unique equilibrium states}, Math. Systems
  Theory \textbf{8} (1974), no.~3, 193--202. \MR{0399413 (53 \#3257)}

\bibitem[Bow75]{rB75}
\bysame, \emph{Equilibrium states and the ergodic theory of {A}nosov
  diffeomorphisms}, Lecture Notes in Mathematics, Vol. 470, Springer-Verlag,
  Berlin-New York, 1975. \MR{0442989 (56 \#1364)}

\bibitem[Bow78]{rB78}
\bysame, \emph{On {A}xiom {A} diffeomorphisms}, American Mathematical Society,
  Providence, R.I., 1978, Regional Conference Series in Mathematics, No. 35.
  \MR{0482842}

\bibitem[Boy08]{mB08}
Mike Boyle, \emph{Open problems in symbolic dynamics}, Geometric and
  probabilistic structures in dynamics, Contemp. Math., vol. 469, Amer. Math.
  Soc., Providence, RI, 2008, pp.~69--118. \MR{2478466 (2010h:37023)}

\bibitem[BS03]{BS03}
J{\'e}r{\^o}me Buzzi and Omri Sarig, \emph{Uniqueness of equilibrium measures
  for countable {M}arkov shifts and multidimensional piecewise expanding maps},
  Ergodic Theory Dynam. Systems \textbf{23} (2003), no.~5, 1383--1400.
  \MR{2018604 (2004k:37053)}

\bibitem[BSS02]{BSS02}
Luis Barreira, Beno{\^{\i}}t Saussol, and J{\"o}rg Schmeling,
  \emph{Higher-dimensional multifractal analysis}, J. Math. Pures Appl. (9)
  \textbf{81} (2002), no.~1, 67--91. \MR{1994883 (2004g:37038)}

\bibitem[BT09]{BT09}
Henk Bruin and Mike Todd, \emph{Equilibrium states for interval maps: the
  potential {$-t\log\vert Df\vert $}}, Ann. Sci. \'Ec. Norm. Sup\'er. (4)
  \textbf{42} (2009), no.~4, 559--600. \MR{2568876 (2011c:37079)}

\bibitem[Buz97]{jB97}
J{\'e}r{\^o}me Buzzi, \emph{Intrinsic ergodicity of affine maps in
  {$[0,1]^d$}}, Monatsh. Math. \textbf{124} (1997), no.~2, 97--118. \MR{1462857
  (98g:58098)}

\bibitem[Buz04]{jB04}
\bysame, \emph{Entropy of equilibrium measures of continuous piecewise
  monotonic maps}, Stoch. Dyn. \textbf{4} (2004), no.~1, 84--94. \MR{2069369
  (2005h:37071)}

\bibitem[Buz05]{jB05}
\bysame, \emph{Subshifts of quasi-finite type}, Invent. Math. \textbf{159}
  (2005), no.~2, 369--406. \MR{2116278 (2005i:37013)}

\bibitem[CC]{CC}
Vaughn Climenhaga and Van Cyr, \emph{Positive entropy equilibrium states}, 16
  pages, \href{http://arxiv.org/abs/1708.02272}{\color{red}arXiv:1708.02272}.

\bibitem[CFT]{CFT2}
Vaughn Climenhaga, Todd Fisher, and Daniel~J. Thompson, \emph{Equilibrium
  states for {M}a\~n\'e diffeomorphisms}, Ergodic Theory Dynam. Systems, to
  appear, 25 pages,
  \href{http://arxiv.org/abs/1703.05722}{\color{red}arXiv:1703.05722}.

\bibitem[CFT18]{CFT}
Vaughn Climenhaga, Todd Fisher, and Daniel~J. Thompson, \emph{Unique
  equilibrium states for {B}onatti--{V}iana diffeomorphisms}, Nonlinearity
  \textbf{31} (2018), no.~6, 2532--2570.

\bibitem[Con]{sC}
Scott Conrad, \emph{{A coded shift with a H\"older potential that is not
  hyperbolic}}, preprint.

\bibitem[CP]{CP}
Vaughn Climenhaga and Ronnie Pavlov, \emph{One-sided almost specification and
  intrinsic ergodicity}, Ergodic Theory and Dynamical Systems, 25 pages, to
  appear.

\bibitem[CS09]{CS09}
Van Cyr and Omri Sarig, \emph{Spectral gap and transience for {R}uelle
  operators on countable {M}arkov shifts}, Comm. Math. Phys. \textbf{292}
  (2009), no.~3, 637--666. \MR{2551790 (2011f:37041)}

\bibitem[CT12]{CT1}
Vaughn Climenhaga and Daniel~J. Thompson, \emph{Intrinsic ergodicity beyond
  specification: {$\beta$}-shifts, {$S$}-gap shifts, and their factors}, Israel
  J. Math. \textbf{192} (2012), no.~2, 785--817. \MR{3009742}

\bibitem[CT13]{CT2}
\bysame, \emph{Equilibrium states beyond specification and the {B}owen
  property}, J. Lond. Math. Soc. (2) \textbf{87} (2013), no.~2, 401--427.
  \MR{3046278}

\bibitem[CT14]{CT3}
\bysame, \emph{Intrinsic ergodicity via obstruction entropies}, Ergodic Theory
  Dynam. Systems \textbf{34} (2014), no.~6, 1816--1831. \MR{3272773}

\bibitem[CT16]{CT4}
\bysame, \emph{Unique equilibrium states for flows and homeomorphisms with
  non-uniform structure}, Adv. Math. \textbf{303} (2016), 745--799.
  \MR{3552538}

\bibitem[CTY17]{CTY}
Vaughn Climenhaga, Daniel~J. Thompson, and Kenichiro Yamamoto, \emph{Large
  deviations for systems with non-uniform structure}, Trans. Amer. Math. Soc.
  \textbf{369} (2017), no.~6, 4167--4192. \MR{3624405}

\bibitem[Dao13]{yD13}
Yair Daon, \emph{Bernoullicity of equilibrium measures on countable {M}arkov
  shifts}, Discrete Contin. Dyn. Syst. \textbf{33} (2013), no.~9, 4003--4015.
  \MR{3038050}

\bibitem[FF92]{FF92}
Doris Fiebig and Ulf-Rainer Fiebig, \emph{Covers for coded systems}, Symbolic
  dynamics and its applications ({N}ew {H}aven, {CT}, 1991), Contemp. Math.,
  vol. 135, Amer. Math. Soc., Providence, RI, 1992, pp.~139--179. \MR{1185086
  (93m:54068)}

\bibitem[Gou06]{sG06}
S\'ebastien Gou\"ezel, \emph{Regularity of coboundaries for nonuniformly
  expanding {M}arkov maps}, Proc. Amer. Math. Soc. \textbf{134} (2006), no.~2,
  391--401. \MR{2176007}

\bibitem[Hof78]{fH78}
Franz Hofbauer, \emph{{$\beta $}-shifts have unique maximal measure}, Monatsh.
  Math. \textbf{85} (1978), no.~3, 189--198. \MR{0492180 (58 \#11326)}

\bibitem[Hof79]{fH79}
\bysame, \emph{On intrinsic ergodicity of piecewise monotonic transformations
  with positive entropy}, Israel J. Math. \textbf{34} (1979), no.~3, 213--237
  (1980). \MR{570882 (82c:28039a)}

\bibitem[Hof81]{fH81}
\bysame, \emph{On intrinsic ergodicity of piecewise monotonic transformations
  with positive entropy. {II}}, Israel J. Math. \textbf{38} (1981), no.~1-2,
  107--115. \MR{599481 (82c:28039b)}

\bibitem[IRRL12]{IRRL12}
Irene Inoquio-Renteria and Juan Rivera-Letelier, \emph{A characterization of
  hyperbolic potentials of rational maps}, Bull. Braz. Math. Soc. (N.S.)
  \textbf{43} (2012), no.~1, 99--127. \MR{2909925}

\bibitem[IT10]{IT10}
Godofredo Iommi and Mike Todd, \emph{Natural equilibrium states for multimodal
  maps}, Comm. Math. Phys. \textbf{300} (2010), no.~1, 65--94. \MR{2725183
  (2011m:37048)}

\bibitem[Kat82]{aK82}
A.~Katok, \emph{Entropy and closed geodesics}, Ergodic Theory Dynam. Systems
  \textbf{2} (1982), no.~3-4, 339--365 (1983). \MR{721728 (85b:53047)}

\bibitem[Kel89]{gK89}
Gerhard Keller, \emph{Lifting measures to {M}arkov extensions}, Monatsh. Math.
  \textbf{108} (1989), no.~2-3, 183--200. \MR{1026617 (91b:28011)}

\bibitem[Kif90]{yK90}
Yuri Kifer, \emph{Large deviations in dynamical systems and stochastic
  processes}, Trans. Amer. Math. Soc. \textbf{321} (1990), no.~2, 505--524.
  \MR{1025756 (91e:60091)}

\bibitem[Kit98]{bK98}
Bruce~P. Kitchens, \emph{Symbolic dynamics}, Universitext, Springer-Verlag,
  Berlin, 1998, One-sided, two-sided and countable state Markov shifts.
  \MR{1484730}

\bibitem[Kni98]{gK98}
Gerhard Knieper, \emph{The uniqueness of the measure of maximal entropy for
  geodesic flows on rank {$1$} manifolds}, Ann. of Math. (2) \textbf{148}
  (1998), no.~1, 291--314. \MR{1652924 (2000b:37016)}

\bibitem[Kwa00]{jK00}
Jaroslaw Kwapisz, \emph{Cocyclic subshifts}, Math. Z. \textbf{234} (2000),
  no.~2, 255--290. \MR{1765882 (2001j:37025)}

\bibitem[Kwa04]{jK04}
\bysame, \emph{Transfer operator, topological entropy and maximal measure for
  cocyclic subshifts}, Ergodic Theory Dynam. Systems \textbf{24} (2004), no.~4,
  1173--1197. \MR{2085908 (2005d:37018)}

\bibitem[LM95]{LM95}
Douglas Lind and Brian Marcus, \emph{An introduction to symbolic dynamics and
  coding}, Cambridge University Press, Cambridge, 1995. \MR{1369092
  (97a:58050)}

\bibitem[LRL14]{LRL14}
Huaibin Li and Juan Rivera-Letelier, \emph{Equilibrium states of weakly
  hyperbolic one-dimensional maps for {H}\"older potentials}, Comm. Math. Phys.
  \textbf{328} (2014), no.~1, 397--419. \MR{3196990}

\bibitem[McM08]{cM08}
Curtis~T. McMullen, \emph{Thermodynamics, dimension and the {W}eil-{P}etersson
  metric}, Invent. Math. \textbf{173} (2008), no.~2, 365--425. \MR{2415311
  (2010c:37101)}

\bibitem[Orn70]{dO70}
Donald Ornstein, \emph{Factors of {B}ernoulli shifts are {B}ernoulli shifts},
  Advances in Math. \textbf{5} (1970), 349--364 (1970). \MR{0274717}

\bibitem[Pet89]{kP89}
Karl Petersen, \emph{Ergodic theory}, Cambridge Studies in Advanced
  Mathematics, vol.~2, Cambridge University Press, Cambridge, 1989, Corrected
  reprint of the 1983 original. \MR{1073173 (92c:28010)}

\bibitem[PP90]{PP90}
William Parry and Mark Pollicott, \emph{Zeta functions and the periodic orbit
  structure of hyperbolic dynamics}, Ast\'erisque (1990), no.~187-188, 268.
  \MR{1085356 (92f:58141)}

\bibitem[PS08]{PS08}
Yakov Pesin and Samuel Senti, \emph{Equilibrium measures for maps with inducing
  schemes}, J. Mod. Dyn. \textbf{2} (2008), no.~3, 397--430. \MR{2417478
  (2009k:37072)}

\bibitem[PSZ14]{PSZ}
Yakov Pesin, Samuel Senti, and Ke~Zhang, \emph{{Thermodynamics of towers of
  hyperbolic type}}, to appear in Trans. Amer. Math. Soc., arXiv:1403.2989,
  2014.

\bibitem[Rue76]{dR76}
David Ruelle, \emph{A measure associated with axiom-{A} attractors}, Amer. J.
  Math. \textbf{98} (1976), no.~3, 619--654. \MR{0415683 (54 \#3763)}

\bibitem[Rue92]{dR92}
\bysame, \emph{Thermodynamic formalism for maps satisfying positive
  expansiveness and specification}, Nonlinearity \textbf{5} (1992), no.~6,
  1223--1236. \MR{1192516 (94a:58115)}

\bibitem[Sar99]{oS99}
Omri~M. Sarig, \emph{Thermodynamic formalism for countable {M}arkov shifts},
  Ergodic Theory Dynam. Systems \textbf{19} (1999), no.~6, 1565--1593.
  \MR{1738951 (2000m:37009)}

\bibitem[Sar01]{oS01}
\bysame, \emph{Phase transitions for countable {M}arkov shifts}, Comm. Math.
  Phys. \textbf{217} (2001), no.~3, 555--577. \MR{1822107 (2002b:37040)}

\bibitem[Sar13]{oS13}
\bysame, \emph{Symbolic dynamics for surface diffeomorphisms with positive
  entropy}, J. Amer. Math. Soc. \textbf{26} (2013), no.~2, 341--426.
  \MR{3011417}

\bibitem[Sar15]{oS15}
\bysame, \emph{Thermodynamic formalism for countable {M}arkov shifts},
  Hyperbolic dynamics, fluctuations and large deviations, Proc. Sympos. Pure
  Math., vol.~89, Amer. Math. Soc., Providence, RI, 2015, pp.~81--117.
  \MR{3309096}

\bibitem[Sch97]{jS97}
J{\"o}rg Schmeling, \emph{Symbolic dynamics for {$\beta$}-shifts and
  self-normal numbers}, Ergodic Theory Dynam. Systems \textbf{17} (1997),
  no.~3, 675--694. \MR{1452189 (98c:11080)}

\bibitem[Sin72]{yS72}
Ja.~G. Sina\u\i, \emph{Gibbs measures in ergodic theory}, Uspehi Mat. Nauk
  \textbf{27} (1972), no.~4(166), 21--64. \MR{0399421}

\bibitem[Tho06]{kT06}
Klaus Thomsen, \emph{On the ergodic theory of synchronized systems}, Ergodic
  Theory Dynam. Systems \textbf{26} (2006), no.~4, 1235--1256. \MR{2247640
  (2009a:37011)}

\bibitem[Tho11]{dT11}
Damien Thomine, \emph{A spectral gap for transfer operators of piecewise
  expanding maps}, Discrete Contin. Dyn. Syst. \textbf{30} (2011), no.~3,
  917--944. \MR{2784627}

\bibitem[Wal78]{pW78}
Peter Walters, \emph{Equilibrium states for {$\beta $}-transformations and
  related transformations}, Math. Z. \textbf{159} (1978), no.~1, 65--88.
  \MR{0466492 (57 \#6370)}

\bibitem[Wal82]{pW82}
\bysame, \emph{An introduction to ergodic theory}, Graduate Texts in
  Mathematics, vol.~79, Springer-Verlag, New York-Berlin, 1982. \MR{648108
  (84e:28017)}

\bibitem[You98]{lY98}
Lai-Sang Young, \emph{Statistical properties of dynamical systems with some
  hyperbolicity}, Ann. of Math. (2) \textbf{147} (1998), no.~3, 585--650.
  \MR{1637655 (99h:58140)}

\bibitem[You99]{lY99}
\bysame, \emph{Recurrence times and rates of mixing}, Israel J. Math.
  \textbf{110} (1999), 153--188. \MR{1750438 (2001j:37062)}

\bibitem[Zwe05]{rZ04}
Roland Zweim\"uller, \emph{Invariant measures for general(ized) induced
  transformations}, Proc. Amer. Math. Soc. \textbf{133} (2005), no.~8,
  2283--2295. \MR{2138871}

\end{thebibliography}

\end{document}